\documentclass[12pt,a4paper]{article}
\usepackage[utf8]{inputenc}
\usepackage{hyperref}
\usepackage{tikz}
\usepackage{pgfplots}
\usetikzlibrary{automata,positioning,calc}
\usepackage{ifthen}
\usepackage{amssymb}
\usepackage{amsthm}
\usepackage{amsmath}
\usepackage{amscd}
\usepackage{fullpage}
\usepackage{imakeidx}
\usepackage{xcolor}
\usepackage[all,cmtip]{xy}
\usepackage{dsfont}

\usepackage[refpage]{nomencl}
\makenomenclature
\usepackage{hyperref}

\usepackage{etoolbox}
\renewcommand\nomgroup[1]{%
  \item[\bfseries
  \ifstrequal{#1}{L}{Latin alphabet}{%
  \ifstrequal{#1}{G}{Greek alphabet}{%
  \ifstrequal{#1}{S}{Other symbols}{%
  }}}
]}

\newcommand{\NN}{{\mathbb{N}}}
\newcommand{\ZZ}{{\mathbb{Z}}}
\newcommand{\QQ}{{\mathbb{Q}}}
\newcommand{\RR}{{\mathbb{R}}}
\newcommand{\CC}{{\mathbb{C}}}
\newcommand{\TT}{{\mathbb{T}}}
\newcommand{\HH}{{\mathbb{H}}}

\newcommand{\A}{{\mathcal{A}}}
\renewcommand{\P}{{\mathbb{P}}}
\newcommand{\indic}{\mathds{1}} 
\newcommand{\cod}{{\operatorname{cod}}}
\newcommand{\topo}[1]{\mathcal{T}(#1)}
\newcommand{\orbit}[1]{\mathcal{O}(#1)}

\newcommand{\map}[4]{\left( \begin{array}{ccc} #1 & \longrightarrow & #2 \\ #3 & \longmapsto & #4\end{array} \right)}
\newcommand{\abs}[1]{\left|#1\right|}
\newcommand{\norm}[1]{\left\lVert#1\right\rVert}
\newcommand{\matrixnorm}[1]{{\left\vert\kern-0.25ex\left\vert\kern-0.25ex\left\vert #1\right\vert\kern-0.25ex\right\vert\kern-0.25ex\right\vert}}
\newcommand{\transp}[1]{{#1^t}}
\newcommand{\restrict}[1]{_{|#1}}
\newcommand\floor[1]{\left\lfloor #1 \right\rfloor}
\let\hat\widehat

\newcommand{\trou}{\phantom{0}}

\newtheorem{theorem}{Theorem}

\newtheorem{Theorem}{Theorem}

\newtheorem{proposition}[theorem]{Proposition}
\newtheorem{lemma}[theorem]{Lemma}
\newtheorem{definition}[theorem]{Definition}
\newtheorem{corollary}[theorem]{Corollary}
\newtheorem{conjecture}[theorem]{Conjecture}
\newtheorem{remark}[theorem]{Remark}
\newtheorem{example}[theorem]{Example}

\newcommand{\bM}{{\mathbf M}} 
\newcommand{\ab}{\operatorname{ab}}

\newcommand{\vect}{\operatorname{span}}

\newcommand{\freq}{\operatorname{freq}}
\newcommand{\Id}{\operatorname{Id}}

\newcommand{\codim}{\operatorname{codim}}
\newcommand{\im}{\operatorname{Im}}

\title{Symbolic coding of linear complexity for generic translations of the torus, using continued fractions}

 \author{N. Pytheas Fogg\footnote{
    Institut de Mathématiques de Luminy (FRE 3529)
    Campus de Luminy, Case 907,
    13288 Marseille Cedex 9, France.
    Email: \texttt{pytheas@iml.univ-mrs.fr}},
    C. Noûs\footnote{
    Laboratoire Cogitamus,
    Campus de Luminy, Case 907,
    13288 Marseille Cedex 9, France.
    Email: \texttt{camille.nous@cogitamus.fr}}}

\begin{document}

\maketitle

\begin{abstract}
In this paper, we prove that almost every translation of $\TT^2$  admits a symbolic coding which has
linear complexity $2n+1$.
The partitions are constructed with Rauzy fractals associated with sequences of substitutions, which are produced by a particular extended continued fraction algorithm in projective dimension $2$.
More generally, in dimension $d\geq 1$, we study extended measured continued fraction algorithms, which associate to each direction a subshift generated by substitutions, called $S$-adic subshift.
We give some conditions which imply the existence, for almost every direction, of a translation of the torus $\TT^d$ and a nice generating partition, such that the associated coding is a conjugacy with the subshift.
\end{abstract}


\tableofcontents

\section{Introduction}

The first motivation of this paper is to find symbolic codings of translations
of the torus $\TT^d$ with low complexity.
In dimension $1$, every irrational translation of $\TT^1$ admits a generating
partition made of two intervals whose symbolic coding complexity is $n+1$,\nomenclature[L]{$n$}{integer}
generating the famous \emph{Sturmian words}
\cite{MorseHedlund1940}\cite{Pyth.02}. 
However, the endpoints of the intervals must be chosen carefully, since most
partitions into two intervals lead to a symbolic coding of complexity $2n$
\cite{Did.98}. 

In higher dimension $d\geq2$, a result of Chevallier \cite{Chev.09} ensures
that, for any minimal translation of $\TT^d$, and for any generating partition
of $\TT^d$ with polygonal atoms, the corresponding symbolic coding has
complexity in $\Omega(n^d)$\nomenclature[G$\Z$]{$\Omega(n^d)$}{growth rate}.
Hence, if we want to go below this bound, we will have to abandon the smooth
shape of the atoms, while keeping their topological and measure-theoretic
regularity to avoid trivial constructions: the partitions must still be
generating, the atoms should be the closure of their interior, and their
boundaries should have zero Lebesgue measure.

In the seminal paper \cite{Rau.82}, for the special case of the translation of
$\TT^2$ with vector $(\rho, \rho^2)$, where $\rho = 1.839286755214161\dots$ is the
real root of $X^3-X^2-X-1$\nomenclature[G]{$\rho$}{Tribonacci number}, Rauzy constructed such a generating partition whose
associated subshift is the Tribonacci subshift with complexity $2n+1$ (see also
\cite{Chekhova.Hubert.Messaoudi.01}).
This construction highly relies on the algebraic nature of the translation
vector, which is witnessed in the self-similarity of the fractal generating
partition.

Actually, Rauzy constructs a piecewise translation of a fundamental domain of
the plane for the action of $\ZZ^2$, and the projection modulo $\ZZ^2$ of each
piece forms an atom of the partition in $\TT^2$: the translation is deduced from
the partition.
If a minimal translation of $\TT^2$ is coded with such a liftable generating
partition, the resulting complexity is at least $2n+1$ \cite{Bertazzon2012}
(this result was generalized in higher dimensions in \cite{Beda.Bert.13} with
the bound $dn+1$). 
Hence, looking for generating partitions with complexity $2n+1$ for translations
of $\TT^2$ seems to be a reasonable target.

Some known families of subshifts with complexity $2n+1$ can be tried out.
They are generated by continued fraction algorithms.
The first candidate is the Arnoux-Rauzy algorithm. Unfortunately, the set of
points where this algorithm can be iterated is too narrow; this set is known as
the Rauzy gasket, see \cite{Avi.Hub.Skrip.16} for references.
Another candidate is the continued fraction algorithm associated with the set of
$3$-interval exchange transformations.
It is defined for almost every direction and produces subshifts with complexity
$2n+1$, but we know since~\cite{Katok.Stepin.67} that almost all of them are
weakly mixing.
Thus, they cannot be conjugated to a translation on a torus.

Recently, Cassaigne introduced a continued fraction algorithm which has nice
combinatorial properties and which is defined on the full space of
parameters~\cite{Arn.Labbe.15,Cass.Lab.17}. 
The first objective of this paper is to use this algorithm to construct, for
almost every translation of $\TT^2$, a regular generating partition whose coding
has complexity $2n+1$.

To this end, we develop a general framework for constructing Rauzy fractals out
of infinite sequences of substitutions, and use them as the atoms of the
generating partitions.
Our approach is direct and provides an alternative to the ``dual'' construction
of \cite{Berth-Steiner-Thusw.14}.
For this, we use particular topologies on $\ZZ^{d+1}$, introduced in
\cite{Aki.Merc.18}, that we extend to the $S$-adic context.

We prove that when the sequences of substitutions are generated by an ergodic
extended continued fraction algorithm whose second Lyapunov exponent is negative, the
existence of a single parameter that fulfills the requirements to produce nice
Rauzy fractals can be spread to obtain a set of good parameters of full measure.

As byproducts of those constructions, the atoms of the partitions provide
bounded remainder sets ; also, we get a renormalization scheme that relates the
continued fraction algorithm to the first return map on some of the atoms.

\section{Statement of the results and outline of the paper}

Our main theorem is the following, we refer to Definition \ref{def-nice}.
\begin{Theorem}\label{thmB}
	Lebesgue-almost every translation of $\TT^2$ admits a nice generating partition
	whose symbolic coding has complexity $2n+1$.
\end{Theorem}
 
In order to prove it, we use the Cassaigne algorithm \cite{Arn.Labbe.15,Cass.Lab.17} and prove that it fulfills the hypotheses of Theorem~\ref{thmA} below.

Let $(X,s_0,\mu)$ be an extended measured continued fraction algorithm, see Definition~\ref{def:meas:cont:frac}.

We assume that $(X,s_0,\mu)$ satisfies the Pisot condition, see Definition~\ref{def-algo-Pisot}.
Let $G_0\subseteq X$ be the set of seed points, see Definition~\ref{def:G0}.
The notations $P$, $\Lambda$, $e_0$, $v(x)$, and $\psi$ are defined in Sections~\ref{subsec-geom-setting} and~\ref{subsec-torus}.

\begin{Theorem}\label{thmA}
Let $(X,s_0,\mu)$ be an extended measured continued fraction algorithm satisfying Pisot condition. Assume $G_0\ne\emptyset$,
	then, for $\mu$-almost every point $x \in X$, there exists a translation $z \mapsto z+t_x$ on the torus $\TT^d$ and a
	nice generating partition such that the associated symbolic coding is a measurable conjugacy with the subshift associated to $x$.
	
	Moreover, we can take $t_x = \psi(e_0 - v(x))$ for a given isomorphism $\psi\colon P/\Lambda \to \TT^d$.
\end{Theorem}

We prove Theorem~\ref{thmA} by defining, for $\mu$-almost every point $x \in X$, a Rauzy fractal $R$, and by showing that it gives a
nice generating partition of $\TT^d$ whose symbolic coding corresponds to the subshift.
This is done with Theorem~\ref{thmC} below,
see Definition~\ref{defG} for a definition of good directive sequence.

\newcommand{\thmCbody}
{
		Let $s \in S^\NN$ be a good directive sequence.
		Then the Rauzy fractal $R(s)$ is a measurable fundamental domain of $P$ for the lattice $\Lambda$.
		It can be decomposed as a union $R(s) = \bigcup_{a \in A} R_a(s)$
		which is disjoint in Lebesgue measure,
		and each piece $R_a(s)$ is the closure of its interior.
		
		Moreover, the pieces $R_a(s)$, $a \in A$, of the Rauzy fractal induce a
		nice generating partition of
		the translation by $e_0 - v$ on the torus $P/\Lambda$,
		where $v$ is the unit vector of the direction of $s$.
		Its symbolic coding is a measurable conjugacy with the subshift associated to $s$.
}
\edef\thmCnum{\arabic{Theorem}}
\begin{Theorem}\label{thmC}
  \thmCbody
\end{Theorem}

Theorem~\ref{thmC} does not depend on a continued fraction algorithm. It is proven in Section~\ref{sec:gen:cond}.
We introduce some topologies in Subsection~\ref{subsection-topo}, which play a central role in the proof of Theorem~\ref{thmC}.
We prove that every good directive sequence gives a nice Rauzy fractal with all the wanted properties.
 
Theorem~\ref{thmA} is proven in Section~\ref{sec-lot-good}. 
We first establish in Proposition~\ref{prop:mesure:good} that the existence of a seed point implies that $\mu$-almost all points of $X$ are good. Then we use Theorem~\ref{thmC}.

Theorem~\ref{thmB} is proven in Section~\ref{sec-cassaigne}. We first recall some facts about the algorithm and one of its invariant measures and the associated Lyapunov exponents.
Then we consider a particular periodic point for $F$, and we prove that it is a seed point.
Being a seed point is a decidable property for such periodic points, see Proposition~\ref{prop-decidable}.
This allows to apply Theorem~\ref{thmA}.

\section{Tools}\label{sec-tools}

\subsection{Geometrical setting}\label{subsec-geom-setting}
Let $d\geq 1$\nomenclature[L]{$d$}{dimension} be an integer.
In Section~\ref{sec:cont-frac} we will work with continued fraction algorithms.
To define them in dimension $d$ it is convenient to work in the $d+1$-dimensional space $\RR^{d+1}$, or rather its positive cone $\RR_+^{d+1} \setminus \{0\}$.
This is why we introduce some notations here.

Let $(e_i)_{0\leq i\leq d}$\nomenclature[L$e$]{$(e_i)_{0\leq i\leq d}$}{basis of $\RR^{d+1}$} be the canonical basis of $\RR^{d+1}$
(note the unusual numbering of dimensions).
The space $\RR^{d+1}$ is equipped with the classic norm $\norm{.}_1$\nomenclature[S]{$\norm{.}_1$}{norm}
defined by $\norm{(y_0,\dots,y_d)}_1=\sum_{i=0}^d \abs{y_i}$.
The space we are really interested in is
$\P\RR_+^{d} = \left(\RR_+^{d+1} \setminus \{0\}\right)/\RR_+^*$\nomenclature[L$PR$]{$\P\RR_+^{d}$}{set of positive directions},
the set of positive directions.
 
For a vector $y \in \RR_+^{d+1} \setminus \{0\}$, we denote by $[y] = \RR_+^* y \in \P \RR_+^d$ the corresponding direction.\nomenclature[S]{$[y]$}{direction of the vector $y$}
Conversely, for every $x\in \P\RR_+^d$, we denote by $v(x)\nomenclature[L]{$v(x)$}{representative of $x$ of norm $1$} \in \RR_+^{d+1}$
the unique representative of $x$ such that $\norm{v(x)}_1=1$.
And for every matrix $M\in M_{d+1}(\RR)\nomenclature[L]{$M_{d+1}(\RR)$}{square matrices}$ and $x\in \P\RR_+^d$, we write $Mx=[Mv(x)]$ if $Mv(x) \in \RR_+^{d+1} \setminus \{0\}$.

We define a distance on $\P\RR_+^d$, making it a metric space, by $d(x,y)=\norm{v(y)-v(x)}_1$\nomenclature[L]{$d(x,y)$}{distance}.
Note that $\P\RR_+^d$ is thus isometric to the simplex $\Delta = \{y\in\RR_ +^{d+1} \mid \norm{y}_1=1\}$\nomenclature[G]{$\Delta$}{simplex}.
Open balls in $\P\RR_+^d$ are denoted $B(x,r)$\nomenclature[L]{$B(x,r)$}{ball in the projective space}.

Let $h$\nomenclature[L]{$h$}{sum of coordinates} denote the linear form on $\RR^{d+1}$ defined by $h(y_0,\dots,y_{d})=\sum_{i=0}^d y_i$. Note that, when $y\in \RR_+^{d+1}$, $h(y)=\norm{y}_1$.
Let $P$\nomenclature[L]{$P$}{hyperplane where $h$ cancels} be the hyperplane $\{y \in \RR^{d+1} \mid h(y) = 0\}$.
In the following we consider the lattice $\Lambda = P \cap \ZZ^{d+1}= \langle e_1-e_0, \dots, e_d-e_0\rangle$\nomenclature[G]{$\Lambda$}{integer lattice in $P$}.
Let us denote by $\lambda$\nomenclature[G]{$\lambda$}{Lebesgue measure} the Lebesgue measure on $P$.

For $y \in \RR_+^{d+1} \setminus \{0\}$, let $\pi_y$\nomenclature[G]{$\pi_y$, $\pi_{v(x)}$}{projection on $P$ along $y$ or $v(x)$} denote the projection along
$y$ onto $P$ (note that $y$ does not belong to $P$ as $h(y)=\norm{y}_1>0$).
This map sends a vector $z\in \RR^{d+1}$ to $z-h(z)\frac{y}{h(y)}$.
For $x \in \P\RR_+^d$, we also denote $\pi_x = \pi_{v(x)}$.
Remark that $\Lambda \subseteq P$, so $\Lambda$ is preserved by every projection $\pi_y$.

We say that $y=(y_0,\ldots,y_d) \in \RR_+^{d+1}$ has a \emph{totally irrational direction}, or that $[y] \in \P\RR_+^{d}$ is a totally irrational direction, if $y_0, \ldots, y_d$ are linearly independent over $\QQ$.

\subsection{Translations on the torus}\label{subsec-torus}

We define the $d$-dimensional torus as $P/\Lambda$, and let
$q \colon P \to P/\Lambda$\nomenclature[L]{$q$}{quotient map $P \to P/\Lambda$} denote the quotient map.
We still denote by $\lambda$ the Lebesgue measure transported on $P/\Lambda$.
Note that our definition of a torus differs from the usual one $\TT^d=\RR^d/\ZZ^d$\nomenclature[L$T$]{$\TT^d$}{torus},
but they are isomorphic, in a non-canonical way that depends on a choice
of a basis of $\Lambda$.
To fix an isomorphism, let $L\colon P\to\RR^d$\nomenclature[L]{$L$}{linear map from $P$ to $\RR^d$} be the restriction to $P$ of the
linear map which sends $(x_i)_{0\leq i\leq d}$ to $(x_i)_{1\leq i\leq d}$.
Since $L(\Lambda)=\ZZ^d$, $L$ induces a map $\psi\colon P/\Lambda\to\TT^d$\nomenclature[G$\Y$]{$\psi$}{torus isomorphism from $P/\Lambda$ to $\TT^d$} such
that the following diagram commutes:
$$\begin{CD}P @>L>> \RR^d\\  @VVV @VVV\\  P/\Lambda @>\psi>> \TT^d\end{CD}$$

Let $t\in P / \Lambda$, and $\hat t=(t_0, \dots, t_d)\in P$ be a representative of $t$.
Then $t$ is said to be a \emph{totally irrational vector} if $\hat t+e_0$ has a totally irrational direction,
i.e., if $1, t_1, \dots, t_d$ are linearly independent over $\QQ$.
Note that totally irrational vectors should not be confused with totally irrational directions.

For $t\in P / \Lambda$, we consider the associated translation\nomenclature[L]{$T_t$}{translation of the torus by vector $t$}
$$T_t = \map{P / \Lambda}{P / \Lambda}{z}{z+t}$$
We recall that a translation $T_t$ is \emph{minimal} if, and only if, $t$ is a totally irrational vector \cite{Tab.95}.

If $x\in \P\RR_+^d$ is a direction, we denote $T_x = T_{\pi_x(e_0)+\Lambda}$\nomenclature[L]{$T_x$}{translation of the torus associated with the direction $x$}.

Finally remark that any isomorphism $\psi\colon P/\Lambda \to \TT^d$ preserves Lebesgue measure (up to a multiplicative constant) and totally irrational vectors.
This will be used in Theorem~\ref{thmA}.
\subsection{Words}

For a fixed $d \geq 1$, we define the \emph{alphabet} $A$\nomenclature[L]{$A$}{alphabet} as the finite set
$A=\{0,\dots,d\}$. Its elements are called \emph{letters}.
A \emph{finite word} is an element of the monoid $A^* = \bigcup_{n\in\NN}A^n$\nomenclature[L]{$A^*$}{finite words}.
The \emph{length} of a finite word $u\in A^n$ is denoted by $|u|=n$\nomenclature[S]{$\abs{u}$}{length of a word}.
The set of non-empty words is the semigroup $A^+ =\bigcup_{n\geq 1}A^n$.
An \emph{infinite word} is an element of $A^\NN$\nomenclature[L]{$A^\NN$}{infinite words}.
A \emph{word} can be finite or infinite.

The set of words $A^*\cup A^\NN$ is endowed with the topology of coordinatewise
convergence.

When a word $u$\nomenclature[L]{$u$}{word} can be written as a product of three words $pfs$,
$p$ is called a \emph{prefix} of~$u$,
$f$ is called a \emph{factor} of~$u$, 
$s$ is called a \emph{suffix} of~$u$, 
and the length of $p$ is called an \emph{occurrence} of~$f$ in~$u$.
The number of occurrences of a finite word $f$ in a word $u$ is denoted by
$\abs{u}_f$\nomenclature[S]{$\abs{u}_f$}{number of occurences of $f$ in $u$}.

The \emph{complexity} of an infinite word $u$ is the map $p\colon\NN\to\NN$\nomenclature[L]{$p(n)$}{complexity function} which
associates to any integer $n$ the number of factors of $u$ of length $n$.

A \emph{substitution} is an element $\sigma$\nomenclature[G]{$\sigma$}{generic substitution} of $\hom(A^*,A^*)$\nomenclature[L$hom$]{$\hom(A^*,A^*)$}{substitutions}: for all finite
words $u,v \in A^*$, we have $\sigma(uv)=\sigma(u)\sigma(v)$.
A substitution is characterized by the images of letters.
A \emph{non-erasing substitution} is an element $\sigma$ of $\hom(A^+,A^+)$: it is a substitution that maps every letter to non-empty words.\nomenclature[L$hom$]{$\hom(A^+,A^+)$}{non-erasing substitutions}

The \emph{abelianization map} is the monoid morphism $\ab \colon A^* \to \ZZ^{d+1}$\nomenclature[L$ab$]{$\ab$}{abelianization} such
that $\ab(a)=e_a$ for every letter $a$ in $A$
(recall that $(e_a)_{a\in A}$ is the canonical basis of $\RR^{d+1}$).
We use the same notation for the map from $\hom(A^*,A^*)$ to $M_{d+1}(\ZZ)$ such
that $\ab(\sigma)\ab(w)=\ab(\sigma(w))$ for a substitution $\sigma$ and a word
$w$.
A substitution $\sigma$ is said to be \emph{unimodular} if
$\abs{\det(\ab(\sigma))}=1$.

The action of a non-erasing substitution $\sigma$ can be extended to infinite words by the
limit procedure:
$$\sigma(u) = \lim_{\substack{\vphantom{0} u = ps \\ \abs{p}\to\infty}} \sigma(p)$$
An infinite word $u$ is a \emph{fixed point} of $\sigma$ if $\sigma(u)=u$. An infinite word $u$ is a \emph{periodic point} of $\sigma$ is there exists an integer $p\geq 1$ such that $\sigma^p(u)=u$.

For an integer $k\geq 1$, an infinite word $u$ is \emph{$k$-balanced} if for any
two factors $v$, $w$ of $u$ of the same length, and any $a\in A$, we have
$\abs{\abs{v}_a-\abs{w}_a}\leq k$. An infinite word is \emph{balanced} if it is
\emph{$k$-balanced} for some integer $k$.

For an infinite word $u$, the (possibly undefined) \emph{frequency vector} of $u$ is\nomenclature[L$freq$]{$\freq(u)$}{frequency vector}
$$\freq(u) = \lim_{\substack{\vphantom{0} u = ps \\ \abs{p}\to\infty}} \frac{\ab(p)}{\abs{p}} \in \RR^{d+1}.$$
When this limit exists, we say that $u$ admits a frequency vector.
This is in particular the case if $u$ is balanced (see Proposition~\ref{prop-bal-freq}) or if it is an element of a uniquely ergodic subshift.

For a non-empty finite word $w\in A^+$, we denote by $w^\omega$ the infinite word $\lim_{n\to\infty} w^n$\nomenclature[G$\z$]{$w^\omega$}{periodic word}.

Finally we define the \emph{shift map} $T$\nomenclature[L]{$T$}{shift map} on $A^\NN$
that maps an infinite word $u$ to its suffix $Tu$ such that $u=aTu$ with $a\in A$.
Remark that with the coordinatewise topology, the shift map $T$ is continuous, and a subset $X\subseteq A^\NN$ is called a \emph{subshift} if $X$ is closed and shift-invariant.

The \emph{orbit} of $u \in A^\NN$ is the set $\orbit{u} = \{T^n u \mid n \in \NN \}$\nomenclature[L$O$]{$\orbit{u}$}{orbit} and
the \emph{subshift generated} by $u$ is its orbit closure $\Omega_u =
\overline{\orbit{u}}$\nomenclature[G$\Z$]{$\Omega_u$}{subshift generated by $u$}.
To a finite factor $w$ of $u$ we associate the \emph{cylinder}
$[w]=\{x\in\Omega_u\mid \exists s \in \Omega_u, x = ws\}$\nomenclature[S]{$[w]$}{cylinder}.

\subsection{Symbolic coding}\label{subsec-coding}

A \emph{measured topological dynamical system} is a triple $(X,T,\mu)$ such that
$X$\nomenclature[L]{$X$}{base set of a dynamical system} is a compact topological space, 
$\mu$\nomenclature[G]{$\mu$}{Borel measure} is a finite Borel measure, 
and $T\colon X\to X$ is a $\mu$-almost everywhere continuous map such that 
$\mu(T^{-1}(B)) = \mu(B)$ for any Borel set $B$ of $X$.

Given a measured topological dynamical system $(X,T,\mu)$ and a measurable
partition $(P_i)_{i \in I}$ of $X$, 
we associate the \emph{coding} $\cod\colon X \to
I^\NN$\nomenclature[L$cod$]{$\cod$}{symbolic coding} defined by $\cod(y)=(i_n)_{n\in\NN}$ and $\forall n \in \NN$, ${T}^ny\in
P_{i_n}$.
The map $\cod$ is a \emph{symbolic coding} of the system $(X,T,\mu)$ and the
closure of $\cod(X)$ defines a subshift over the alphabet $I$.
A \emph{generating partition} of the map $T$ is a measurable partition whose
coding is injective $\mu$-almost everywhere.

The atoms of the partitions we will construct will not be smooth, but they will
keep some topological and measure-theoretic regularity: a generating partition
$(P_i)_{i\in I}$ of $X$ is \emph{regular} if every set $\overline P_i$ is the
closure of its interior and if the boundary of each $P_i$ has measure zero.

A measurable subset $A$ of $X$ is said to be a \emph{bounded remainder set} for
the map $T$ if there exists a constant $K$ such that, for $\mu$-almost every
$x$ in $X$ and every integer $N$,
$$
	\abs{\sum_{n=0}^{N-1} \indic_{A}(T^n(x)) - N\frac{\mu(A)}{\mu(X)}}\leq K,
$$
where $\indic_{A}$\nomenclature[S]{$\indic_{A}$}{indicator function} is the indicator function of the subset $A$.
As we shall see, the atoms of the generating partition we will construct are
bounded remainder sets.

Now, let $T_x$ be a translation of the torus $P/\Lambda$. The triple $(P/\Lambda,
T_x, \lambda)$ is a measured topological dynamical system, where $\lambda$
denotes the Lebesgue measure inherited from $P$.
The generating partitions we will construct on $P/\Lambda$ actually come from a
piecewise translation of a measurable fundamental domain of $P$ for the action of
$\Lambda$:
a finite measurable partition $(P_i)_{i \in I}$ of $P/\Lambda$ is said to be \emph{liftable} with
respect to the translation $T_t\colon z\mapsto z+t$ of $P/\Lambda$ if there exists:
\begin{itemize}
  \item a measurable fundamental domain $D \subseteq P$ for the action of $\Lambda$
  \item a measurable partition $(D_i)_{i \in I}$ of $D$
  \item some vectors $(t_i)_{i\in I}$ in $P^I$
\end{itemize}
such that for every $i$ in $I$:
\begin{itemize}
  \item $D_i+t_i \subseteq D$
  \item $q(D_i) = P_i$
  \item $q(t_i) = t$
\end{itemize}

The map $E=\map{D}{D}{y}{y+t_i\mbox{ if }y\in D_i}$\nomenclature[L]{$E$}{domain exchange} is called a \emph{piecewise
translation} or a \emph{domain exchange}, and is measurably conjugated to the
translation $T_t$ via the quotient map $q : P \to P/\Lambda$.

\begin{definition}\label{def-nice}
A finite measurable partition $(P_i)_{i\in I}$ of $P/\Lambda$ is said to be
a \emph{nice generating partition} with respect to the translation
$T_t\colon z\mapsto z+t$ of
$P/\Lambda$ if it is generating, regular, liftable, and
every $P_i$ is a bounded remainder set.
\end{definition}

\subsection{\texorpdfstring{$S$}{S}-adic systems and \texorpdfstring{$S$}{S}-adic subshifts} \label{ss:Sadic}

Let $S \subseteq \hom(A^+,A^+)$\nomenclature[L]{$S$}{finite set of substitutions} be a finite set
of non-erasing substitutions on the alphabet $A$.

An \emph{$S$-adic system} is a shift-invariant subset of $S^\NN$.
Note that we do not impose that $S$-adic systems are topologically closed.
For instance, in Section~\ref{exemple-1d-sturmien} we will consider
$S=\{\tau_0,\tau_1\}$ and the $S$-adic system
$\{(s_k)\in S^\NN \mid \text{each element of $S$ occurs infinitely often in }(s_k)\}$.

An element $s=(s_k)$\nomenclature[L$s_k$]{$(s_k)$}{directive sequence} of an $S$-adic system is called a \emph{directive sequence}.
\begin{definition}[$S$-adic subshift]
Let $s$ be a directive sequence. Then the \emph{$S$-adic subshift}
associated with $s$ is the subshift $\Omega_s$\nomenclature[G$\Z$]{$\Omega_s$}{$S$-adic subshift} defined as follows.
Let first $L\subseteq A^*$ be the language of all factors of
finite words of the form $s_{[0,n)}(a)$ for all $n\in\NN$ and $a\in A$,
where $s_{[k,n)} = s_k \circ \dots \circ s_{n-1}$\nomenclature[L$s_k$]{$s_{[k,n)}$}{product of substitutions}.
Then $\Omega_s$ is the set of
infinite words $w\in A^\NN$ such that all factors of $w$ are in $L$.
\end{definition}

$S$-adic subshifts were introduced by Ferenczi \cite{Fer.96}, where he proves
that every word of linear complexity is an element of some $S$-adic subshift
in an $S$-adic system with some additional conditions.
This notion has been used in many places thereafter.
We refer to \cite{Dur.Ler.Rich.13} and \cite{Bert.Delec.14} for reference.

\begin{remark}
There are alternative ways to define subshifts from a directive sequence.
One is to consider the set
$$\Omega'_s = \bigcap_{n\in\NN}\{T^k s_{[0,n)}(w)\mid w\in A^\NN, k\in\NN\}.$$
$\Omega'_s$ is the set of words that are infinitely desubstituable by $s$,
and it always holds that $\Omega_s\subseteq\Omega'_s$.

Another way is to first define an infinite word $u$ by starting from
a fixed letter $a\in A$ and taking a limit point
of the sequence of finite words
$$s_0(a), s_0(s_1(a)), s_0(s_1(s_2(a))), \dots$$
then consider the subshift generated by $u$ (i.e., the smallest closed subset
of $A^\NN$ invariant by the shift and containing $u$), which is a subset
of $\Omega_s$.
\end{remark}

Here, we will let the directive sequence act on sequences of infinite words, each
word representing a scale on which the corresponding substitution acts.

A \emph{word sequence} is an element $u = (u_k)$ of ${(A^\NN)}^\NN$\nomenclature[L]{$u=(u_k)$}{word sequence, usually fixed point}.

Directive sequences act naturally on word sequences as follows:

    $$\map{S^\NN\times{(A^\NN)}^\NN}{{(A^\NN)}^\NN}{(s,u)}{\map{\NN}{A^\NN}{k}{s_k(u_{k+1})}}$$

\begin{definition} \label{def:fixed:point}
	A \emph{fixed point} of a directive sequence $s$ is a fixed point
	for the above action, that is, a word sequence $u$ satisfying:
		\[
			\forall k \in \NN,\ s_k(u_{k+1}) = u_k.
		\]
\end{definition}

Example~\ref{exple-sturmien} gives an example of fixed point of a directive sequence.

Directive sequences always admit fixed points.
Indeed, choose a letter $a\in A$ and for each $n\in\NN$, consider the word sequence $u^{(n)}=(u_k^{(n)})_{k\in\NN}$ defined by $u_k^{(n)}=a^\omega$ when $k\geq n$ and $u_k^{(n)}=s_{[k,n)}(a^\omega)$ when $k<n$,
where $s_{[k,n)} = s_k \circ \dots \circ s_{n-1}$.
Then let $u$ be a limit point of this sequence of word sequences when $n$ tends to infinity, in the compact space ${(A^\NN)}^\NN$ (with the coordinatewise topology).
This $u$ is a fixed point of $s$.

Fixed points of a directive sequence are not unique in general.

This generalizes the notion of fixed point and the notion of periodic point for a single substitution $\sigma$.
Let $\sigma^\omega$\nomenclature[G]{$\sigma^\omega$}{constant directive sequence} denote the constant directive sequence with all terms equal to $\sigma$. Similarly, for $v\in A^\NN$, let $v^\omega$ denote the word sequence with all terms equal to $v$.

\begin{lemma}\label{lem-lien-pt-fixes}
	Let $\sigma \in S$ be a substitution. We have
	\begin{itemize}
		\item $v \in A^\NN$ is a fixed point of $\sigma$ if, and only if, $v^\omega \in {(A^\NN)}^\NN$ is a fixed point of $\sigma^\omega$, 
		\item if $u \in {(A^\NN)}^\NN$ is a fixed point of $\sigma^\omega$, then $u_0$ is a periodic point of $\sigma$.
	\end{itemize}
\end{lemma}

\begin{proof}
	The first point is clear.
	Let $u$ be a fixed point of $\sigma^\omega$, and for all $n \in \NN$, let $a_n \in A$ be the first letter of $u_n$.
	Then, the sequence $a_n$ is periodic, with a period $p \leq d+1$ since $a_{n}$ is completely determined by $a_{n+1}$.
	Now, if $\lim_{n \to \infty} \abs{\sigma^n(a_0)} = \infty$, then $\sigma^{np}(a_0)$ converges as $n$ tends to infinity to the word $u_0 = u_{p}$, so $u_0$ is a periodic point of $\sigma$.
	Otherwise, we have for all $n \in \NN$, $\sigma^n(a_0) = a_n$,
	and the directive sequence $(T u_n)_{n \in \NN}$ is also a fixed point of $\sigma^\omega$,
	where $T\colon A^\NN \to A^\NN$ is the shift map.
	If we iterate the argument and take the least common multiple of the periods obtained, it gives a period $p$ for which $u_0$ is a fixed point of $\sigma^p$. 
\end{proof}

\begin{definition}
	For a fixed point $u \in {(A^\NN)}^\NN$ of a directive sequence $s$,
	we define the subshift $\Omega_u$ as the subshift $\Omega_{u_0}$, that is
	the smallest closed subset of $A^\NN$ invariant by the shift and containing $u_0$.
\end{definition}

\begin{definition}
	We say that a directive sequence $s \in S^\NN$ is \emph{primitive} if
	\[
		\forall k,\ \exists n \geq k,\ \forall a \in A,\ s_{[k,n)}(a) \text{ contains every letter}.
	\]
	It is equivalent to
	$
		\forall k,\ \exists n \geq k,\ \ab(s_{[k,n)}) > 0.
	$
\end{definition}

\begin{definition} \label{def:everywhere:growing}
	We say that a directive sequence $s \in S^\NN$ is \emph{everywhere growing} if for all $a \in A$, we have
	\[
		\lim_{n \to \infty} \abs{s_{[0,n)}(a)} = \infty.
	\]
	It is equivalent to say that the $1$-norm of each column of the matrix $\ab(s_{[0,n)})$ tends to infinity.
\end{definition}

Remark that if a directive sequence $s$ is primitive, then for all $k \in \NN$, and all $a \in A$,
we have $\abs{s_{[k,n)}(a)} \xrightarrow[n \to \infty]{} \infty$. In particular, $s$ is everywhere growing.

\begin{proposition}\label{prop-omega-u-omega-s}
	Let $s \in S^\NN$ be a primitive directive sequence.
	Then the subshift $\Omega_s$ is minimal.
	In particular, for every fixed point $u \in {(A^\NN)}^\NN$ of $s$, we have
	$
		\Omega_u = \Omega_s.
	$
	Thus, $\Omega_u$ does not depend on the choice of the fixed point $u$.
\end{proposition}

\begin{proof}
	Let $w$ and $w' \in \Omega_s$ be two words of the subshift.
	Let $p$ be a prefix of $w$.
	Then, there exists $n \in \NN$ and $a \in A$ such that $p$ is a factor of $s_{[0,n)}(a)$.
	Using the primitivity, let $N \geq n$ such that $\ab(s_{[n,N)}) > 0$.
	Now, take a factor $f$ of $w'$ of length at least $2 \max_{c \in A} \abs{s_{[0,N)}(c)}$.
	There exists $k \in \NN$ and $b \in A$ such that $f$ is a factor of $s_{[0,k)}(b)$.
	Necessarily, we have $k \geq N$, and $s_{[0,k)}(b)$ is a concatenation of words $s_{[0,N)}(c)$, for each letter $c$ of $s_{[N,k)}(b)$.
	Hence, there exists $c \in A$ such that $s_{[0,N)}(c)$ is a factor of $f$.
	Then, the letter $a$ appears in the word $s_{[n,N)}(c)$.
	So $p$ is a factor of $s_{[0,n)}(a)$ which is a factor of $s_{[0,N)}(c)$ which is a factor of $f$ which is a factor of $w'$.
	We conclude that for every $w, w' \in \Omega_s$, every prefix of $w$ is a factor of $w'$,
	thus the subshift $\Omega_s$ is minimal.
	
	To end the proof, remark that for any fixed point $u$ of $s$, we have $u_0 \in \Omega_s$ since we have $\lim_{n \to \infty} s_{[0,n)}(a_n) = u_0$, where $a_n$ is the first letter of $u_n$.
	Hence, by minimality we get $\Omega_u = \Omega_s$.
\end{proof}

Let us give an example of a fixed point of a directive sequence.
The reader will recognize that each $u_{k}$ is a Sturmian word, see \cite{Pyth.02}.
\begin{example}\label{exple-sturmien}
	Let $S = \{\tau_0, \tau_1\}$, with $\tau_0 = \left\{\begin{array}{ccl} 0 &\mapsto& 0 \\ 1 &\mapsto& 01 \end{array} \right.,
				\tau_1 = \left\{\begin{array}{ccl} 0 &\mapsto& 10 \\ 1 &\mapsto& 1 \end{array} \right.$,
	and let us consider the directive sequence $s = \tau_0 \tau_1 \tau_1 \tau_0 \tau_0 \tau_1 \tau_0 \tau_0 \tau_1 ...$.
    Then, there exists a fixed point $u \in {(\{0,1\}^\NN)}^\NN$ of $s$ beginning with
	\begin{align*}
		u_{0} &=& 01010010100101010010100101001010100101001010010101... \\
		u_{1} &=& 11011011101101101110110110111011011101101101110110... \\
		u_{2} &=& 10101101010110101011010110101011010101101011010101... \\
		u_{3} &=& 00100010001001000100010010001000100010010001000100... \\
		u_{4} &=& 01001001010010010100100100101001001010010010010100... \\
		u_{5} &=& 10101101011010101101011010101101011010101101011010... \\
		u_{6} &=& 00100100010010001001000100100100010010001001000100... \\
	\end{align*}
\end{example}

Fixed points encompass both the time and scale dynamics in a single object.
In the context of this paper, the time dynamics will correspond to the action of
the translation on the torus and the scale dynamics will correspond to the action
of the continued fraction algorithm on the space of translations.
Symbolically, the shift map on $\Omega_{u_0}$ encodes the time dynamics, while
shifting the fixed point $(u_k)\mapsto (u_{k+1})$ corresponds to accelerating
the time dynamics.

In the example \ref{exple-sturmien} above, we can vizualize how fixed points
grasp the the multi-scale structure of the dynamical system with the following
alignment:

{
    \let~\trou 
	\begin{align*}
		u_{0} &=& 01010010100101010010100101001010100101001010010101... \\
		u_{1} &=& 1~1~01~1~01~1~1~01~1~01~1~01~1~1~01~1~01~1~01~1~1~... \\
		u_{2} &=& 1~0~~1~0~~1~1~0~~1~0~~1~0~~1~1~0~~1~0~~1~0~~1~1~0~... \\
		u_{3} &=& 0~~~~0~~~~1~0~~~~0~~~~0~~~~1~0~~~~0~~~~0~~~~1~0~~~... \\
		u_{4} &=& 0~~~~1~~~~~~0~~~~0~~~~1~~~~~~0~~~~0~~~~1~~~~~~0~~~... \\
		u_{5} &=& 1~~~~~~~~~~~0~~~~1~~~~~~~~~~~0~~~~1~~~~~~~~~~~1~~~... \\
		u_{6} &=& 0~~~~~~~~~~~~~~~~0~~~~~~~~~~~~~~~~1~~~~~~~~~~~0~~~... \\
	\end{align*}
}

As we will see in section \ref{sec-renormalization}, when the substitutions enjoy
some recognizability properties, the scale dynamics corresponds to inducing on
some atoms of the partition.

Rokhlin towers and ordered Bratteli diagrams are other combinatorial objects
that account for the multi-scale structure of dynamical systems.
An ordered Bratteli diagram $\mathcal{B}$ can be associated to a directive
sequence $(s_k)$ \cite{Durand.10}.\nomenclature[L$B$]{$\mathcal{B}$}{Bratteli diagram}
When $(s_k)$ is everywhere growing, the minimal infinite paths of $\mathcal{B}$
are in bijective correspondance with the fixed points $(u_k)$ of $(s_k)$: the
$k$th edge of the infinite path is encoded by the first letter of the word
$u_k$.

\subsection{Matrices}
To each substitution $\sigma$ is associated a matrix $\ab(\sigma)$.
To obtain precise results, we need to recall some facts about matrices.
Recall that $\RR^{d+1}$ is equipped with the norm $\norm{.}_1$.
The operator norm of a matrix $M\in M_{d+1}(\RR)$ is defined by
$$\matrixnorm{M}_1=\sup_{v\in \RR^{d+1}\setminus\{0\}}\frac{\norm{Mv}_1}{\norm{v}_1}.$$
\nomenclature[S]{$\matrixnorm{M}_1$}{operator norm}

Moreover we also define a semi-norm for a subspace $V$:
$$\matrixnorm{M\restrict{V}}_1=\sup_{v\in V\setminus\{0\}}\frac{\norm{Mv}_1}{\norm{v}_1}.$$
\nomenclature[S]{$\matrixnorm{M\restrict{V}}_1$}{operator semi-norm} 
Finally we write $M>0$ if every coefficient of $M$ is positive.

Given a directive sequence $s$, we define $M_{k}(s) = \ab(s_k)$\nomenclature[L]{$M_k(s)$}{$k$-th matrix of $s$},
denoted simply by $M_k$ when there is no ambiguity on what is the directive sequence.
We use the classical notation $M_{[k,n)} = M_k \dots M_{n-1}$\nomenclature[L]{$M_{[k,n)}$}{product $M_k\dots M_{n-1}$}.

A matrix $M\in M_{d+1}(\RR)$ is said to be \emph{Pisot} if it has non-negative integer entries, its dominant eigenvalue is simple and all other eigenvalues have absolute values less than one. A substitution $\sigma$ is said to be \emph{Pisot} if the matrix $\ab(\sigma)$ is Pisot.

\subsection{Topologies on the integer half-space and worms}\label{subsection-topo}

Let us define the \emph{integer half-space}:
$\HH = \{z \in \ZZ^{d+1} \mid h(z) \geq 0\}$\nomenclature[L$H$]{$\HH$}{integer half-space}
and for $i\geq 0$, $\HH_i = \{z \in \ZZ^{d+1} \mid h(z)=i\}=\Lambda+ie_0$.
The following two definitions are crucial in the rest of the paper.      
\begin{definition}\label{def-topo-paul}
    For any fixed $x\in \P\RR_+^{d}$, we define the topology
    $\topo{x}$\nomenclature[L$T$]{$\topo{x}$}{topology on $\HH$} on $\HH$: a subset $V \subseteq \HH$ is \emph{open} if there exists
    an open set $U\subseteq P$ such that $V = \pi_x^{-1}(U)\cap \HH$. 
\end{definition}    
\nomenclature[L]{$U$}{open subset of $P$}
\nomenclature[L]{$V$}{open subset of $\HH$ for some topology $\topo{x}$}
Remark that $\topo{x}$ is the finest topology on $\HH$ that makes $\pi_x\colon\HH\to P$ continuous.
It is metrizable if, and only if, $x\cap\HH=\{0\}$,
which is the case if $x$ is a totally irrational direction.     
     
We introduce the notion of \emph{worm}: 
\begin{definition} \label{def:worm}
        Given an infinite word $u\in A^{\NN}$, its \emph{worm} is the set
		$$W(u) = \{\ab(p) \mid p \text{ prefix of } u\} \subseteq \NN^{d+1}
		\subseteq \HH.$$\nomenclature[L]{$W(u)$}{worm}

	For a worm and a letter $a$ we can define the subsets
	$$W_a(u)= \{\ab(p) \mid pa \text{ prefix of } u\}=W(u)\cap(W(u)-e_a).$$\nomenclature[L]{$W_a(u)$}{subset of a worm}
\end{definition}

\begin{remark}
	The subsets $W_a(u)$, $a \in A$, form a partition of $W(u)$: $W(u) = \sqcup_{a \in A} W_a(u)$.
\end{remark}

An example of worm is depicted in Figure~\ref{fig:worm}.

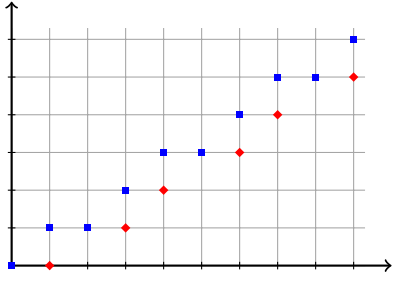
\begin{figure}[ht]
    \centering
	\begin{tikzpicture}[scale=.5, xscale=1, yscale=1]
		\draw [gray!70] (0,0) grid (9.3,6.3);
		\draw[->, thick] (0,0) -- (10,0);
		\draw[->, thick] (0,0) -- (0,7);
		\foreach \y in {1,2,3,4,5,6}:
			\draw (.1,\y) -- (-.1,\y); 
		\foreach \x in {1,2,3,4,5,6,7,8,9}:
			\draw (\x, .1) -- (\x, -.1); 
		\def\d{.125}
		\def\dd{.0884}
		\foreach \x/\y/\c in {0/0/0, 1/0/1, 1/1/0, 2/1/0, 3/1/1, 3/2/0, 4/2/1, 4/3/0, 5/3/0, 6/3/1, 6/4/0, 7/4/1, 7/5/0, 8/5/0, 9/5/1, 9/6/0}
		{
			\ifthenelse{\c = 0}
			{
				\fill[blue] (\x-\dd, \y-\dd) -- (\x-\dd, \y+\dd) -- (\x+\dd, \y+\dd) -- (\x+\dd, \y-\dd)
			}{
				\fill[red] (\x, \y-\d) -- (\x+\d,\y) -- (\x,\y+\d) -- (\x-\d,\y)
			};
		}
	\end{tikzpicture}
    \caption{Worm of the word $u = (01001)^\omega$. $W_0(u)$ in blue and $W_1(u)$ in red.} \label{fig:worm}
\end{figure}

\begin{lemma}[tiling]\label{lem:worm:tiling}
  A worm $W$ tiles the integer half-space by translations: $\HH = W \oplus \Lambda$.
\end{lemma}

Figure~\ref{fig:H:worm} shows an example of such a tiling by the worm $(01001)^\omega$, for $d=1$.

\begin{proof}
  
The lattice $\Lambda$ acts on $\HH$ by translation.
The orbits of this action are the cosets $\HH_i$.
A worm intersects each coset exactly once.
\end{proof}

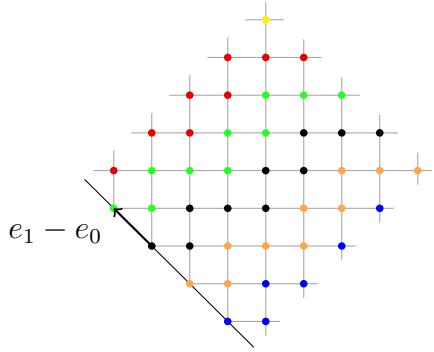
\begin{figure}[ht]
    \centering
	\begin{tikzpicture}[scale=.5]
		\begin{scope}
			\clip[rotate around={-45:(0,0)}] (-2.5,0) rectangle (3.8,6.7);
			\draw [gray!70] (-9,-9) grid (9,9);
		\end{scope}
		\draw[rotate around={-45:(0,0)}] (-2.5,0) -- (3.8,0);
		\begin{scope}
			\clip[rotate around={-45:(0,0)}] (-2.5,-.2) rectangle (3.8,6.7);
			\foreach \x/\y in {0/0, 1/0, 1/1, 2/1, 3/1, 3/2, 4/2, 4/3, 5/3, 6/3}:
			{
				\foreach \n/\c in {-3/yellow, -2/red!90!black, -1/green!85, 0/black, 1/orange!70, 2/blue}:
					\fill[\c] (\x+\n,\y-\n) circle (.1);
			};
		\end{scope}
    \draw[->, thick] (0,0) -- (-1,1) node[below left] {$e_1-e_0$};
	\end{tikzpicture}
    \caption{Tiling of $\HH$ by a worm, by translation by the group $\Lambda = \langle e_1-e_0\rangle$} \label{fig:H:worm}
\end{figure}

\begin{lemma}[automatic balance] \label{lem:worm:balance}
If a worm $W(u)$ has non-empty interior for some topology $\topo{x}$ with
$x\in\P\RR_+^{d}$, then $\pi_x(W(u))$ is bounded.
\end{lemma}

\begin{proof}
Let $U$ be an open subset of $P$ such that $\emptyset \neq \pi_x^{-1}(U)\cap \HH \subseteq W(u)$.
Up to restricting it, we assume that $U$ is included in an open ball $B(0,M)$ for some $M>0$.

Let us consider the translation by $\pi_x(e_0)$ modulo $\Lambda$ of the
$d$-dimensional torus $P/\Lambda$:

        $$T_{x} = T_{\pi_x(e_0)+\Lambda} = \map{P/\Lambda}{P/\Lambda}{z}{z+\pi_x(e_0)}$$

For any integer $i$, $\pi_x^{-1}(U)$ intersects $\HH_i$ if, and only if,
$T_{x}^i(0)$ belongs to the open subset $U+\Lambda$ of the torus.
The translation $T_{x}$ acts minimally on every orbit closure.
By hypothesis, the open set $U+\Lambda$ intersects the orbit of $0$, hence
$T_{x}^i(0)$ belongs to $U+\Lambda$ for $i$ in a syndetic subset of $\NN$:
$\exists K\geq 1, \forall i\in\NN, \exists 0\leq k\leq K, T_{x}^{i+k}(0) \in U+\Lambda$.
If, for each integer $i$, we denote by $w_i$ the single element of
$W(u)\cap\HH_i$, we have $\norm{w_{i+1}-w_i}_1=1$.
Hence, any point of the worm $W(u)$ is at distance at most $K$ of a point of
$\pi_x^{-1}(U)$ (see Figure \ref{fig:escape}).

\begin{figure}[h]
\centering{
\begingroup%
  \makeatletter%
  \providecommand\color[2][]{%
    \errmessage{(Inkscape) Color is used for the text in Inkscape, but the package 'color.sty' is not loaded}%
    \renewcommand\color[2][]{}%
  }%
  \providecommand\transparent[1]{%
    \errmessage{(Inkscape) Transparency is used (non-zero) for the text in Inkscape, but the package 'transparent.sty' is not loaded}%
    \renewcommand\transparent[1]{}%
  }%
  \providecommand\rotatebox[2]{#2}%
  \newcommand*\fsize{\dimexpr\f@size pt\relax}%
  \newcommand*\lineheight[1]{\fontsize{\fsize}{#1\fsize}\selectfont}%
  \ifx\svgwidth\undefined%
    \setlength{\unitlength}{224.20109065bp}%
    \ifx\svgscale\undefined%
      \relax%
    \else%
      \setlength{\unitlength}{\unitlength * \real{\svgscale}}%
    \fi%
  \else%
    \setlength{\unitlength}{\svgwidth}%
  \fi%
  \global\let\svgwidth\undefined%
  \global\let\svgscale\undefined%
  \makeatother%
  \begin{picture}(1,0.77076091)%
    \lineheight{1}%
    \setlength\tabcolsep{0pt}%
    \put(0,0){\includegraphics[width=\unitlength,page=1]{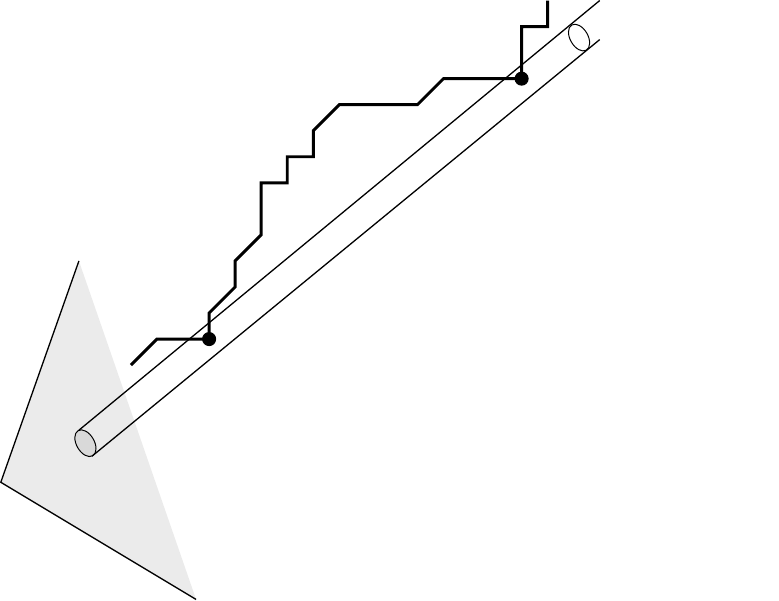}}%
    \put(0.18517229,0.05766422){\color[rgb]{0,0,0}\makebox(0,0)[lt]{\lineheight{1.25}\smash{\begin{tabular}[t]{l}$P$\end{tabular}}}}%
    \put(0,0){\includegraphics[width=\unitlength,page=2]{escape.pdf}}%
    \put(0.04918968,0.17689564){\color[rgb]{0,0,0}\makebox(0,0)[lt]{\lineheight{1.25}\smash{\begin{tabular}[t]{l}$U$\end{tabular}}}}%
    \put(0.78176419,0.74084518){\color[rgb]{0,0,0}\makebox(0,0)[lt]{\lineheight{1.25}\smash{\begin{tabular}[t]{l}$\pi_x^{-1}(U)$\end{tabular}}}}%
    \put(0.57985685,0.3758588){\color[rgb]{0,0,0}\makebox(0,0)[lt]{\lineheight{1.25}\smash{\begin{tabular}[t]{l}$\leq K$\end{tabular}}}}%
  \end{picture}%
\endgroup%

\caption{A worm can not escape too far between consecutive interior points}
\label{fig:escape}
}
\end{figure}

Since the direction $x$ is in $\P\RR_+^{d+1}$, the projection
$\pi_x$ is $2$-Lipschitz.
Hence, $\pi_x(W(u))\subseteq B(0,M+2K)$, which concludes the proof.

\end{proof}

\begin{lemma}[uniform automatic balance] \label{lem:worm:balance:uniform}
	If there exists a ball $B$ of $P$, a sequence of directions $x^{(n)}\in\P\RR_+^{d}$ such that $x^{(n)} \xrightarrow[n \to \infty]{} x^\infty$ with $x^\infty$ a totally irrational direction, and a sequence of worms $W(u_n)$ such that $\forall n \in \NN$, $\pi_{x^{(n)}}^{-1}(B) \cap \HH \subseteq W(u_n)$, then $\pi_{x^{(n)}}(W(u_n))$ is uniformly bounded for $n\in\NN$.
\end{lemma}

\begin{proof}
Following the proof of the previous lemma, we consider the translation $T_{x^\infty}$
by $\pi_{x^\infty}(e_0)+\Lambda$ on the torus $P/\Lambda$.
Let $M > 0$ and $p$ such that $B=B(p,M)$, and define $B'=B(p,M/2)$.
Since $x^\infty$ is a totally irrational direction, the translation
$T_{x^\infty}$ acts minimally on the whole torus $P/\Lambda$.
Hence, there exists a constant $K$
such that for all $x \in B$, there exists $0 < i \leq K$
such that $T_{x^\infty}^i(x) \in B'$.
If we take $n_0$ such that for all $n \geq n_0$,
$d(x^{(n)}, x^\infty)\leq\frac{M}{2K}$,
then we have that for all $y \in B$, there exists $0 < i \leq K$ such that
for all $n \geq n_0$, $T_{x^{(n)}}^i(y) \in B$.
Similarly, there exists $0 < j \leq K$ such that 
for all $n \geq n_0$, $T_{x^{(n)}}^{-j}(y) \in B$.
Hence, for all $n \geq n_0$, $\pi_{x^{(n)}}(W(u_n)) \subseteq B(p, M+2K)$.
The result follows since $W(u_n)$ is bounded for $n<n_0$ by Lemma~\ref{lem:worm:balance}.
\end{proof}

The following lemma is useful to propagate non-emptiness of the interior.

\begin{lemma} \label{lem:O:matrix}
	Let $W \subseteq \HH$, let $x\in\P\RR_+^{d}$ be a totally irrational direction, and let $M \in GL_{d+1}(\ZZ) \cap M_{d+1}(\NN)$.
	If $W$ has non-empty interior for the topology $\topo{x}$, then $M W$ has non-empty interior for the topology $\topo{Mx}$.
\end{lemma}
  
\begin{proof}
	Let $U \subseteq P$ be a bounded non-empty open subset such that $\pi_x^{-1}(U) \cap \HH \subseteq W$.
	We have $M \in GL_{d+1}(\ZZ)$, so $M^{-1} \ZZ^{d+1} = \ZZ^{d+1}$, and $M^{-1} \HH$ is a half-space
	\[
		M^{-1} \HH = \{z \in \ZZ^{d+1} \mid \sum_{i=0}^d (M z)_i \geq 0\} = \{z \in \ZZ^{d+1} \mid \sum_{i=0}^d \alpha_i z_i \geq 0\},
	\]
	for some coefficients $\alpha_i$.
        Since the matrix $M$ is non-negative and invertible, we have $\alpha_i>0$ for all $i$, so $v(x)$ is in the half-space $\{z \in \RR^{d+1} \mid \sum_{i=0}^d \alpha_i z_i > 0\}$.
        Hence, for every $t \in \RR^{d+1}$, the intersection of the line $t + \RR v(x)$ with the set
	\[
		\{ z \in \RR^{d+1} \mid \sum_{i=0}^d \alpha_i z_i \geq 0 \} \setminus \{ z \in \RR^{d+1} \mid \sum_{i=0}^d z_i \geq 0 \}
	\]
	is bounded.
	Using that moreover $U$ is bounded we get that
	\[
		L = \pi_x^{-1}(U) \cap M^{-1} \HH \setminus \HH \subseteq \ZZ^{d+1}
	\]
	is a finite set.
	Moreover, we have $\pi_x^{-1}(\pi_x(L)) \cap M^{-1} \HH = L$ since $\pi_x$ is injective on $\ZZ^{d+1}$, and we have $M \pi_x^{-1}(U \setminus \pi_x(L)) = \pi_{Mx}^{-1}(\pi_{Mx}(M(U\setminus \pi_x(L))))$,
so we have
	\begin{align*}
		M W &\supseteq M (\pi_x^{-1}(U) \cap \HH) \\
			&\supseteq M (\pi_x^{-1}(U) \cap M^{-1} \HH \setminus L) \\
			&= M (\pi_x^{-1}(U \setminus \pi_x(L))) \cap \HH \\
			&= \pi_{Mx}^{-1}(\pi_{Mx}(M(U \setminus \pi_x(L)))) \cap \HH.
	\end{align*}
	Finally $\pi_{Mx}(M(U \setminus \pi_x(L))) \neq \emptyset$ is open, so $M W$ has non-empty interior for $\topo{M x}$.
\end{proof}

We finish this subsection with a result that shows how to relate properties of the worm to some combinatorial properties of $u$:

\begin{proposition}\label{prop-bal-freq}
	An infinite word $u \in A^\NN$ is balanced if, and only, if there exists a direction
	$x \in \P\RR_+^d$
	such that $\pi_x(W(u))$ is bounded.
\end{proposition}
\begin{proof}
	First of all, remark that a word $u$ is balanced if, and only if,
    there exists a constant $K$ such that, for any two factors $v$, $w$ of $u$,
    $\norm{\ab(v)}_1 = \norm{\ab(w)}_1 \implies \norm{\ab(v)-\ab(w)}_1 \leq K$.
	
	Assume that $\pi_x(W(u)) \subseteq B(0,L)$.
	First of all remark that for a finite word $w$ its length fulfills $\abs{w} = h(\ab(w)) = \norm{\ab(w)}_1$.
	Moreover if $p$ is a prefix of $u$, then 
	we have
	$\ab(p)-\abs{p} v(x) = \pi_x(\ab(p)) \in \pi_x(W(u))$, so we get $\norm{\ab(p)-\abs{p} v(x)}_1\leq  L$.

	Let $w$ be a factor of $u$ and let $p$ be a prefix of $u$ such that $pw$ is a prefix of $u$.
	Since $\ab(w)=\ab(pw)-\ab(p)$ and $\abs{w}=\abs{pw}-\abs{p}$ we obtain:
	\[  \norm{\ab(w)-\abs{w} v(x)}_1 \leq  \norm{\ab(pw)-\abs{pw} v(x)}_1+ \norm{\ab(p)-\abs{p} v(x)}_1 \leq 2L.\]
	Thus if we consider two factors $w_1$ and $w_2$ of $u$ of the same length, we deduce:
	\[  \norm{\ab(w_1)-\ab(w_2)}_1 \leq \norm{\ab(w_1)-\abs{w_1}v(x)}_1+\norm{\ab(w_2)-\abs{w_2}v(x)}_1  \leq 4L.\]
	Thus the word $u$ is balanced.

	Now, assume that $u$ is balanced, and let $K$ such that
	for any two factors $v$, $w$ of $u$, $\norm{\ab(v)}_1 = \norm{\ab(w)}_1\implies
	\norm{\ab(v)-\ab(w)}_1 \leq K$.
	Let $p_n$ be the prefix of $u$ of length~$n$.
	For every $k \geq 1$, by cutting $p_n$ into $\floor{\frac{n}{k}}$ parts of length $k$ and a remaining factor of length less than $k$, we get
	\[
		\norm{\ab(p_n) - \floor{\frac{n}{k}} \ab(p_k)}_1 \leq K \floor{\frac{n}{k}} + k.
	\]
	Hence, for every $N \geq n \geq 1$ and every $k \geq 1$, we have
	\begin{align*}
		\norm{ \frac{1}{n} \ab(p_n) - \frac{1}{N} \ab(p_N) }_1
		&\leq K \left( \frac{\floor{\frac{n}{k}}}{n} + \frac{\floor{\frac{N}{k}}}{N} \right) + \frac{k}{n} + \frac{k}{N} + k \abs{ \frac{1}{n} \floor{\frac{\vphantom{N}n}{k}} - \frac{1}{N} \floor{\frac{N}{k}} } \\
		& \leq \frac{2K}{k} + \frac{3k}{n}.
	\end{align*}

	Thus, by taking $k = \floor{\sqrt{n}}$, we see that $(\frac{1}{n} \ab(p_n))_{n \geq 1}$ is a Cauchy sequence, so it converges to some vector $v \in \RR_+^{d+1}$ with $\norm{v}_1 = 1$.

	Now, for every $n \in \NN$ we have
	\[
		\norm{2\ab(p_n) - \ab(p_{2n})}_1 = \norm{\ab(p_n) - \ab(q_n)}_1 \leq K,
	\]
	where $q_n$ is such that $p_{2n} = p_n q_n$.
	So, for all $n \in \NN$, we have
	\[
		\norm{\pi_v(\ab(p_n))}_1 = \norm{\ab(p_n) - n v}_1 \leq \sum_{k=0}^{\infty} \frac{1}{2^{k+1}} \norm{2\ab(p_{n 2^k}) - \ab(p_{n 2^{k+1}})}_1
								\leq \sum_{k=0}^\infty \frac{K}{2^{k+1}} = K,
	\]
	since $\lim_{k \to \infty} \frac{1}{2^k} \ab(p_{n 2^k}) = nv$.
	Hence, we have $\pi_v (W(u)) \subseteq B(0,K)$.
\end{proof}

\subsection{Worms and Dumont-Thomas numeration}

In all the following we consider $S \subseteq \hom(A^+,A^+)$ a finite set of unimodular substitutions on the alphabet $A$. 
We give a definition of the Dumont-Thomas numeration, which is a generalization, for a finite set $S$ of substitutions, of the one given for a single substitution in \cite{Dumont-Thomas}.

\begin{definition}\label{def-dt}
	The Dumont-Thomas alphabet associated to $S$ is defined as
	\[
		\Sigma=\{\ab(p) \mid \exists \sigma \in S,\ \exists a,b\in A,\ pb\text{ prefix of } \sigma(a)\} \subseteq \ZZ^{d+1}.
	\]\nomenclature[G]{$\Sigma$}{Dumont-Thomas alphabet}
	Remark that it is a finite set, since $S$ and $A$ are finite.
\end{definition}

And we introduce an automaton:
\begin{definition} \label{def:automaton}
	We call \emph{abelianized prefix automaton} of the set of substitutions $S$, the automaton $\A$\nomenclature[L$A$]{$\A$}{abelianized prefix automaton} defined by
	\begin{itemize}
		\item alphabet $\Sigma \times S$,
		\item set of states $A$,
		\item transition $a \xrightarrow{t, \sigma} b$, with $(a, t, \sigma, b) \in A \times \Sigma \times S \times A$ if, and only if, there exist $u,v \in A^*$ such that $\sigma(a) = ubv$, with $\ab(u) = t$.
	\end{itemize}
	We denotes by $a \xrightarrow{t_n, s_n} ... \xrightarrow{t_0, s_0} b \in \A$ if we have a path in the automaton: there exist states $a=a_{n+1}, a_n, ..., a_1, a_{0} = b$ such that for all $0 \leq k \leq n$, $a_{k+1} \xrightarrow{t_k, s_k} a_{k}$ is a transition in the automaton.
\end{definition}

\begin{example} \label{ex:dt}
	Let $S = \{\tau_0, \tau_1\}$, with $\tau_0 = \left\{\begin{array}{ccl} 0 &\mapsto& 0 \\ 1 &\mapsto& 01 \end{array} \right.,
		\tau_1 = \left\{\begin{array}{ccl} 0 &\mapsto& 10 \\ 1 &\mapsto& 1 \end{array} \right.$.
	Then, the Dumont-Thomas alphabet is $\Sigma = \{0, e_0, e_1\}$, and the abelianized prefix automaton $\A$ is depicted in Figure~\ref{fig:apa:ex}.
	
	For every word $u \in \{0,1\}^\NN$ we have the relations
	\begin{align*}
		W_0(\tau_0(u)) &= M_0 W_0(u) \sqcup M_0 W_1(u) \\
		W_1(\tau_0(u)) &= M_0 W_1(u) + e_0 \\
		W_0(\tau_1(u)) &= M_1 W_0(u) + e_1 \\
		W_1(\tau_1(u)) &= M_1 W_0(u) \sqcup M_1 W_1(u),
	\end{align*}
	where $e_0 = (1,0)$, $e_1 = (0,1)$, $M_0 = \ab(\tau_0)$ and $M_1 = \ab(\tau_1)$.
\end{example}

\begin{figure}[ht]
    \centering
	\begin{tikzpicture}[shorten >=1pt,node distance=2cm,auto] 
	   \node[state] (q_1)   {$1$}; 
	   \node[state] (q_0) [left=of q_1] {$0$};
		\path[->]
		(q_0) edge  [loop above] node {$0, \tau_0$} (q_0)
		      edge  [loop below] node {$e_1, \tau_1$} (q_0)
		      edge  [bend left] node {$0, \tau_1$} (q_1)
		(q_1) edge  [loop above] node {$e_0, \tau_0$} (q_1)
		      edge  [loop below] node {$0, \tau_1$} (q_1)
		      edge  [bend left] node {$0, \tau_0$} (q_0);
	\end{tikzpicture}
	\caption{Abelianized prefix automaton $\A$ for the set of substitutions $S$ of Example~\ref{ex:dt}}
	\label{fig:apa:ex}
\end{figure}
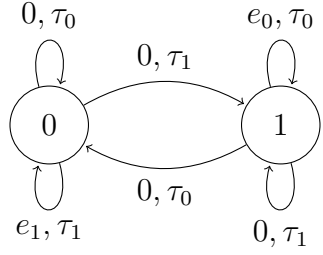

The automaton $\A$ is depicted in Figure~\ref{fig:apa:cassaigne} for the set of
Cassaigne substitutions, and in Figure~\ref{fig:apa:arnoux:rauzy} for the set of
Arnoux-Rauzy substitutions.

Remark that for every $u \in A^\NN$, $\sigma \in S$ and $a \in A$, we have the following relation
\begin{equation}\label{automate-worm}
    W_a(\sigma(u)) = \bigcup_{b \xrightarrow{t, \sigma} a} \ab(\sigma)W_b(u)  + t.
\end{equation}

If we iterate Equation~(\ref{automate-worm}), we get

\begin{lemma}[Dumont-Thomas numeration] \label{dumont:thomas}
	Let $s \in S^\NN$ be a directive sequence and consider a fixed point $u \in {(A^\NN)}^\NN$ of $s$.
	Let $b_n$ be the first letter of the word $u_n$.
	We assume that $\abs{s_{[0,n)}(b_n)} \xrightarrow[{n \to \infty}]{} \infty$.
	Then, for every $a \in A$ we have
	\[
		W_a(u_0) = \bigcup_{n \in \NN} \{ \sum_{k=0}^n M_{[0,k)} t_k \mid b_{n+1} \xrightarrow{t_n, s_n} ... \xrightarrow{t_0, s_0} a \}.
	\]
\end{lemma}

\begin{remark}
	In the following we will use the fact that for every $a,b \in A$, in the automaton, the number of paths
	$b \xrightarrow{t_{l-1}, s_{l-1}}\dots\xrightarrow{t_k, s_k} a$
	is equal to $(M_{[k,l)})_{a,b}$.
\end{remark}

\subsection{Rauzy fractals}\label{sec:Rauzy}

We recall the following result, see \cite[Theorem 5.7]{Bert.Delec.14}.

\begin{proposition}\label{prop-uniquely-ergodic}
Let $s \in S^\NN$ be a directive sequence and let $u$ be a fixed point of $s$.
	Assume that $\bigcap_{n \in \NN} M_{[0,n)}(s) \P\RR_+^{d} = \{x\}$ and that $s$ is everywhere growing (see Definition~\ref{def:everywhere:growing}).
	Then the subshift $\Omega_s$ is uniquely ergodic, and for every word $w \in \Omega_s$, we have $\freq(w) = v(x)$.
	In particular, if $u \in {(A^\NN)}^\NN$ is a fixed point of $s$, then $\freq(u_0) = v(x)$.
\end{proposition}

Since the matrix $M_k(s)$ is invertible, remark that $\bigcap_{n \in \NN} M_{[0,n)}(s) \RR_+^{d+1}$ is a line if, and only if,
$\bigcap_{n \in \NN} M_{[k,k+n)}(s) \RR_+^{d+1}$
is a line for every $k\in\NN$.
When it is the case, we denote by $v^{(k)} \in \RR^{d+1}$ the vector such that $\norm{v^{(k)}}_1=1$ and $\RR_+ . v^{(k)} = \bigcap_{i\geq k} M_{[k,i)}(s) \RR_+^{d+1}$.

\begin{definition}\label{def-Rauzy}
	Let $u \in A^\NN$ be an infinite word admitting a frequency vector $v = \freq(u)$.
	We define $R(u)$\nomenclature[L]{$R(u), R(s)$}{Rauzy fractal} as the closure of $\pi_v W(u)\subseteq P$.
    For a letter $a \in A$, we also define $R_a(u)$\nomenclature[L]{$R_a(u)$}{atom of the Rauzy fractal} as the closure of $\pi_v W_a(u)$.
	The set $R(u)$ is called Rauzy fractal. 
    It is a generalization of the classical notion, for a fixed point of a substitution, see \cite{Rau.82,Pyth.02} for references.
\end{definition}

\begin{figure}[h]
	\centering
	\begin{tikzpicture}[scale=.7, xscale=1, yscale=1]
		\draw [gray!70] (0,0) grid (9.3,6.3);
		\draw[->, thick] (0,0) -- (10,0);
		\draw[->, thick] (0,0) -- (0,7);
		\def\d{.125}
		\def\dd{.0884}

        \draw (2,-2) -- (-3,3) node[above right] {$P$};
		\draw[->, thick] (2.5,2.5) -- node[above left] {$\pi_v$} (.4,1.1);

        \foreach \y in {1,2,3,4,5,6}{
                \draw (.1,\y) -- (-.1,\y);
                \draw[left] (0,\y) node {$\y$};
        }
		\foreach \x in {1,2,3,4,5,6,7,8,9} 
        {
		\draw (\x, .1) -- (\x, -.1);
                \draw[below] (\x,0) node {$\x$};
        };
  
        \foreach \x/\y\X\Y\c in {1/0/.4/ -.4/1, 1/1/-.2/.2/0, 2/1/.2/-.2/0, 3/2/0/0/0, 3/1/.6/-.6/1, 4/2/.4/-.4/1} {
                \draw[dashed, gray] (\x,\y) -- (\X, \Y);
                \ifthenelse{\c = 0}
                {   
                        \fill[blue] (\X-\dd, \Y-\dd) -- (\X-\dd, \Y+\dd) -- (\X+\dd, \Y+\dd) -- (\X+\dd, \Y-\dd)
                }{  
                        \fill[red] (\X, \Y-\d) -- (\X+\d,\Y) -- (\X,\Y+\d) -- (\X-\d,\Y)
                };
        };
		\foreach \x/\y/\c in {0/0/0, 1/0/1, 1/1/0, 2/1/0, 3/1/1, 3/2/0, 4/2/1, 4/3/0, 5/3/0, 6/3/1, 6/4/0, 7/4/1, 7/5/0, 8/5/0, 9/5/1, 9/6/0}
		{
			\ifthenelse{\c = 0}
			{
				\fill[blue] (\x-\dd, \y-\dd) -- (\x-\dd, \y+\dd) -- (\x+\dd, \y+\dd) -- (\x+\dd, \y-\dd)
			}{
				\fill[red] (\x, \y-\d) -- (\x+\d,\y) -- (\x,\y+\d) -- (\x-\d,\y)
			};
		}
	\draw (6,2) node[above right] {$W(u) = \textcolor{blue}{W_0(u)} \sqcup \textcolor{red}{W_1(u)}$};
        \draw (.2,-.2) node[below left] {$R(u) = \textcolor{blue}{R_0(u)} \cup \textcolor{red}{R_1(u)}$};
        \end{tikzpicture}
        \caption{The Rauzy fractal $R(u)$ as the closure of the projection of the worm $W(u)$ on the hyperplane $P$. Example for $u = (01001)^\omega$, so $v = (3/5, 2/5)$.} \label{fig:ex:rauzy}
\end{figure}
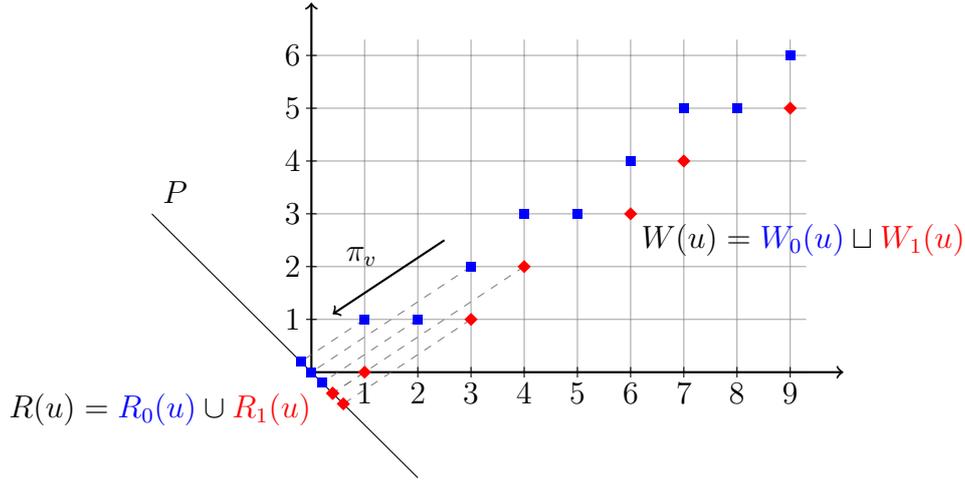

\begin{example}
	For $u = (01001)^\omega$, we have $\freq(u) = (3/5, 2/5)$.
	So we can define the Rauzy fractal by projecting on the hyperplane (i.e., line) $x+y = 0$, and we get a Rauzy fractal with only $5$ points.
	See Figure~\ref{fig:ex:rauzy}.
\end{example}

Examples of non-substitutive Rauzy fractals are drawn in Figure~\ref{fig:ex:fractals} and in Figure~\ref{fig:ex:Rauzy}.

\begin{figure}[!h]
    \centering
    \begin{tikzpicture}[scale=3]
    	\def\decx{0.334167480468748cm}
    	\def\decy{0.0228622267305872cm}
		\node (pic) at (\decx,\decy) {\includegraphics[width=5.17071473121642cm]{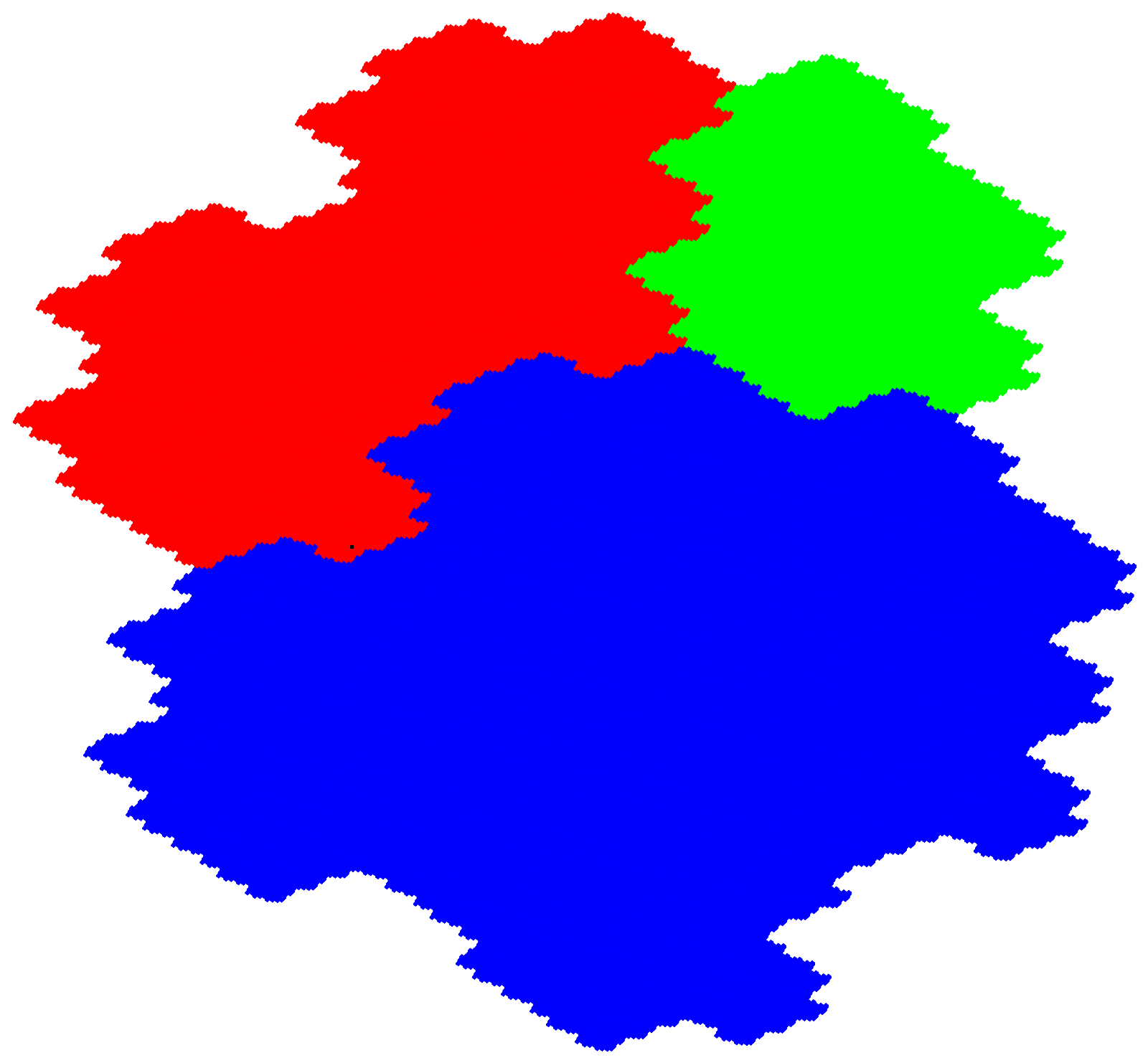}};
		\draw[->, thick] (0,0) -- (1.414213562373095cm, 0) node[below] {$e_0-e_1$};
		\draw[->, thick] (0,0) -- (-0.7071067811865475cm, 1.2247448713915894cm) node[above] {$e_2-e_0$};
		\draw[->, thick] (0,0) -- (0.7071067811865475cm, 1.2247448713915894cm) node[above] {$e_2-e_1$};
		\node[below] at (0,0) {$0$};
	\end{tikzpicture}
	\caption{Approximation of the Rauzy fractal of a directive sequence beginning with\hfill\hbox{} $c_1c_1c_0c_1c_0c_0c_0c_1c_1
c_1c_0c_1c_0c_0c_1c_0c_1c_0c_0c_0c_0c_1c_0c_1c_1c_0c_0
c_1c_1c_1c_0c_1c_1c_1c_0c_0c_0c_1c_0c_0c_0c_1c_1c_0c_0$,\hfill\hbox{}
where $c_0$ and $c_1$ are defined in Section~\ref{sec-cassaigne}.\hfill\break
Here, $v = (0.279291082100669\dots, 0.1294709739854265\dots, 0.5912379439139045\dots)$.} \label{fig:ex:Rauzy}
\end{figure}
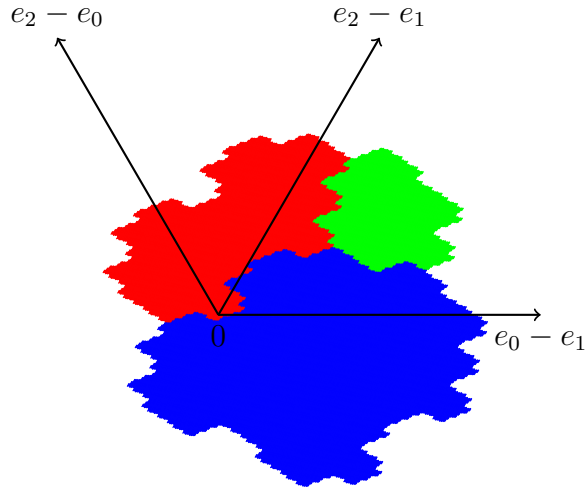

\begin{remark} \label{rk:def:rauzy}
	If $s \in S^\NN$ is a directive sequence such that $\bigcap_{n \in \NN} M_{[0,n)}(s) \P\RR_+^d = \{x\}$, with $x$ a totally irrational direction,
	then, by Lemma~\ref{lem:primitivity} and by Proposition~\ref{prop-uniquely-ergodic}, for every fixed point $u \in {(A^\NN)}^\NN$ of $s$, the infinite word $u_0$ admits $v(x)$ as frequency vector,
	hence we can define the Rauzy fractal $R(u_0)$.

	For a given directive sequence $s$, we do not have uniqueness of fixed point $u$ in general.
	But under some assumptions it defines a unique Rauzy fractal (see Proposition~\ref{prop:R:uniqueness} below).
	When it is the case, we denote the Rauzy fractal by $R(s)$.
\end{remark}

Using Rauzy fractals, we can give a characterization of the interior of $W_a(u)$ for the topology $\topo{x}$, with the following lemma.

\begin{lemma} \label{lem:car:interior}
	For every open subset $B$ of the plane $P$,
	for every totally irrational direction $x\in\P\RR_+^d$,
	for every infinite word $u \in A^\NN$,
	and for every letter $a \in A$
	we have the equivalence between
	\begin{enumerate}
		\item $\HH \cap \pi_x^{-1}(B) \subseteq W_a(u)$,
		\item $\forall b \in A \setminus \{a\},\ \forall t \in \Lambda \setminus \{0\},\ B\cap R_b = \emptyset = B\cap (R+t)$,
	\end{enumerate}
	where $R_b$ is the closure of $\pi_x W_b(u)$ and $R$ is the closure of $\pi_x W(u)$.
	
	In particular, $p \in \HH$ is in the interior of $W_a(u)$ for the topology $\topo{x}$ if, and only if,
	\[
		\pi_x(p) \not\in \bigcup_{a \in A \setminus \{a\}} R_a \cup \bigcup_{t \in \Lambda \setminus \{0\}} R + t.
	\]
\end{lemma}

\begin{proof}
	We have
	\[
		W_a(u) = \HH \setminus \left( \bigcup_{b \in A\setminus \{a\}} W_b(u) \cup \bigcup_{t \in \Lambda \setminus \{0\}} W(u) + t \right),
	\]
	$\pi_x$ is injective on $\HH$, and $\pi_x(\HH)$ is dense in $P$, so we have the equivalences
	\begin{eqnarray*}
		\HH \cap \pi_x^{-1}(B) \subseteq W_a(u) &\Longleftrightarrow& B \cap \pi_x(\HH) \subseteq \pi_x \left( W_a(u) \right) \\
												&\Longleftrightarrow& B \cap \pi_x \left( \bigcup_{b \in A\setminus \{a\}} W_b(u) \cup \bigcup_{t \in \Lambda \setminus \{0\}} W(u) + t \right) = \emptyset\\
												&\Longleftrightarrow& B \cap \left( \bigcup_{b \in A\setminus \{a\}} R_b \cup \bigcup_{t \in \Lambda \setminus \{0\}} R + t \right) = \emptyset.
	\end{eqnarray*}
\end{proof}

The following proposition allows to show that the Rauzy fractal does not depend on the choice of a fixed point $u$ of $s$, and it gives a useful characterization,
with left-infinite paths in the abelianized prefix automaton.

\begin{proposition} \label{prop:R:uniqueness}
	Let $s \in S^\NN$. We assume that we have
	\begin{itemize}
		\item primitivity: $\forall n \in \NN, \exists k \geq n, M_{[n,k)} > 0$,

		\item strong convergence: $\sum_n \matrixnorm{\pi_v M_{[0,n)}}_1$ converges, for some vector $v$.
	\end{itemize}
	
	Then, for every letter $a \in A$ and every fixed point $u \in {(A^\NN)}^\NN$ of $s$, we have
	\[
		R_a(u_0) = \{ \sum_{n=0}^\infty \pi_v(M_{[0,n)} t_n) \mid ... \xrightarrow{t_n, s_n} ... \xrightarrow{t_0, s_0} a \in \A \}.
	\]
	
	In particular, the Rauzy fractal does not depend on the choice of the fixed point $u$ and is compact.
\end{proposition}

\begin{proof}
	Let $u$ be a fixed point of $s$.
	We denote by $b_n$ the first letter of the word $u_n$.
	By Lemma~\ref{dumont:thomas}, for every letter $a \in A$, we have the equality
	\[
		W_a(u_0) = \bigcup_{n \in \NN} \{ \sum_{k=0}^n M_{[0,k)} t_k \mid b_{n+1} \xrightarrow{t_n, s_n} ... \xrightarrow{t_0, s_0} a \}. 
	\]
	Hence, we have
	\[
		R_a = \overline{\bigcup_{n \in \NN} \{ \sum_{k=0}^n \pi_v(M_{[0,k)} t_k) \mid b_{n+1} \xrightarrow{t_n, s_n} ... \xrightarrow{t_0, s_0} a \}}.
	\]
	
	Let us show one inclusion. Let $... \xrightarrow{t_n, s_n} ... \xrightarrow{t_0, s_0} a$ be a left-infinite path in the automaton $\A$.
	Let $\epsilon > 0$. By the strong convergence hypothesis, and using that the Dumont-Thomas alphabet $\Sigma$ is finite, there exists $n \in \NN$ such that
	\[
		\max_{t \in \Sigma - \Sigma} \norm{t}_1 \sum_{k=n+1}^\infty \matrixnorm{\pi_v M_{[0,k)}}_1 \leq \epsilon.
	\]
	Using primitivity, there exists a path $b_{N+1} \xrightarrow{t_N', s_N} ... \xrightarrow{t_0', s_0} a \in \A$, with $t_k' = t_k$ for every $k \leq n$.
	We have
	\[
		\norm{ \sum_{k=0}^\infty \pi_v(M_{[0,k)} t_k) - \sum_{k=0}^N \pi_v(M_{[0,k)} t_k') }_1
		\leq \max_{t \in \Sigma - \Sigma} \norm{t}_1 \sum_{k=n+1}^\infty \matrixnorm{\pi_v M_{[0,k)}}_1 \leq \epsilon.
	\]
	Since $\sum_{k=0}^N \pi_v(M_{[0,k)} t_k') \in R_a$ and since $R_a$ is closed, we deduce the inclusion 
	\[
		\{ \sum_{n=0}^\infty \pi_v(M_{[0,n)} t_n) \mid ... \xrightarrow{t_n, s_n} ... \xrightarrow{t_0, s_0} a \in \A \} \subseteq R_a.
	\]
	
	Let us show the other inclusion.
	We have the inclusion
	\[
		\bigcup_{n \in \NN} \{ \sum_{k=0}^n \pi_v(M_{[0,k)} t_k) \mid b_{n+1} \xrightarrow{t_n, s_n} ... \xrightarrow{t_0, s_0} a \}
		\subseteq \{ \sum_{n=0}^\infty \pi_v(M_{[0,n)} t_n) \mid ... \xrightarrow{t_n, s_n} ... \xrightarrow{t_0, s_0} a \in \A \}
	\]
	because for every $n \in \NN$, there exists a left-infinite path labeled by zeroes going to $b_n$ since for every $k \geq n$, $u_n = s_{[n,k)}(u_k)$.
	To end the proof, it remains to show that the set $\{ \sum_{n=0}^\infty \pi_v(M_{[0,n)} t_n) \mid ... \xrightarrow{t_n, s_n} ... \xrightarrow{t_0, s_0} a \in \A \}$ is compact.
	We define a natural distance on the set of left-infinite paths in the automaton $\A$ by taking a distance $2^{-n}$ between two paths that coincide for the last $n$ transitions.
	This distance makes the set of left-infinite paths compact, and the map sending a left-infinite path $... \xrightarrow{t_n, s_n} ... \xrightarrow{t_0, s_0} a$ to the corresponding sum $\sum_{n=0}^\infty \pi_v(M_{[0,n)} t_n)$ is continuous. So we get the compactness.
\end{proof}

The primitivity hypothesis of Proposition~\ref{prop:R:uniqueness} can be replaced by the hypothesis that $v$ has a totally irrational direction.
Indeed, we have the following lemma and remark.

\begin{lemma} \label{lem:primitivity}
	Let $v \in \RR_+^{d+1}$ having a totally irrational direction such that
	\[
		\bigcap_{n \in \NN} M_{[0,n)} \RR_+^{d+1} = \RR_+ v,
	\]
	then we have primitivity:
	\[
		\forall k \in \NN,\ \exists n \geq k,\ M_{[k,n)} > 0.
	\]
\end{lemma}

\begin{proof}
	For all $k \in \NN$, we have
	\[
		\bigcap_{n \in \NN} M_{[k,n)} \RR_+^{d+1} = \RR_+ M_{[0,k)}^{-1} v,
	\]
	and $M_{[0,k)}^{-1} v$ has a totally irrational direction,
	so it is enough to prove the result for $k = 0$.
	We have $\bigcap_{n \in \NN} M_{[0,n)} \RR_+^{d+1} = \RR_+ v$,
	with $v$ having a totally irrational direction, so there exists $n \in \NN$ such that
	$M_{[0,n)} \RR_+^{d+1}$ does not meet the boundary
	$\bigcup_{i = 0}^d \RR_+^i \times \{0\} \times \RR_+^{d-i}$,
	and this is equivalent to $M_{[0,n)} > 0$.
\end{proof}

\begin{remark} \label{rem:cv:cv}
	Let $s \in S^\NN$ be a directive sequence, and $v \in \RR_+^{d+1}$.
	If $\sum \matrixnorm{\pi_v M_{[0,n)}}_1$ converges,
	then we have $\bigcap_{n \in \NN} M_{[0,n)} \RR_+^{d+1} = \RR_+ v$.
	Indeed, if we take such vector $v$ with $\norm{v}_1 = 1$, then we have
	\[
		\forall y \in \RR_+^{d+1},\ \norm{M_{[0,n)}y - \norm{M_{[0,n)}y}_1 v}_1 \leq \matrixnorm{\pi_v M_{[0,n)}}_1 \norm{y}_1 \xrightarrow[n \to \infty]{} 0.
	\]
\end{remark}

\begin{corollary} \label{cor:bounded:cover}
	Let $s \in S^\NN$ be a directive sequence such that
	the sum $\sum_n \matrixnorm{\pi_x M_{[0,n)}(s)}_1$ converges
	for a totally irrational direction $x \in \P\RR_+^d$.
	
	Then, for every letter $a \in A$ and every fixed point $u \in {(A^\NN)}^\NN$ of $s$, we have
	\[
		R_a(u_0) = \{ \sum_{n=0}^\infty \pi_x(M_{[0,n)} t_n) \mid ... \xrightarrow{t_n, s_n} ... \xrightarrow{t_0, s_0} a \in \A \}.
	\]
	In particular, we have the properties:
	\begin{itemize}
		\item the Rauzy fractal $R(s)$ and its pieces do not depend on the choice of a fixed point,
		\item $R(s)$ is bounded,
		\item $R(s)$ covers the plane: $\bigcup_{t \in \Lambda} R(s) + t = P$,
		\item $R(s)$ has non-empty interior.
	\end{itemize}
\end{corollary}

\begin{proof}
	Thanks to Remark~\ref{rem:cv:cv} and Lemma~\ref{lem:primitivity}, we can apply Proposition~\ref{prop:R:uniqueness},
	hence we deduce the formula, the fact that the Rauzy fractal and its pieces do not depend on the choice of a fixed point, and the boundedness of $R(s)$.
	Now, for any fixed point $u \in {(A^\NN)}^\NN$ of $s$ we have $W(u_0) \oplus \Lambda = \HH$, and we have that $\pi_x(\HH)$ is dense in $P$ since $x$ is a totally irrational direction.
	Hence, we deduce that the union $\bigcup_{t \in \Lambda} R(s) + t$ is dense in $P$. Since $R(s)$ is bounded, this union is locally finite, thus locally closed.
	Hence we get the wanted covering.
	The last point is a consequence of the Baire category theorem: if the interior of $R(s)$ was empty, then the interior of the countable union $\bigcup_{t \in \Lambda} R(s) + t = P$ would be empty, which is absurd.
\end{proof}

\begin{remark}
	We emphasis the fact the the interior of $R(s)$ does not correspond to the interior of $W(u_0)$ for the topology $\topo{x}$, where $u \in {(A^\NN)}^\NN$ is a fixed point of $s$.
	For a totally irrational direction $x$, if an open set $O$ of $P$ is such that $\pi_x^{-1}(O) \cap \HH \subseteq W_a(u_0)$, then $O$ is included in the interior of $R_a(u_0)$, but the converse is false in general.
\end{remark}

\subsection{Technical lemmas} \label{sec:technical:lemmas}

This subsection contains technical lemmas that we use in our proofs.

The following lemma assumes an exponential convergence that implies the strong convergence of Proposition~\ref{prop:R:uniqueness}.
It says that this exponential convergence is invariant by the shift of the directive sequence.

\begin{lemma} \label{lem:theta2}
	Let $s$ be a directive sequence and let $x \in \P\RR_+^{d}$ be a direction. 
	Then, for every $k \in \NN$, we have the equality
	\[
		\limsup_{n \to \infty} \frac{1}{n} \ln \matrixnorm{\pi_x M_{[0,n)}}_1 = \limsup_{n \to \infty} \frac{1}{n} \ln \matrixnorm{\pi_{x^{(k)}} M_{[k, k+n)}}_1,
	\]
	where $x^{(k)} = M_{[0,k)}^{-1} x$\nomenclature[L]{$x^{(k)}$}{$k$th element of a sequence of directions}.
	In particular, if we have that $\limsup_{n \to \infty} \frac{1}{n} \ln \matrixnorm{\pi_x M_{[0,n)}}_1 < 0$, then we have
	\[
		\forall k \in \NN,\ \exists C > 0,\ \exists n_0 \in \NN,\ \forall n \geq n_0,\ \matrixnorm{\pi_{x^{(k)}} M_{[k,k+n)}}_1 \leq e^{-n C}.
	\]
\end{lemma}

\begin{proof}
	
	Let $N$ be the linear endomorphism of $P$ such that $\pi_x M_{[0,k)} = N \pi_{x^{(k)}}$.
	Remark that $N$ is invertible.
	We have the inequalities
	\[
		\matrixnorm{\pi_x M_{[0,n+k)}}_1 \leq \matrixnorm{N}_1 \matrixnorm{\pi_{x^{(k)}} M_{[k,n+k)}}_1,
	\]
	and
	\[
		\matrixnorm{\pi_{x^{(k)}} M_{[k,k+n)}}_1 \leq \matrixnorm{N^{-1}}_1 \matrixnorm{\pi_{x} M_{[0,n+k)}}_1.
	\]
	So we get the wanted equality
	\[
		\limsup_{n \to \infty} \frac{1}{n} \ln \matrixnorm{\pi_x M_{[0,n)}}_1
		= \limsup_{n \to \infty} \frac{1}{n} \ln \matrixnorm{\pi_x M_{[0,n+k)}}_1
		= \limsup_{n \to \infty} \frac{1}{n} \ln \matrixnorm{\pi_{x^{(k)}} M_{[k, k+n)}}_1.
	\]
	
	We deduce the second part of the lemma by taking
	\[
		C = - \frac{1}{2} \limsup_{n \to \infty} \frac{1}{n} \ln \matrixnorm{\pi_x M_{[k,k+n)}}_1.
	\]
\end{proof}

The remaining lemmas in this subsection are topology exercises and are not specific to our subject.

\begin{lemma} \label{lem:Baire}
    Let $B$, $C$, $D$ be open subsets of $P$.
    If $D \subseteq \overline{B}$ and $D \subseteq \overline{C}$, then $D \subseteq \overline{B \cap C}$.
\end{lemma}

\begin{proof}
Let $x \in D$.
Let $r_0>0$ small enough to have $B(x,r_0)\subseteq D$ (balls are assumed open in this proof). Let $r>0$ such that $r\leq r_0$.
\begin{itemize}
\item
As $D \subseteq \bar B$, we get that $B(x,r)\subseteq D \subseteq \bar B$.
If $B(x,r) \cap  B = \varnothing$, then $x\notin \bar B$ which is absurd.
So there exists $y\in B(x,r)\cap B$, and since these are open sets, there exists $r'>0$ such that
$B(y,r')\subseteq B(x,r)\cap B$.
\item
Also, $B(y,r')\subseteq B(x,r)\subseteq \bar C$ thus there exists $z$ and $t>0$ such that
$B(z,t)\subseteq B(y,r') \cap C$. 
\item
Finally, for all $r$ small enough we have found $z\in B \cap C$ such that $d(x,z)<r$ (since $z\in B(x,r)$).
\end{itemize}
Therefore $x\in \overline{B\cap C}$.
\end{proof}

The next technical lemmas are useful in the proof of Proposition~\ref{prop:mesure:good}.

\begin{lemma} \label{lem:ouvert:meas}
	Let $H$ be a closed subset of a metric space $X$, and let $\mu$ be a finite measure on $X$ such that $\mu(H) = 0$.
	Then for every $\epsilon > 0$ there exists an open subset $O$ such that $\mu(O) \leq \epsilon$ and $H \subseteq O$.
\end{lemma}

\begin{proof}
	For $n\geq1$, let $H_n = \{x \in X \mid d(x,H) < \frac{1}{n}\}$.
	We have $\bigcup_{n \geq 1} X \setminus H_n = X \setminus H$, because $H$ is closed.
	Thus, we have $\lim_{n \to \infty} \mu(X \setminus H_n) = \mu(X \setminus H) = \mu(X)$.
	Let $\epsilon > 0$. There exists $n \geq 1$ such that $\mu(H_n) \leq \epsilon$.
	Then, the open set $O = H_n$ suits.
\end{proof}

\begin{lemma} \label{lem:ouvert:petite:mesure}
	Let $X \subseteq \P\RR_+^d$ and let $\mu$ be a probability measure on $X$.
	Let $N \subseteq X$ be the set of non totally irrational directions of $X$.
	We assume that $\mu(N) = 0$.
	Then, for every $\epsilon > 0$ there exists an open set $O$ of $X$
	such that $O$ contains all the non totally irrational directions
	and such that $\mu(O) \leq \epsilon$.
\end{lemma}

\begin{proof}
	The set $N$ is the union of kernels of linear forms with rational coefficients.
	Thus it is a countable union of closed subsets.
	Let $(N_n)_{n \in \NN}$ be closed subsets such that $N = \bigcup_{n \in \NN} N_n$.
	Let $\epsilon > 0$.
	For every $n \in \NN$, let $O_n$ be an open set given by Lemma~\ref{lem:ouvert:meas} such that $\mu(O_n) \leq \frac{\epsilon}{2^{n+1}}$ and $N_n \subseteq O_n$.
	Then, the open set $O = \bigcup_{n \in \NN} O_n$ satisfies what we want: we have $N \subseteq O$ and
	\[
		\mu(O) \leq \sum_{n \in \NN} \mu(O_n) \leq \sum_{n \in \NN} \frac{\epsilon}{2^{n+1}} = \epsilon.
	\]
\end{proof}

\section{General conditions for the existence of nice Rauzy fractals} \label{sec:gen:cond}

\subsection{Statement}

\begin{definition}\label{defG}
	We say that a directive sequence $s \in S^\NN$ is \emph{good} if
	\begin{enumerate}
		\item $\limsup_{n \to \infty} \frac{1}{n} \ln \matrixnorm{\pi_x M_{[0,n)}}_1 < 0$, for some direction $x \in \P\RR_+^{d}$,

		\item The direction $x$ is totally irrational,

		\item There exists a fixed point $u \in {(A^\NN)}^\NN$ of $s$,
		an increasing sequence of integers $(k_n)_{n \in \NN}$\nomenclature[L$k$]{$(k_n)$}{integer sequence},
		and a positive radius $r > 0$
		such that
		\[
			\forall n \in \NN,\ \forall a \in A,\ \exists p \in P,
			\ \HH \cap \pi_{x^{(k_n)}}^{-1}(B(p,r)) \subseteq W_a(u_{k_n}),
		\]
		where $x^{(k_n)} = M_{[0,k_n)}^{-1} x$,
		
		\item The sequence $x^{(k_n)}$ has a limit 
		which is a totally irrational direction.
	\end{enumerate}
	By Remark~\ref{rem:cv:cv}, the direction $x$ is unique. We call it the \emph{direction of $s$}.

\end{definition}

Remark that for a good directive sequence, the Rauzy fractal does not depend on the choice of a fixed point, is compact and has non-empty interior
by Corollary~\ref{cor:bounded:cover}. We recall Theorem~\ref{thmC} that will be proven in the rest of this section:

\begin{NoHyper}
\setcounter{Theorem}{\thmCnum}
\begin{Theorem}
\thmCbody
\end{Theorem}
\end{NoHyper}

\begin{remark}
	If we consider a directive sequence of the form $\sigma^\omega$, where
	$\sigma$ is an unimodular substitution, then we are back in the classical
	setting of the Rauzy fractal associated with a single substitution.
	In this sense, Theorem~\ref{thmC} gives a generalization
	of~\cite[Theorem~1.3.3]{Aki.Merc.18}.

	The converse of Theorem~\ref{thmC} is true for directive sequences of the form $\sigma^\omega$: if the subshift is conjugate to a translation on a torus, then $\sigma^\omega$ is good,
	where $\sigma$ is an irreducible Pisot unimodular substitution. See Subsection~\ref{sec:pisot} for more details.

	The Pisot substitution conjecture gives that for every irreducible Pisot unimodular substitution $\sigma$, the directive sequence $\sigma^\omega$ is good.
	See Subsection~\ref{sec:pisot} for more details.
\end{remark}

\subsection{Proof of Theorem~\ref{thmC}}

  In all this subsection we assume that $s$ is a good directive sequence, $x^{(k)} \in \P\RR_+^{d}$ is such that
  \[
  	\{x^{(k)}\} = \bigcap_{n \geq k} M_{[k,n)} \P\RR_+^{d},
  \]
  and $u$ is a fixed point of $s$.
  We also denote $x = x^{(0)}$. Remark that for all $k \in \NN$ we have $x^{(k)} = M_{[0,k)}^{-1} x$.
  We denote by $R^{(k)} = \overline{\pi_{x^{(k)}} W(u_k)}$ and $\forall a \in A$, $R_a^{(k)} = \overline{\pi_{x^{(k)}} W_a(u_k)}$ the Rauzy fractal uniquely defined by the good directive sequence $(s_n)_{n \geq k}$.
  
  \subsubsection{Step 1: proof that we have a topological tiling}
  
  \begin{lemma}\label{Wopen}
	For every $k \in \NN$ and every $a \in A$,
          the set $W_a(u_k)$ has non-empty interior for $\topo{x^{(k)}}$.

  \end{lemma}
  
  \begin{proof}
    Consider $k\in\NN$ and $a\in A$. We have by Equation~(\ref{automate-worm})
    
     \[
        W_a(u_k) = \bigcup_{b \xrightarrow{t, s_k} a} M_{k}W_b(u_{k+1})  + t.
     \]
   
Now, if we assume that the interior of $W_b(u_{k+1})$ is not empty
for $\topo{x^{(k+1)}}$, for some $b\in A$ such that $a$ occurs in $s_k(b)$,
then by Lemma~\ref{lem:O:matrix},
the interior of $M_k W_b(u_{k+1})$ is non-empty for $\topo{x^{(k)}}$.
So $W_a(u_k)$ also has non-empty interior.
 
Then we iterate the process: For $l>k$ we have
    \begin{equation}\label{eq-iter-worm}
        W_a(u_k) = \bigcup_{b \xrightarrow{t_{l-1}, s_{l-1}}\dots\xrightarrow{t_k, s_k} a} M_{[k,l)}W_b(u_l) + \sum_{i=k}^{l-1} M_{[k,i)} t_i.
    \end{equation}	
	By hypothesis, for a fixed $k$ we can find $l \geq k$ where the interior of $W_b(u_l)$ is non-empty for all $b\in A$.
	Since $M_{[k,l)}$ is invertible, there exists at least one $b\in A$ such that the union in~(\ref{eq-iter-worm}) is non-empty, and we deduce the result.
  \end{proof}
  
  \begin{lemma} \label{Wdense}
	For every $k \in \NN$ and every $a \in A$, the interior
    of $W_a(u_k)$ is dense in $W_a(u_k)$ for $\topo{x^{(k)}}$.
	
  	Now consider $U \subseteq P$ an open set such that $\pi_{x^{(k)}}^{-1}(U) \cap \HH$ is the interior of $W_a(u_k)$. Then the set $U$ is dense in $R_a^{(k)}$.
  \end{lemma}
  
  \begin{proof}
    Consider $m\in W_a(u_k)$ and $V$ open set containing $m$.
    We want to find an element of $V$ in the interior of $W_a(u_k)$. 
    By Equation~(\ref{eq-iter-worm}), $m$ belongs to a set of the following form for each $l \geq k$:
    \[
        M_{[k,l)}W_b(u_l)+t
    \]
    By Lemma~\ref{lem:worm:balance:uniform} the sets $\pi_{x^{(k_n)}}W_b(u_{k_n})$ are uniformly bounded for $n \in \NN$,
    thus we deduce
    with $\limsup_{n \to \infty} \frac{1}{n} \ln \matrixnorm{\pi_x M_{[0,n)}}_1 < 0$ and Lemma~\ref{lem:theta2}
    that the diameter of
    \[
    	\pi_{x^{(k)}}M_{[k,k_n)}W_b(u_{k_n})
    \]
    is arbitrarily small for $n \in \NN$ large enough, hence
	there exists $n \in \NN$ such that we have the inclusion $M_{[k,k_n)}W_b(u_{k_n})+t\subseteq V$.
	As this set has non-empty interior by Lemma~\ref{lem:O:matrix},
	it follows that $V$ intersects the interior of $W_a(u_k)$.
	This proves that the interior of $W_a(u_k)$ is dense in $W_a(u_k)$.
	
	Now, if $U$ is an open subset of $P$ such that $\pi_{x^{(k)}}^{-1}(U) \cap \HH$ is the interior of $W_a(u_k)$, then
	the projection $U \cap \pi_{x^{(k)}}(\HH)$ is dense in $\pi_{x^{(k)}}(W_a(u_k))$ which is dense in $R_a^{(k)}$.
	Thus $U$ is dense in $R_a^{(k)}$.

  \end{proof}
  
  \begin{lemma}\label{lem-Wintersection}
	  $\;$
	\begin{itemize}
		\item For every $t \in \Lambda \setminus \{0\}$, $R \cap (R+t)$ has empty interior.
		\item For $a \neq b \in A$, $R_a \cap R_b$ has empty interior. 
	\end{itemize}
  \end{lemma}
  
  \begin{proof}
    We denote $\pi = \pi_{x^{(0)}}$.
    By Lemma~\ref{lem:worm:tiling}, we have $W(u_0)\cap (W(u_0)+t)=\emptyset$.
    Now consider $U \subseteq P$ an open set such that $\pi^{-1}(U) \cap \HH$ is the interior of $W(u_0)$.
    Then, we have
    \begin{align*}
    	\pi^{-1}(U) \cap ( \pi^{-1}(U) + t) \cap \HH = \emptyset &\Longrightarrow U \cap (U + \pi(t)) \cap \pi(\HH) = \emptyset \\
    																	&\Longrightarrow U \cap (U + \pi(t)) = \emptyset,
    \end{align*}
    because $\pi(\HH)$ is dense in $P$ since the direction $x^{(0)}$ is totally irrational.
    Moreover, by Lemma~\ref{Wdense}, we have that $U$ is dense in the $R$.
    Then, by Lemma~\ref{lem:Baire}, the empty set $U \cap (U + t)$ is dense in the interior of $R \cap (R + t)$. We deduce that the interior of $R \cap (R + t)$ is empty.
    
    For $a \neq b$, we have $W_{a}(u_0) \cap W_{b}(u_0) = \emptyset$. Let $U_a$ and $U_b$ be open subsets of $P$ such that $W_{a}(u_0) = \pi^{-1}(U_a) \cap \HH$ and $W_b(u_0) = \pi^{-1}(U_b) \cap \HH$.
    Then, by Lemma~\ref{lem:Baire}, the empty set $U_a \cap U_b$ is dense in the interior of $R_a \cap R_b$. We deduce that the interior of $R_a \cap R_b$ is empty.
    
  \end{proof}
  
  \subsubsection{Step 2: proof that the boundary has zero Lebesgue measure}\label{proof-ThC-step2}

For every $k \in \NN$, we denote $v^{(k)} = v(x^{(k)})$\nomenclature[L]{$v^{(k)}$}{vector} the unique vector such that $[v^{(k)}] = x^{(k)}$ and $\norm{v^{(k)}}_1 = 1$.
Remark that the direction $x^{(k)}$ being totally irrational, the numbers $v^{(k)}_a$ can not be equal to zero. Let us then define $g_k =\max_{a\in A}\frac{\lambda( R^{(k)}_a)}{v^{(k)}_a}$\nomenclature[L]{$g_k$}{measure of Rauzy fractal}, $f_k =\max_{a\in A}\frac{\lambda(\partial R^{(k)}_a)}{v^{(k)}_a}$\nomenclature[L]{$f_k$}{measure of the boundary of the Rauzy fractal}. Let $N_k$\nomenclature[L]{$N_k$}{endomorphism of $P$} be the linear endomorphism of $P$ such that $N_k \circ \pi_{x^{(k+1)}} = \pi_{x^{(k)}} \circ M_k$. This map is well-defined since $M_k x^{(k+1)} = x^{(k)}$. Observe that $N_k$ is an invertible map.

  \begin{lemma} \label{detN}
    For all $l > k$, we have
    \[
    	\det(N_{[k,l)}) = \cfrac {\det(M_{[k,l)})} {\norm{M_{[k,l)}v^{(l)}}_1}.
    \]
  \end{lemma}
  
  \begin{proof}
Consider two bases of $\RR^{d+1}$ made by a basis of $P$ and $v^{(l)}$ for one,
and by the same basis of $P$ and $v^{(k)}$ for the second one.
Then we compute the matrix of the linear map $M_{[k,l)}$ in these bases.
To do this we use the definition of $N_{[k,l)}$
and the fact that $M_{[k,l)} v^{(l)} = \norm{M_{[k,l)}v^{(l)}}_1 v^{(k)}$.
Thus we obtain
$$
	\begin{pmatrix}
		\left[N_{[k,l)}\right]	&  0	\\
		*		& \norm{M_{[k,l)}v^{(l)}}_1
	\end{pmatrix}
$$
The matrix of change of basis is
$\begin{pmatrix}
	\Id	&  *	\\
	0	& 1
\end{pmatrix}$.
Then we compute the determinant of the matrix, and obtain the result.
  \end{proof}

  \begin{lemma} \label{ineq_meas}
    For every $k \in \NN$, we have $g_k\leq g_{k+1}$ and $f_k\leq f_{k+1}$.
  \end{lemma}
  \begin{proof}
      For every $k \in \NN$ and $a \in A$, we have the equality
    \[
        R^{(k)}_a = \bigcup_{b \xrightarrow{t, s_k} a} N_k R^{(k+1)}_{b} + \pi_{v^{(k)}} (t).\]
        
 We deduce
  \begin{align*}
  	(\lambda (R_a^{(k)}))_{a \in A} &\leq \left( \sum_{b \xrightarrow{t, s_k} a} \abs{\det N_k} \lambda(R_b^{(k+1)}) \right)_{a \in A} \\
  									&= \abs{\det N_k} M_k(\lambda(R_a^{(k+1)}))_{a \in A} \\
  									&\leq \abs{\det N_k} M_k g_{k+1} v^{(k+1)}.
  \end{align*}
  By Lemma~\ref{detN} and using $\abs{\det M_k}=1$ we have
  \[
  	(\lambda (R_a^{(k)}))_{a \in A} \leq \frac{1}{\norm{M_kv^{(k+1)}}_1}
    g_{k+1} M_k v^{(k+1)} = g_{k+1}v^{(k)},
  \]
  thus the sequence $g_k$ is nondecreasing.
  The proof is similar for the sequence $f_k$.
  \end{proof}
  
  	Let $I = \{k_n \mid n \in \NN\}$, and
    let $a \in A$. For every $b \in A$, $k \in I$ and $l \in I$, let
    \[
        L_b^{k,l} = \{\pi_{v^{(k)}} \left( \sum_{i=k}^{l-1} M_{[k,i)} t_i \right) \in P \mid b \xrightarrow{t_{l-1}, s_{l-1}} ... \xrightarrow{t_k, s_k} a\}
    \]
    \[
        P_b^{k,l} = \{t \in L_b^{k,l} \mid N_{[k,l)} R^{(l)}_{b} + t \subseteq B(p_l, r) \},
    \]
    where $r > 0$ and $p_l \in P$ are such that $\HH \cap \pi_{v^{(l)}}^{-1} B(p_l, r) \subseteq W_a(u_l)$.
    It is possible by definition of $k_n$, since the directive sequence is good (see Definition~\ref{defG}).

    \begin{lemma}\label{lem-rapport-L}
    	There exists a uniform constant $C_a > 0$ such that for all $n \in I$, there exists $l_0$ such that for all $l> l_0$, $l \in I$, we have
		\[
		    \sum_{b\in A} v^{(l)}_b\#P_b^{n, l} \geq C_a\sum_{b\in A} v^{(l)}_b\#L_b^{n, l}.
		\]
    \end{lemma}
    
    \begin{proof}
    	Let $n \in I$.
	    Let $l_0 \in \NN$ such that $\forall l \geq l_0$ with $l \in I$, the diameter of $N_{[n,l)} R^{(l)}_{b}$ is less than $r/2$. 
	    It is possible, using Lemma~\ref{lem:theta2}, and because $R^{(l)}$, $l \in I$, are uniformly bounded by Lemma~\ref{lem:worm:balance:uniform}.
	    Then, for every $t \in L_b^{n, l}$, if $N_{[n,l)} R^{(l)}_{b} + t$ meets $B(p_{l}, r/2)$, then it is included in $B(p_{l}, r)$.
	    Thus, we have 
	    
	    \[
	    	g_{l} \sum_{b\in A} v^{(l)}_b\#P_b^{n, l} \geq \sum_{b\in A} \lambda (R_b^{(l)})\#P_b^{n, l} \geq \frac{\lambda(B(p_{l}, r/2))}{\abs{\det(N_{[n,l)})}}.
	    \]
	    
	    Moreover we have $\sum_{b\in A} v^{(l)}_b\#L_b^{n, l} = (M_{[n, l)} v^{(l)})_{a} = \norm{M_{[n, l)} v^{(l)}}_1 v_a^{(n)}$.
	    
	    We deduce by Lemma~\ref{detN}
	    \[
	    	\frac{\sum_{b \in A}v^{(l)}_b\#P_b^{n, l} }{\sum_{b \in A} v^{(l)}_b\# L_b^{n, l}}\geq \frac{1}{g_{l}}\lambda(B(p_{l}, r/2))\frac{1}{v^{(n)}_a}.
	    \]
		Since $1/g_{l}$ and $v_a^{(n)}$ converges to non-zero values, by hypothesis of total irrationality on the limit of the sequence of directions $x^{(n)}$, for $n \in I$, we deduce the result.
    \end{proof}

  \begin{proposition} \label{bmaj}
  	There exists $c<1$ such that, for every $k \in I$, there exists $l > k$ in $I$ such that $f_k \leq c f_l$,
  	where $I = \{k_n \mid n \in \NN\}$.
  \end{proposition}
  
  \begin{proof}
  	Let $k \in I$. For every $l > k$, we have
    \[
        R^{(k)}_a = \bigcup_{b \xrightarrow{t_{l-1}, s_{l-1}} ... \xrightarrow{t_k, s_k} a} N_{[k,l)} R^{(l)}_{b} + \sum_{i=k}^{l-1} \pi_{v^{(k)}}( M_{[k,i)} t_i).
    \]
    
    Hence, we have
    \[
        \partial R^{(k)}_a \subseteq \bigcup_{b \in A} \bigcup_{t \in L_b^{k,l} } N_{[k,l)} \partial R^{(l)}_{b} + t.
    \]
 
 Now if $l \in I$ and $t\in P_b^{k,l}$, then $N_{[k,l)} \partial R^{(l)}_{b} + t$ included in the interior of $R_a^{(k)}$, thus we deduce
 
    \[
        \partial R^{(k)}_a \subseteq \bigcup_{b \in A} \bigcup_{t \in L_b^{k,l} \setminus P_b^{k,l}} N_{[k,l)} \partial R^{(l)}_{b} + t,
    \]
    Using the inclusion
    we deduce
    \begin{align*}
        \lambda( \partial R^{(k)}_a ) &\leq \abs{\det(N_{[k,l)})} \sum_{b \in A} \lambda(\partial R^{(l)}_{b}) \#(L_b^{k,l} \setminus P_b^{k,l}) \\
        							&\leq \abs{\det(N_{[k,l)})} f_l\sum_{b \in A} v^{(l)}_b\# (L_b^{k,l} \setminus P_b^{k,l}) \\
									&\leq \abs{\det(N_{[k,l)})} f_l(1-C_a)\sum_{b \in A} v^{(l)}_b\# L_b^{k,l} \\
									&= (1-C_a)f_l \abs{\det(N_{[k,l)})} \norm{M_{[k,l)} v^{(l)}}_1 v^{(k)}_a \\
									&= (1-C_a) f_l v^{(k)}_a.
	\end{align*}
Thus $f_k\leq (1-C)f_l$, with $C = \min_{a \in A} C_a$.
\end{proof}
  
  \begin{lemma} \label{lem:zero:measure}
  	For every $a \in A$, we have $\lambda(\partial R_a) = 0$.
  \end{lemma}
  	
  \begin{proof}
	  We deduce from Proposition~\ref{bmaj} that there exists $k \in \NN$ such that $f_k = 0$. 
  	Then we have $\lambda(\partial R^{(n)}_a) = 0$ for every $n \leq k$ and every $a \in A$ by Lemma~\ref{ineq_meas}. 
  \end{proof}
  
  \subsubsection{Step 3: proof that the translation is conjugate to the subshift}
  
  We refer to~\cite{Aki.Merc.18}. In the theorem that we recall below, the authors give conditions to prove that the translation by $\pi_x(e_0)$ on the torus $P/\Lambda \simeq \TT^d$ is measurably conjugate to the subshift $\Omega_w$ generated by a word $w \in A^\NN$.
  We check that each condition is satisfied for the word $w = u_0$:
	\begin{itemize}
		\item The boundedness is given by the Corollary~\ref{cor:bounded:cover},
		\item Minimality of the subshift $\Omega_{u_0} = \Omega_{u}$ is given by Proposition~\ref{prop-omega-u-omega-s} since we have primitivity by Lemma~\ref{lem:primitivity},
		\item Boundaries of $R_a$, $a \in A$, have zero Lebesgue measure thanks to Lemma~\ref{lem:zero:measure},
		\item Thanks Lemma~\ref{lem-Wintersection}, we have a topological tiling, and thanks to Lemma~\ref{lem:zero:measure} the boundaries have zero Lebesgue measure.
			Hence we deduce that the union $\bigcup_{t \in \Lambda} R + t = P$ is disjoint in Lebesgue measure.
	\end{itemize}

Thus, we can use the following theorem for $w = u_0$.

\begin{theorem}\cite[Theorem 2.3]{Aki.Merc.18} \label{thm-step3}
	Let $w \in A^\NN$ be an infinite word, and let $x$ be a totally irrational direction.
	Let $R = \overline{\pi_x(W(w))}$ and for all $a \in A$, $R_a = \overline{\pi_x(W_a(w))}$.
	We assume that we have the following:
	\begin{itemize}
		\item the set $\pi_x(W(w))$ is bounded,
		\item the subshift $(\Omega_{w}, T)$ generated by $w$ is minimal,
		\item the boundaries of $R_a$, $a \in A$, have zero Lebesgue measure,
		\item the union $\bigcup_{t \in \Lambda} R + t = P$ is disjoint in Lebesgue measure.
	\end{itemize}
	
	Then there exists a Borel $T$-invariant measure $\mu$ such that the subshift $(\Omega_w, T, \mu)$
	is measurably conjugate
	to the translation on the torus $(P/\Lambda, T_{\pi_x(e_0)}, \lambda)$.
\end{theorem}

    The idea to prove this theorem is to show that the natural conjugacy that we have between the shift map on the orbit $\orbit{w} = \{T^n w \mid n \in \NN \}$ 
    and the translation by $\pi_x(e_0)$ on the quotient $\HH/\Lambda$, gives a measurable conjugacy after taking the closure. See Figure~\ref{fig_idea}.
	Hence, we have:
	\begin{remark} \label{rem:coding}
		The symbolic coding coming from the partition of the Rauzy fractal into pieces $R_a(s)$, $a \in A$ is a measurable conjugacy between the translation by $\pi_x(e_0)$ on the torus $P/\Lambda$ and the subshift $\Omega_s = \Omega_u$.
		In particular, it gives a generating partition.
	\end{remark}
	
	\begin{figure}[ht]

		\[
			\xymatrix
			{
                    \orbit{w} \ar[d] \ar[r]^T & \orbit{w} \ar[d] & &
					\Omega_w \ar[d] \ar[r]^T & \Omega_w \ar[d] \\
				W(w) \ar[d] \ar[r]^E & W(w) \ar[d] & \leadsto &
					R \ar[d] \ar[r]^E & R \ar[d] \\
				\HH / \Lambda \ar[r]^{T_{\pi_x(e_0)}} & \HH / \Lambda & &
					P / \Lambda \ar[r]^{T_{\pi_x(e_0)}} & P / \Lambda \\
			}
		\]
		\caption{Commutative diagrams of the conjugacy between the shift $T$, the domain exchange $E$ and the translation on the quotient $T_{\pi_x(e_0)}$, before and after taking the closure} \label{fig_idea}
	\end{figure}
	
	Figure~\ref{fig:ex:tiling} shows the tiling of $P$ by the Rauzy fractal of Figure~\ref{fig:ex:Rauzy} for the lattice $\Lambda$. The vector $\pi_x(e_0)$ giving the translation in the quotient $P/\Lambda$ is also depicted.
	
\begin{figure}[ht]
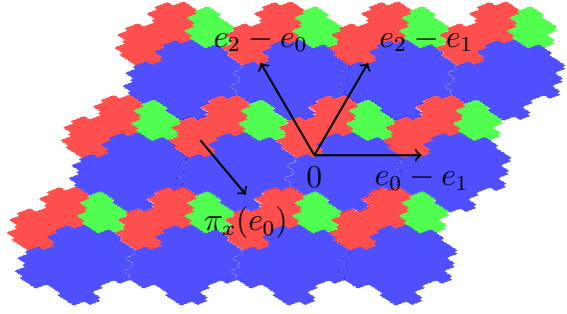

    \centering
    \begin{tikzpicture}[scale=1]
    	\def\decx{0.334167480468748cm}
    	\def\decy{0.0228622267305872cm}
    	\def\w{1.72357157707214cm}
    	\def\transp{.7}
		\node[opacity=\transp] (pic) at (\decx,\decy) {\includegraphics[width=\w]{Cadic.png}};
		\node[opacity=\transp] (pic) at (\decx+1.414213562373095cm,\decy) {\includegraphics[width=\w]{Cadic.png}};
		\node[opacity=\transp] (pic) at (\decx-1.414213562373095cm,\decy) {\includegraphics[width=\w]{Cadic.png}};
		\node[opacity=\transp] (pic) at (\decx-2*1.414213562373095cm,\decy) {\includegraphics[width=\w]{Cadic.png}};
		\node[opacity=\transp] (pic) at (\decx-0.7071067811865475cm,\decy+1.2247448713915894cm) {\includegraphics[width=\w]{Cadic.png}};
		\node[opacity=\transp] (pic) at (\decx+0.7071067811865475cm,\decy+1.2247448713915894cm) {\includegraphics[width=\w]{Cadic.png}};
		\node[opacity=\transp] (pic) at (\decx+0.7071067811865475cm+1.414213562373095cm,\decy+1.2247448713915894cm) {\includegraphics[width=\w]{Cadic.png}};
		\node[opacity=\transp] (pic) at (\decx-0.7071067811865475cm-1.414213562373095cm,\decy+1.2247448713915894cm) {\includegraphics[width=\w]{Cadic.png}};
		\node[opacity=\transp] (pic) at (\decx-0.7071067811865475cm,\decy-1.2247448713915894cm) {\includegraphics[width=\w]{Cadic.png}};
		\node[opacity=\transp] (pic) at (\decx+0.7071067811865475cm,\decy-1.2247448713915894cm) {\includegraphics[width=\w]{Cadic.png}};
		\node[opacity=\transp] (pic) at (\decx-0.7071067811865475cm-1.414213562373095cm,\decy-1.2247448713915894cm) {\includegraphics[width=\w]{Cadic.png}};
		\node[opacity=\transp] (pic) at (\decx-0.7071067811865475cm-2*1.414213562373095cm,\decy-1.2247448713915894cm) {\includegraphics[width=\w]{Cadic.png}};
		\draw[->, thick] (0,0) -- (1.414213562373095cm, 0) node[below] {$e_0-e_1$};
		\draw[->, thick] (0,0) -- (-0.7071067811865475cm, 1.2247448713915894cm) node[above] {$e_2-e_0$};
		\draw[->, thick] (0,0) -- (0.7071067811865475cm, 1.2247448713915894cm) node[above right] {$e_2-e_1$};
		\draw[->, thick] (-1.5cm,.2cm) -- (-1.5cm+0.6011679667801577cm, .2cm-0.7241156395806627cm) node[below] {$\pi_x(e_0)$};
		\node[below] at (0,0) {$0$};
	\end{tikzpicture}
	\caption{Tiling of $P$ by a Rauzy fractal, giving the translation by $\pi_x(e_0)$ on the torus $P/\Lambda$} \label{fig:ex:tiling}
\end{figure}

\subsubsection{Conclusion} 
Now, we prove Theorem~\ref{thmC}.
Starting from a good directive sequence, there exists a Rauzy fractal $R$ by Remark~\ref{rk:def:rauzy}.
By Lemma~\ref{lem-Wintersection} we know that $R$ and $R+t$, $t \in \Lambda \setminus \{0\}$ have intersection of empty interior. Since $\bigcup_{t\in\Lambda} R+t = P$, by Corollary~\ref{cor:bounded:cover}, we deduce that $R$ defines a topological tiling of the torus $P/\Lambda \simeq \TT^d$.
By step 2 we know that the boundaries of $R$ and of the pieces $R_a$ have zero Lebesgue measures.
Thus we deduce that up to a set of zero Lebesgue measure, $R$ is a measurable fundamental domain of $P$ for the action of $\Lambda$.
By Lemma~\ref{lem-Wintersection} we know that the interior of the intersection of two pieces is empty.
So such intersection $R_a \cap R_b$ is included in the boundary of $R_a$ and it has zero Lebesgue measure.
Thus, we get that the union $R(s) = \bigcup_{a \in A} R_a(s)$ is disjoint in Lebesgue measure.
By Lemma~\ref{Wdense}, for all $a \in A$, the piece $R_a$ is the closure of an open set, thus it is the closure of its interior.
By step 3, Theorem~\ref{thm-step3} and Remark~\ref{rem:coding} give the expected conjugacy.

Now, we prove that for every $a \in A$, $R_a$ is a bounded remainder set for the translation by $\pi_x(e_0)$ on the torus $P/\Lambda$.
The Rauzy fractal being bounded, there exists a constant $K$ such that $R-R \subseteq B(0,K)$, where $R-R = \{p-q \mid p,q \in R\}$ is the set of differences.
Let $u$ be a fixed point of $s$.
By the previous conjugacy, for $\lambda$-almost every $z \in P/\Lambda$, there exists an infinite word $w \in \Omega_u$ such that
$w$ is the coding of the orbit of $z$ by the translation for the measurable partition $(R_a)_{a \in A}$ of the torus $P/\Lambda$.
Then, for every $a \in R_a$ we have the equality
\[
	\sum_{n=0}^{N-1} \indic_{R_a}(T_{\pi_x(e_0)}^n(z)) = \abs{p_N}_a,
\]
where $p_N$ is the prefix of length $N$ of the word $w$.
Since $w \in \Omega_u$, for every $N \in \NN$ the word $p_N$ is a factor of $u_0$.
Thus
\[
	\ab(p_N) - Nv(x) = \pi_x(\ab(p_N)) \in \pi_x(W(u_0) - W(u_0)) \subseteq R-R \subseteq B(0,K).
\]
Hence, for $\lambda$-almost every $z \in P/\Lambda$ and for every $a \in A$ we get the inequality
\[
	\abs{\sum_{n=0}^{N-1} \indic_{R_a}(T_{\pi_x(e_0)}^n(z)) - N v(x)_a} \leq K.
\]
Now, by Birkhoff ergodic theorem, we get that for every $a \in A$, $v(x)_a = \frac{\lambda(R_a)}{\lambda(R)}$, so $R_a$ is a bounded remainder set.

Finally note that, by construction, $(R_a)_{a\in A}$ is a liftable partition of
the torus.
Altogether, we get that $(R_a)_{a\in A}$ is a nice generating partition.

\section{Dynamics of continued fractions}\label{sec:cont-frac}

\subsection{Extended continued fraction algorithms}

\begin{definition} \label{def:continued:fraction}
	An \emph{extended continued fraction algorithm}, denoted $(X,s_0)$, is the data of
	\begin{itemize}
		\item a subset $X \subseteq \P\RR_+^d$,
		\item a finite alphabet $A =\{0,\dots,d\}$, with $d \geq 1$,
		\item a finite set $S \subseteq \hom(A^+,A^+)$ of unimodular substitutions on the alphabet $A$,
		\item a map $s_0\colon X \to S$\nomenclature[L]{$s_0$}{extended continued fraction algorithm}
	such that for all $x\in X$, $\ab(s_0(x))^{-1} x\in X$.  
\item 
	a map defined by
	$$F = \map{X}{X}{x}{\ab(s_0(x))^{-1} x}.$$\nomenclature[L]{$F$}{continued fraction algorithm}

	\end{itemize}
\end{definition}

We use the word \emph{extended} to indicate that the algorithm uses substitutions. If we do not use the substitutions, we can retain their matrices $\ab(s_0(x))$ only, or even just the map $F$. We then speak of a \emph{continued fraction algorithm}, denoted $(X,F)$.

Given a continued fraction algorithm $(X,F)$, there are several possible choices for $S$ and $s_0$ to turn it into an extended continued fraction algorithm $(X,s_0)$. These choices do not yield the same associated subshifts, and not the same complexity function.

\noindent Moreover we define $X_0$ as the biggest subset of $X$ such that
\begin{itemize}
	\item $F^n$ is continuous on $X_0$ for all $n \in \NN$,
	\item $X \setminus X_0$ contains all the non totally irrational directions.
\end{itemize}

For $x \in X$, we also denote $s_i(x) = s_0(F^i(x)), i\geq 0$.
The matrices associated to the substitutions are denoted by $M_i = M_i(x) = \ab(s_i(x))$, and we use the classical notation: if $p \leq q$, $M_{[p,q)}$ stands for the product of matrices $M_p M_{p+1} \dots M_{q-1}$. With the map $s_0$ we can do some symbolic dynamics: it allows to define a map
$$s = \map{X}{S^\NN}{x}{s(x) = (s_n(x))_{n \in \NN}},$$

\nomenclature[L]{$s=s(x)$}{directive sequence associated to $x$}

\begin{definition} \label{def:meas:cont:frac}

Let $(X,s_0)$ be an extended continued fraction algorithm, equipped with a measure $\mu$.
We say that $(X,s_0,\mu)$ is an \emph{extended measured continued fraction algorithm} if 
	\begin{enumerate}
		\item $\mu$ is an ergodic $F$-invariant Borel probability measure,
		\item The map $s_0$ is measurable with respect to $\mu$,
		\item $\mu(X_0) = 1$,
		\item for all measurable $Y \subseteq X$ we have $\mu(Y) = 0 \Longrightarrow \mu(F(Y)) = 0$, 
		\item $\exists \epsilon>0, \forall x\in X_0, \forall n \geq 1, \mu(F^n(\{y\in X \mid M_{[0,n)}(y) = M_{[0,n)}(x)\})) > \epsilon$.

\end{enumerate}
\end{definition}

As above, if we are not interested in the particular choice of substitution, we will consider instead the \emph{measured continued fraction algorithm} $(X,F,\mu)$.

Remark that for usual continued fraction algorithms, $X \setminus X_0$ is an invariant set.
And in this case, the ergodicity of $\mu$ gives $\mu(X \setminus X_0) = 1$ or $\mu(X \setminus X_0)=0$. Hence, the hypothesis $\mu(X_0) = 1$ is equivalent to say that $\mu$ is not supported only by $X \setminus X_0$.

Now we give a criterion to prove that a map satisfies the hypotheses of Definition~\ref{def:meas:cont:frac}:
\begin{proposition} \label{prop:cf}
Assume that we have a map
	\[
		F = \map{X}{X}{x}{\ab(s_0(x))^{-1} x}.
	\]
	 such that
	there exists a finite union $H$ of rational hyperplanes of $\P\RR_+^{d}$
	that partition $X \setminus H$
	into a finite number of pieces $(X_i)_{i \in I}$ such that
	for every $i \in I$,
	\begin{itemize}
		\item $\ab(s_0)$ is constant on $X_i$,
	 	\item $(\ab(s_0(x))^{-1} X_i) \setminus H$ is a union of pieces:
			$(\ab(s_0(x))^{-1} X_i) \setminus H = \bigcup_{j\in J}X_j$
			for some $J \subseteq I$.
	\end{itemize}
	If $\mu$ is a Borel ergodic probability measure on $X$ such that
	\begin{itemize}
		\item for all $i \in I$, $\mu(X_i) > 0$,
		\item for every measurable subset $Y$, we have $\mu(Y) = 0 \Longrightarrow \mu(F(Y)) = 0$.
		\item the measure
			of the set of non totally irrational directions is zero,
	\end{itemize}
	then $(X, s_0, \mu)$ is an extended measured continued fraction algorithm
	as defined in Definition~\ref{def:meas:cont:frac}.
\end{proposition}

Such a family $(X_i)_{i \in I}$ is sometimes called a Markov partition. 

\begin{proof}
	Let $X_{\mathrm{irr}} \subseteq X$ be the set of totally irrational directions of $X$. 
	With such hypotheses, the map $F$ is continuous on
	$X_{\mathrm{irr}} \subseteq X \setminus H$,
	and the set $X_{\mathrm{irr}}$ is invariant by $F$.
	So, for all $n \in \NN$, $F^n$ is continuous on
	$X_{\mathrm{irr}}$, so $X_{0} = X_{\mathrm{irr}}$.
	By hypothesis, we have $\mu(X_0) = 1$.
	It remains to show the property
	\[
		\exists \epsilon > 0,\ \forall x \in X_0,\ \forall n \geq 1,\ \mu(F^n(\{ y \in X \mid M_{[0,n)}(x) = M_{[0,n)}(y) \})) > \epsilon.
	\]
	
	Let us show that for all $n \geq 1$, we have
	\[
		F^n(\{ y \in X \mid M_{[0,n)}(x) = M_{[0,n)}(y) \})
		\supseteq M_{n-1}^{-1}(x) X_{i(F^{n-1} x)}
		= F(X_{i(F^{n-1} x)}),
	\]
	where $i(y) \in I$ is such that $y \in X_{i(y)}$.
	It will ends the proof since the sets $F(X_{i(F^{n-1} x)})$ have positive measure
	\[
		\mu(F X_{i(F^{n-1} x)})
		= \mu(F^{-1}(F X_{i(F^{n-1} x)}))
		\geq \mu( X_{i(F^{n-1} x)} ) > 0,
	\]
	and there are finitely many of them.
	
	The inclusion is equivalent to
	\[
		\{ y \in X \mid M_{[0,n)}(x) = M_{[0,n)}(y) \} \supseteq M_{[0,n-1)}(x) X_{i(F^{n-1} x)}.
	\]
	
	We show the inclusion for every $x \in X_0$ by induction on $n$.
	
	Let $y \in M_{[0,n-1)}(x) X_{i(F^{n-1} x)}$.
	If $n = 1$, we have $y \in X_{i(x)}$, so $M_{0}(x) = M_0(y)$.
	Otherwise, by hypothesis,
	we have $X_{i(F(x))} \subseteq F(X_{i(x)})$, so
	\[
		M_0(x) X_{i(F(x))} \subseteq X_{i(x)}.
	\]
	If we iterate this, we see that we have $M_{[0,n-1)}(x) X_{i(F^{n-1}(x))} \subseteq X_{i(x)}$.
	So, $M_0(x) = M_0(y)$ and we have $F(y) \in M_{[0,n-2)}(F(x)) X_{i(F^{n-1}(x))}$.
	By induction hypothesis with $x$ replaced with $F(x)$,
	we get that $M_{[0,n)}(x) = M_{[0,n)}(y)$.
\end{proof}

The hypotheses of this proposition are true for the usual continued fraction algorithms of Brun, and Cassaigne.
See Sections~\ref{sec-examples} and~\ref{sec-cassaigne}.

\begin{lemma} \label{lem:near-prefix}
Let $(X,s_0)$ be an extended continued fraction algorithm.
We have
    \[
       \forall x\in X_0, \forall k \in \NN, \exists r > 0,
       d(x,y) \leq r \Longrightarrow \forall i \leq k, s_i(x) = s_i(y).
    \]
\end{lemma}

\begin{proof}
	This is an obvious consequence of the definition of the set $X_0$, where $s_n$ is continuous for every $n \in \NN$.
\end{proof}

\begin{remark}\label{rem-lien-v-x}
	We have $x=M_{[0,k)} F^k(x)$, so that
	$v(x) \in \bigcap_{k\geq 0} M_{[0,k)}(x) \RR_+^{d+1}$.
	If this cone is a line it follows that $v(x)=v^{(0)}$,
	and more generally that $v(F^k(x))=v^{(k)}$.
\end{remark}

\subsection{Lyapunov exponents}

Consider a dynamical system $(X,T)$ with a $T$-invariant Borel probability measure $\mu$ on $X$. {\bf A cocycle} of the dynamical system $(X,T)$ is a map $\bM\colon X\times \NN\to GL_{d+1}(\RR)$ such that
\begin{itemize}
\item $\bM(x,0)=\Id$ for all $x\in X$,
\item $\bM(x,n+m)=\bM(T^n(x),m)\bM(x,n)  $ for all $x\in X$ and $n,m\in\NN$.
\end{itemize}
We denote $\bM(x,-n)=\bM(x,n)^{-1}$ for $n>0$.
Let $\norm{.}$ be any norm on $\RR^{d+1}$.

\begin{theorem}[Oseledets]\label{thm:oseledets}
	Let $(X,T)$ be a dynamical system and $\mu$ be an invariant probability measure  for this system. 
	Let $\bM$ be a cocycle of $(X,T)$ in $GL_{d+1}(\RR)$ such that the maps $x\mapsto \ln{\matrixnorm{\bM(x,1)}}, x\mapsto \ln{\matrixnorm{\bM(x,-1)}}$ are $L^1$-integrable with respect to $\mu$. 

	Then there exists a measurable set $Z\subseteq X$\nomenclature[L]{$Z$}{subset of measure one in Oseledets theorem} with $\mu(Z)=1$ and measurable functions $r,\theta_i$ from $Z$ to $\RR$,
	such that for all $x\in Z$ there is 
	\begin{itemize}
	\item an integer $r(x)$ with $0<r(x)\leq d+1$
	\item $r(x)$ distinct numbers $\theta_1(x)>\dots>\theta_{r(x)}(x)$
	\item a sequence of linear subspaces 
	$$\RR^{d+1}=E_1(x)\supsetneqq \dots \supsetneqq E_{r(x)}(x)\supsetneqq E_{r(x)+1}(x)=\{0\}$$\nomenclature[L]{$E_i(x)$}{Lyapunov spaces}
	such that 
	$$y\in E_i(x)\setminus E_{i+1}(x)\iff \lim_{n\to\infty} \frac{1}{n}\ln\norm{\bM(x,n)y}=\theta_i(x)$$
	\end{itemize}

	If in addition $\mu$ is an ergodic measure, then $Z$ can be chosen so that the functions that map $x$ to $r(x),\theta_1(x),\dots,\theta_{r(x)}(x)$, $\dim E_1(x)\dots, \dim E_r(x)$ are constant on $Z$.
	Then we denote $\theta_i(x)$ by $\theta_i(\mu)=\theta_i(T,\mu)$.
\end{theorem}

The numbers $\theta_{i}(T,\mu), i=1\dots m$ are called {\bf Lyapunov exponents} of the cocycle \cite{Oseledets,Furst-Kest}.
We also use the following formulas, see \cite[Theorem 6.3]{Bert.Delec.14}.
Remark that in order to avoid confusion we denote the Lyapunov exponents of $x$ by $\theta_i(x)$ and by $\theta_i(\mu)$ their value for almost all points with respect to the ergodic measure $\mu$. 
\begin{corollary}\label{lyapunov-furstenberg}
	In the ergodic case we have,
	for every $x\in Z$,
	\[
	\theta_1(\mu) = \lim_{n\to \infty} \frac 1 n \ln \matrixnorm{\bM(x,n)},\quad
	\theta_2(\mu) = \lim_{n\to \infty} \frac 1 n \ln \matrixnorm{\bM(x,n)\restrict{E_2(x)}}.
	\]
\end{corollary}

\nomenclature[G$\h$]{$\theta_2(x)$}{Lyapunov exponent}

\subsection{Lyapunov exponents for a continued fraction algorithm}
In the following, since we use transpose of matrix, we consider the dual space of $\RR^{d+1}$. For a vector $v\in\RR^{d+1}$ we denote $v^\circ$\nomenclature[S]{$v^\circ$}{orthogonal in the dual space} the orthogonal in the dual space, i.e., the set of linear forms which vanish on $v$.

\nomenclature[G]{$\varphi$}{linear form}
\begin{remark} \label{rem:theta2}
	For every $x \in X$, we have the equality
	\[
		{v(x)}^\circ = \{ \varphi \circ \pi_x \mid \varphi \in \left( \RR^{d+1} \right)^* \}.
	\]
	Indeed, we have $\pi_x(v(x)) = 0$ and $\im(\pi_x) = P$ is a hyperplane.
\end{remark}

\begin{lemma} \label{lem-bon-cocycle}
Let $(X,F)$ be a continued fraction algorithm.
\begin{itemize}
\item $\forall y \in X$, $\forall n \in \NN$, $\forall N \geq n$, $M_{[0,N)}(y)=M_{[0,n)}(y)M_{[0,N-n)}(F^ny)$.

\item $\forall y \in X$, $\forall n \in \NN$, $\forall N \geq n$, $\pi_yM_{[0,N)}(y)=\pi_y M_{[0,n)}(y)\pi_{F^n(y)}M_{[0,N-n)}(F^ny)$.

\item The map
$$\map{X\times\NN}{GL_{d+1}(\RR)}{(x,n)}{{\bM(x,n)}=\transp{M_{[0,n)}}(x)}$$
defines a cocycle.
\end{itemize}
\end{lemma}

\begin{proof}
For the first point, consider $i\geq n$, then $F^i(y)=F^{i-n}(F^n(y))$, thus $M_i(y)=M_{i-n}(F^n(y))$. Then the third point is a consequence of the definition of a cocycle. 
It remains to prove the second point.
It is an consequence of the first point and of the identity
\[
	\pi_x M \pi_{M^{-1} x} = \pi_x M,
\]
for every $x \in \P\RR_+^d \cap M\P\RR_+^d$ and every matrix $M \in GL_{d+1}(\RR)$.
This identity comes from the fact that $M\pi_{M^{-1}x}(y) = My - h(y) M v(M^{-1}x)$,
and $\pi_x(M v(M^{-1}x)) = 0$. 
\end{proof}

Let now $(X,F,\mu)$ be a measured continued fraction algorithm, as defined in Section~\ref{sec:cont-frac}.
We use Theorem~\ref{thm:oseledets} and Corollary~\ref{lyapunov-furstenberg} for the cocyle defined in Lemma~\ref{lem-bon-cocycle}.
Remark that the hypothesis of integrability is automatically satisfied since this cocycle takes only a finite numbers of values.
In the following we consider the set $Z$ given by Theorem~\ref{thm:oseledets} for this cocycle $(x,n) \mapsto \transp{M_{[0,n)}}(x)$.
In particular we have the following corollary.

\begin{corollary} \label{cor:oss}
	For every $x\in Z$ and $\varphi \in (\RR^{d+1})^*$ we have 
	\[
		\lim_{n \to \infty} \frac{1}{n} \ln \norm{\varphi M_{[0,n)}}_1 = \theta_i(x) \ \Longleftrightarrow\ \varphi \in E_i(x) \setminus E_{i+1}(x),
	\]
	and we have
	\begin{align*}
		\theta_1(x) = \theta_1(\mu) &= \lim_{n\to \infty} \frac 1 n \ln \matrixnorm{\transp{M_{[0,n)}}(x)}_\infty,\\
		\theta_2(x) = \theta_2(\mu) &= \lim_{n\to \infty} \frac 1 n \ln \matrixnorm{\transp{M_{[0,n)}}(x)\restrict{E_2(x)}}_\infty.
	\end{align*}
\end{corollary}

\begin{definition}\label{def-algo-Pisot}
A measured continued fraction algorithm $(X,F,\mu)$ is said to satisfy Pisot condition
if for $\mu$-almost every point $x$ we have $\theta_1(x)>0>\theta_2(x)$, and $\codim(E_2(x))=1$.
\end{definition}

\begin{lemma}\label{Jeff} Let $x\in Z$.
	Assume $\codim(E_2(x))=1$ and $\displaystyle\bigcap_{n\geq 0}M_{[0,n)}(x)\RR_+^{d+1} = \RR_+ v$.
	Then we have
	\[
		E_2(x) = v^\circ.
	\]
	
\end{lemma}
\begin{proof}

	Let  $y\in E_2(x)$ (a linear form, or a line vector), and let $w = \begin{pmatrix} 1 \\ \vdots \\ 1 \end{pmatrix} \in \RR_+^{d+1}$. \newline
	By Hölder inequality we have
	\[
		\abs{y M_{[0,n)} w} \leq \norm{y M_{[0,n)}}_\infty \norm{w}_1.
	\]
	Let us denote $w_n = \frac{M_{[0,n)} w}{\norm{M_{[0,n)} w}_1}$, we obtain
	\begin{align*}
		\frac{1}{n} \ln \abs{y w_n} &\leq \frac{1}{n} \ln(\norm{y M_{[0,n)}}_\infty) + \frac{1}{n} \ln(\norm{w}_1) - \frac{1}{n} \ln(\norm{M_{[0,n)} w}_1). 
	\end{align*}
	But we have
	\[
		\norm{M_{[0,n)} w}_1 = \sum_{i=0}^d \abs{e_i^* M_{[0,n)} w}.
	\]
	And for $e_i^* \not\in E_2(x)$ (it exists since $d \geq 1$), we have $\lim_{n \to \infty} \frac{1}{n} \ln \norm{e_i^* M_{[0,n)}}_\infty = \theta_1(x)$.
	Moreover, we have $\norm{e_i^* M_{[0,n)}}_\infty \leq \abs{e_i^* M_{[0,n)} w}$ so we have
	\[
		\liminf_{n \to \infty} \frac{1}{n} \ln \norm{M_{[0,n)} w}_1 \geq \lim_{n \to \infty} \frac{1}{n} \ln \norm{e_i^* M_{[0,n)}}_\infty = \theta_1(x).
	\]
	And we have $\lim_{n \to \infty} \frac{1}{n} \ln \norm{y M_{[0,n)}}_\infty = \theta_i(x), i\geq 2$ with $y\in E_i(x)\setminus E_{i+1}(x)$, so we get
	\[
		\limsup_{n \to \infty} \frac{1}{n} \ln \abs{y w_n} \leq \theta_i(x) - \theta_1(x) < 0.
	\]
	And we also have $w_n \xrightarrow[n \to \infty]{} v$ by hypothesis. We deduce that we have
	\[
		y v = \lim_{n \to \infty} y w_n = 0,
	\]
	so $y \in v^\circ$. Since $\dim(E_2(x))=\dim(v^\circ)$ we obtain that $E_2(x) = v^\circ$.

\end{proof}

\begin{lemma} \label{lem-cocycle-matrice}
	Let $x \in Z$. We assume that
	$$\bigcap_{n\geq 0}M_{[0,n)}(\RR_+^{d+1}) = \RR_+ v,$$
	and that $\codim(E_2(x)) = 1$. Then, we have
	$$\matrixnorm{\pi_v M_{[0,n)}}_1 \leq (d+1)\matrixnorm{\transp{M_{[0,n)}}\restrict{E_2(x)}}_\infty.$$
\end{lemma}
\begin{proof}
	Recall that $\pi_v$ is the projection on $P$ with respect to the direction $\RR v$.

	\begin{align*}
	  \matrixnorm{\pi_v M_{[0,n)}}_1
		&= \sup_{w, \norm{w}_1\leq 1} \norm{ \pi_v M_{[0,n)}w}_1\\
		&= \sup_{w, \norm{w}_1\leq 1} \sum_{i=0}^d \abs{e_i^* \pi_v M_{[0,n)}w}\\
		&\leq \sum_{i=0}^d \sup_{w, \norm{w}_1\leq 1} \abs{e_i^* \pi_v M_{[0,n)}w}
		  =  \sum_{i=0}^d \norm{ e_i^* \pi_v M_{[0,n)}}_{\infty}.
	\end{align*}

	On the other hand, we have by Lemma~\ref{Jeff} and Remark~\ref{rem:theta2} 
	    $$E_2(x)=\{\varphi\circ \pi_v \mid \varphi\in (\RR^{d+1})^*\} \subseteq (\RR^{d+1})^*.$$ 
	    
    We conclude

    \[
    	\matrixnorm{\pi_v M_{[0,n)}}_1 \leq \matrixnorm{\transp{M_{[0,n)}}\restrict{E_2(x)}}_{\infty} \sum_{i=0}^d  \norm{e_i^*\circ \pi_v}_\infty.
    \]

    Since $\norm{v}_1=1$, we have by definition $\pi_v(e_j) = e_j - v$, thus we deduce $e_i^*\circ \pi_v(e_j)=\delta_{i,j} - v_i$, and
    $\norm{e_i^*\circ \pi_v}_\infty=\max(v_i, 1-v_i)$. We deduce

	    \[
	        \matrixnorm{\pi_v M_{[0,n)}}_1 \leq (d+1)\matrixnorm{\transp{M_{[0,n)}}\restrict{E_2(x)}}_\infty.
	    \]
		    
\end{proof}

From Lemma~\ref{lem-cocycle-matrice} and Corollary~\ref{cor:oss}, we deduce the following

\begin{corollary} \label{cor:theta2}
	Let $x \in Z$. Assume that $\bigcap_{n\geq 0}M_{[0,n)} \P\RR_+^{d} = \{x\}$ and that \break $\codim(E_2(x)) = 1$.
	Then we have the equality
	\[
		\lim_{n \to \infty} \frac{1}{n} \ln \matrixnorm{\pi_x M_{[0,n)}}_1 = \theta_2(x).
	\]
\end{corollary}

\begin{proof}
	By Lemma~\ref{Jeff} and Remark~\ref{rem:theta2}, we have $E_2(x) = \{ \varphi \circ \pi_x \mid \varphi \in (\RR^{d+1})^* \}$.
	Let $\varphi \in (\RR^{d+1})^* \setminus \{0\}$ be a non-zero linear form such that $\varphi \circ \pi_x \in E_2(x) \backslash E_3(x)$.
	Then, using Lemma~\ref{lem-cocycle-matrice}, we have the inequalities
	\[
		\norm{\varphi\circ \pi_x M_{[0,n)}}_1 \leq \matrixnorm{\pi_x M_{[0,n)}}_1 \norm{\varphi}_1 \leq (d+1) \matrixnorm{\transp{M_{[0,n)}}\restrict{E_2(x)}}_\infty \norm{\varphi}_1.
	\]
	We have $\lim_{n \to \infty} \frac{1}{n} \ln \norm{\varphi \circ \pi_x M_{[0,n)}}_1 = \theta_2(x)$ since $\varphi\circ \pi_x \in E_2(x) \setminus E_3(x)$,
	and we have
	\[
		\lim_{n \to \infty} \frac{1}{n} \ln \matrixnorm{\transp{M_{[0,n)}}\restrict{E_2(x)}}_\infty = \theta_2(x)
	\]
	by Corollary~\ref{cor:oss}.
	Thus by squeeze theorem we get that the limit $\lim_{n \to \infty} \frac{1}{n} \ln \matrixnorm{\pi_x M_{[0,n)}}_1$ exists and is equal to $\theta_2(x)$.
\end{proof}

\section{A lot of good points}\label{sec-lot-good}

The aim of this section is to prove that one seed point gives a set of full $\mu$-measure of good directive sequences (see Proposition~\ref{prop:mesure:good}). With this result, and with Theorem~\ref{thmC}, the proof of Theorem~\ref{thmA} will be easy.

\subsection{Definitions and main result}

\begin{definition} \label{def:G0}
	We define the set of \emph{seed points} $G_0$\nomenclature[L]{$G_0$}{seed points} as the set of points $x \in X$ such that
	\begin{itemize}
	\item $x$ is a totally irrational direction,
	\item for every $n \in \NN$, $F^n$ is continuous at $x$,
	\item $\limsup_{n \to \infty} \frac{1}{n} \ln \matrixnorm{\pi_x M_{[0,n)}(x)}_1 < 0$, 
	\item there exists a letter $a \in A$
		and a fixed point $u \in {(A^\NN)}^\NN$ of $s(x)$
		such that $W_a(u_0)$ has non-empty interior for the topology $\topo{x}$
		(see Definitions~\ref{def:fixed:point}, \ref{def-topo-paul},
		and~\ref{def:worm}).
\end{itemize}
\end{definition}

\begin{definition}
	We define the set of \emph{good points}
	\[
		G = \{ x \in X \mid s(x) \text{ is a good directive sequence} \},
	\]
	where a good directive sequence is defined in Definition~\ref{defG}\nomenclature[L]{$G$}{good points}.
\end{definition}

\begin{proposition} \label{prop:mesure:good}
	 Let $\mu$ be an $F$-invariant ergodic probability measure on $X$
	 satisfying the Pisot condition (see Definition~\ref{def-algo-Pisot}).
	If $G_0 \neq \emptyset$, then $\mu(G) = 1$.
\end{proposition}

\begin{remark}
$\;$
For a periodic point $x$ (i.e., such that $F^p(x) = x$ for some $p \geq 1$), all the conditions to be a seed point, except the last one, are easily tested:
\begin{itemize}
	\item we have $\limsup_{n \to \infty} \frac{1}{n} \ln \matrixnorm{\pi_x M_{[0,n)}(x)}_1 < 0$ if, and only if, $M_{[0,p)}$ is Pisot.
	\item we have that $x$ is a totally irrational direction if the matrix $M_{[0,p)}$ has an irreducible characteristic polynomial,
	\item we have the continuity of $F^n$ for every $n \in \NN$ if, and only, if we have it for $0 < n \leq p$,
	\item the last property is automatic for an irreducible Pisot unimodular substitution if the Pisot substitution conjecture holds (see Subsection~\ref{sec:pisot} for more details).
\end{itemize}
\end{remark}

In the proof of Proposition~\ref{prop:mesure:good}, we need a variant of the notion of seed point:
\begin{definition}
	We define $G_1$\nomenclature[L]{$G_1$}{auxiliary set related to the good points} as the set of $x \in X$ such that
	\begin{itemize}
		\item $x \in X_0$,
		\item $\limsup_{n \to \infty} \frac{1}{n} \ln \matrixnorm{\pi_x M_{[0,n)}(x)}_1 < 0$, 
		\item there exists a fixed point $u$ of $s(x)$ such that,
			for all $a \in A$,
			$W_a(u)$ has non-empty interior for the topology $\topo{x}$.
	\end{itemize}
\end{definition}

The definitions of $G_1$ and $G_0$ differ only by their last properties where we ask that the interior is not empty for every $a \in A$ rather than for one.

In the following, we use some more notations.

\begin{definition}\label{def-ZC-GBC}
	Let us define
	\[
		Z_C = \{ x \in X \mid \forall n \in \NN, \matrixnorm{\pi_xM_{[0,n)}(x)}_1 \leq C e^{ -\frac{n}{C} } \}.
	\]
	
    Let $B$ be a ball of positive radius in $P$ and let $C > 0$. For every $a \in A$, we define
    \[
        G_{B, C}^a = Z_C \cap \{ x \in X \mid \pi_x^{-1}(B) \cap \HH \subseteq W_{a}(x) \} 
    \]
    and
    \[
	    G_{B, C} = \bigcup_{a \in A} G_{B,C}^a.
	\]
\end{definition}
	
	\nomenclature[L]{$G_{B,C}$}{seed set with explicit bound}
	\nomenclature[L]{$Z_C$}{set of points with explicit exponential convergence}
	
\subsection{Proof of Proposition~\ref{prop:mesure:good}}

In all this subsection, we assume that $\mu$ is an $F$-invariant ergodic probability measure on $X$
satisfying the Pisot condition (see Definition~\ref{def-algo-Pisot}).
The strategy is to prove
\[
	 G_0 \neq \emptyset \Longrightarrow \mu(G_0) > 0
			    \Longrightarrow \mu(G_1) > 0
			    \Longrightarrow \mu(G) = 1.
\]

In the following each step corresponds to one of these implications.

 \subsubsection{Step \texorpdfstring{1: $G_0 \neq \emptyset \Longrightarrow \mu(G_0) > 0$}{1}}

	\begin{lemma} \label{lem:Cx}
		Let $x \in X$. If we have $\limsup_{n \to \infty} \frac{1}{n} \ln \matrixnorm{\pi_x M_{[0,n)}}_1 < 0$,
		then there exists $C_x > 0$ such that $x \in Z_{C_x}$.
	\end{lemma}
	
	\begin{proof}
		Let $l = \limsup_{n \to \infty} \frac{1}{n} \ln \matrixnorm{\pi_x M_{[0,n)}}_1$.
		There exists $n_x \in \NN$ such that for all $n \geq n_x$, we have
        \[
            \matrixnorm{\pi_x M_{[0,n)}}_1 \leq e^{n l/2}.
        \]
        If we take $C_x = \max(\max_{n < n_x} \matrixnorm{\pi_x M_{[0,n)}}_1 e^{-n l/2},\ 1,\  -\frac{2}{l})$, we have $x \in Z_{C_x}$. \\
	\end{proof}
	
	Remark that using this lemma, we have the equality
	\[
		G_1 = X_0 \cap \bigcup_{C > 0} \bigcup_{\substack{(B_a)_{a \in A}\\  \text{balls of positive radius} }}   \bigcap_{a \in A} G_{B_a, C}^a.
	\]
	
    \begin{lemma} \label{lem:ZC}
    	We have $$\lim_{C\to \infty} \mu(Z_C)=1.$$
    \end{lemma}

    \begin{proof}
    	Let $Y = \{x \in X \mid \limsup_{n \to \infty} \frac{1}{n} \ln \matrixnorm{\pi_x M_{[0,n)}}_1 < 0\}$.
		We have $\mu(Y) = 1$ by Corollary~\ref{cor:theta2}, and because $\theta_2(F, \mu) < 0$.
    	And by Lemma~\ref{lem:Cx}, we have $Y \subseteq \bigcup_{C > 0} Z_{C}$.
        Since $Z_C$ is increasing with $C$, we get $\lim_{C\to \infty} \mu(Z_C) \geq \mu(Y) = 1$.
    \end{proof}

\begin{lemma} \label{lem:mZC}
    We have
    \[
        \forall x \in G_0, \exists C \in \RR_+, \forall r > 0, \mu(B(x,r) \cap Z_C \cap X_0) > 0.
    \]
\end{lemma}

\begin{proof}
    Let $x \in G_0$.
    Let $\epsilon > 0$ such that
    \[
    	\forall n \geq 1,\ \mu(F^n(\{y \in X_0 \mid M_{[0,n)}(x) = M_{[0,n)}(y)\})) \geq 2 \epsilon.
    \]
    This is given by our hypotheses on the measured continued fraction algorithm (see Definition~\ref{def:meas:cont:frac}).
    
    Now let
    \[
    	O_K = \{ x \in X \mid 1 < K \min_{i}(v(x)_i) \}.
    \]

    We have $\mu(\bigcup_{K > 1} O_K) = 1$, by definition of a measured continued fraction algorithm (see Definition~\ref{def:meas:cont:frac}).
    So there exists $K > 1$ such that $\mu(O_K) > 1 - \epsilon$.

    By Lemma~\ref{lem:ZC}, there exists $C' \geq 1$ such that $\mu(Z_{C'}) > 1 - \epsilon$.
    We choose the constant $C = (K+1) C_x C'$, where $C_x \geq 1$ is such that $x \in Z_{C_x}$.
    Let $r > 0$. Thank to Remark~\ref{rem:cv:cv}, we can take $n \in \NN$ large enough such that $M_{[0,n)}(x) X$ is included in $B(x,r)$.
    Then, we take
    \[
    	Y = M_{[0,n)}(x) \left( Z_{C'} \cap O_K \right) \cap \{y \in X_0 \mid M_{[0,n)}(x) = M_{[0,n)}(y)\}.
    \]
    By the previous inequalities, we have $\mu(F^n(Y)) = \mu(M_{[0,n)}^{-1} Y) > 0$.
    Using that we have a measured continued fraction algorithm (see Definition~\ref{def:meas:cont:frac}), we get that $\mu(Y) > 0$.
    
    We have $Y \subseteq X_0 \cap B(x,r)$ by construction.
    Let us show that we have the inclusion $Y \subseteq Z_{C}$.
    Let $y \in Y$. For all $N \geq n$, we have by Lemma~\ref{lem-bon-cocycle}
    \begin{align*}
        \matrixnorm{\pi_y M_{[0,N)}(y)}_1 &\leq  \matrixnorm{\pi_y M_{[0,n)}(y)}_1 \matrixnorm{\pi_{F^n(y)} M_{[0,N-n)}(F^ny)}_1 \\
                                    &\leq \matrixnorm{\pi_y M_{[0,n)}(x)}_1 \matrixnorm{\pi_{F^n(y)} M_{[0,N-n)}(F^ny)}_1 
    \end{align*}
    because we have $M_{[0,n)}(x) = M_{[0,n)}(y)$ by construction of $Y$.
    Now, let us show that we have $\matrixnorm{\pi_y M_{[0,n)}(x)}_1 \leq (K+1) C_x e^{ -n/C }$.
    
    Recall that $h\colon \RR^{d+1} \to \RR$ is the sum, i.e., the linear form such that for every $w \in \RR_+^{d+1}$, $h(w) = \norm{w}_1$. 
    Now, for every $z \in X$, we have for all $w \in \RR^{d+1}$, $\pi_z w = w - h(w) v(z)$.
    Let $M = M_{[0,n)}(x)$.
    We have $x \in Z_{C_x}$ and $C \geq C_x$, so we have for every $w \in \RR^{d+1}$,
    \[
    	\norm{M w - h(M w) v(x)}_1 = \norm{ \pi_x M w }_1 \leq C_x e^{ -n/C } \norm{w}_1.
    \]
    
    Let $y'$ such that $M y' = y$.
    The previous inequality applied with $w = v(y')$ gives
    \[
    	\norm{Mv(y') - h(M v(y')) v(x)}_1 \leq C_x e^{ -n/C }.
    \]
    By triangular inequality, and using that $v(y) = \frac{M v(y')}{h(M v(y'))}$, we have
    \begin{align*}
    	\norm{ \pi_y M w }_1 &= \norm{ Mw - h(Mw) v(y) }_1 \\
    						&\leq \norm{Mw - h(M w) v(x)}_1 + \abs{\frac{h(Mw)}{h(Mv(y'))}} \norm{M v(y') - h(M v(y')) v(x)}_1,
    \end{align*}
    so we get
    \[
    	\norm{ \pi_y M w }_1 \leq C_x e^{ -n/C } \left( \norm{w}_1 + \abs{\frac{h(Mw)}{h(M v(y'))}} \right).
    \]
    Now, we have $y' \in O_K$, so $h(M v(y')) \geq \frac{1}{K} \max\{\norm{C}_1 \mid C \text{ column of } M\}$, and we have $\abs{h(M w)} \leq \max\{\norm{C}_1 \mid C \text{ column of } M\} \norm{w}_1$. Thus, we have
    \[
    	\norm{w}_1 + \abs{\frac{h(Mw)}{h(M v(y'))}} \leq (K+1) \norm{w}_1.
    \]
    We deduce that $\matrixnorm{\pi_y M_{[0,n)}(x)}_1 \leq (K+1) C_x e^{ -n/C }$.
    And by construction of $Y$ we have $F^n(y) \in Z_{C'}$, and we have $C \geq C'$, so we have
    \[
    	\matrixnorm{\pi_{F^n(y)} M_{[0,N-n)}(F^ny)}_1 \leq C' e^{ -(N-n)/C }.
    \]
    We deduce that
    \[
    	\matrixnorm{\pi_y M_{[0,N)}}_1 \leq C e^{ -N/C }.
    \]
    Hence, we get that $Y \subseteq B(x,r) \cap Z_{C} \cap X_0$ with $\mu(Y) > 0$, so $\mu(B(x,r) \cap Z_C \cap X_0) > 0$.
\end{proof}

The next lemma says that if $x$ is in $G_0$, then there exists a set of positive measure of points close to $x$ where the Rauzy fractals are close to each other for the Hausdorff distance $\delta$ in $P$ defined by
\[
	\delta(A,B) = \max \left( \sup_{x \in A} d(x, B),\ \sup_{x \in B} d(A, x) \right),
\]
for every subsets $A,B$ of $P$.

\nomenclature[G]{$\delta$}{Hausdorff distance}

\begin{lemma} \label{lem:approx}
    For all $x \in G_0$ and for all $\epsilon > 0$, there exists $C > 0$ and $V \subseteq B(x,\epsilon) \cap Z_C \cap X_0$ such that $\mu(V) > 0$ and
    $\forall y \in V,\ \forall a \in A,\ \delta(R_a(x), R_a(y)) \leq \epsilon$.
\end{lemma}
\begin{proof}
Let $x \in G_0$.
    Let $C \in \RR_+$ given by Lemma~\ref{lem:mZC}.
    Let $k \in \NN$ big enough to have
    \[
        \frac{C e^{ -k/C }}{1 - e^{-1/C}} \max_{t \in \Sigma} \norm{t}_1 \leq \frac{\epsilon}{3},
    \]
    where $\Sigma \subseteq \ZZ^{d+1}$ is the finite Dumont-Thomas alphabet for our $S$-adic system, see Definition~\ref{def-dt}.
    Then, we choose $R^{(k)} > 0$ small enough such that for all $y \in B(x, R^{(k)})$, we have
    \[
        \forall (t_i)_{i\leq k} \in \Sigma^{k+1},
        \norm{\pi_x \left( \sum_{i=0}^k M_{[0,i)}(x) t_i \right) - \pi_y \left( \sum_{i=0}^k M_{[0,i)}(x) t_i \right) }_1 \leq \frac{\epsilon}{3}.
    \]
    It is possible because we compare the images by $\pi_x$ and by $\pi_y$ of the same element $\displaystyle\sum_{i=0}^k M_{[0,i)}(x) t_i$ that lives in a finite set. 
    
    Then, we take $r > 0$ given by Lemma~\ref{lem:near-prefix} such that $\norm{x - y}_1 \leq 2r \Longrightarrow \forall i \leq k, s_i(x) = s_i(y)$, and we can assume that $r \leq R^{(k)}$ and $r \leq \epsilon$ up to take the minimum of the three values,
    and we let $V = B(x,r) \cap Z_C \cap X_0$.
    Let's show that the set $V$ satisfy what we want. We have $\mu(V) > 0$ by Lemma~\ref{lem:mZC}.
    
    Let $y \in V$. We have convergence of the series $\sum_n \matrixnorm{\pi_x M_{[0,n)}(y)}_1$ since $y \in Z_C$. 
	Hence, we can use Corollary~\ref{cor:bounded:cover}, and we get that
    \[
    	R_a(y) = \{ \sum_{n=0}^\infty \pi_y(M_{[0,n)}(y) t_n) \mid ... \xrightarrow{t_n, s_n(y)} ... \xrightarrow{t_0, s_0(y)} a \in \A \},
    \]
    and we get the same description for $R_a(x)$.
    
    Let $p \in R_a(x)$, and let $\dots \xrightarrow{t_n, s_n(x)} \dots \xrightarrow{t_0, s_0(x)} a$ be a left-infinite path in the abelianized prefix automaton $\A$ such that
    \[
    	p = \sum_{n=0}^\infty \pi_x(M_{[0,n)(x)} t_n).
    \]
    We have $s_i(x) = s_i(y)$ for all $i \leq k$, and the matrices of substitutions of $S$ are invertible, 
    so we can take a left-infinite path
    $\dots \xrightarrow{t_n', s_n(y)} \dots \xrightarrow{t_0', s_0(y)} a$ in the automaton $\A$ such that $t_i = t_i'$ for all $i \leq k$.
    This defines a point $p' \in R_a(y)$ by
    \[
    	p' = \sum_{n=0}^\infty \pi_y(M_{[0,n)(y)} t_n').
    \]
    
    We have the inequalities
    \begin{align*}
        \norm{p - p'}_1
          \leq{}& \norm{\pi_{x} \left( \sum_{i=0}^k M_{[0,i)}(x) t_i \right) - \pi_{y} \left( \sum_{i=0}^k M_{[0,i)}(y) t_i \right) }_1 \\
	        &+ \norm{ \sum_{i=k+1}^{\infty} \pi_{x} \left( M_{[0,i)}(x) t_i \right) }_1
                 + \norm{ \sum_{i=k+1}^{\infty} \pi_{y} \left( M_{[0,i)}(y) t_i' \right) }_1\\
        \intertext{Then using that $x \in Z_C$ and $y \in Z_C$, we have} 
        \norm{p - p'}_1 
          \leq{}& \frac{\epsilon}{3}+ \sum_{j=k+1}^{\infty}Ce^{-j/C}\norm{t_j}_1+  \sum_{j=k+1}^{\infty}Ce^{-j/C}\norm{t_j'}_1\\
          \leq{}& \frac{\epsilon}{3} +2 \frac{C e^{-k/C}}{1 - e^{-1/C}} \max_{t \in \Sigma} \norm{t}_1 \\
          \leq{}& \epsilon.
    \end{align*}
    
    By reverting the role of $x$ and $y$, we also show that for any point $p \in R_{a}(y)$, there exists a point $p' \in R_{a}(x)$ such that $\norm{p - p'}_1 \leq \epsilon$, so we get the wanted inequality
    \[
    	\delta(R_{a}(x), R_{a}(y)) \leq \epsilon.
    \]
\end{proof}

 \begin{lemma} \label{lem:perturb}
 	If $G_0 \neq \emptyset$ then $\mu(G_0) > 0$.
 \end{lemma}
 
 \begin{proof}
 	Let $x \in G_0$. Let $a \in A$, and let $u$ be a fixed point of $s(x)$, such that there exists an open ball $B_a = B(c_a,r_a)$ of positive radius $r_a > 0$ such that $\HH \cap \pi_{x}^{-1}(B_a) \subseteq W_a(u)$.
    Then, by Lemma~\ref{lem:car:interior}, we have that for all $b \in A \setminus \{a\}$ and ${t \in \Lambda \setminus \{0\}}$
    
    \[
    	B_a \cap R_b(x)=\emptyset= B_a \cap (R(x)+t).
    \]
	
    We take the $C' > 0$ and $V \subseteq B(x, \epsilon) \cap Z_{C'} \cap X_0$ given by Lemma~\ref{lem:approx} for $\epsilon = r_a/2$. 
    Let $B'_a = B(c_a,r_a/2)$ the open ball with half the radius of $B_a$ and same center.
    Let us show that for all $b \in A \setminus \{a\}$ and ${t \in \Lambda \setminus \{0\}}$ we have
    \[
    	\forall y \in V,\ B'_a \cap R_b(y)=\emptyset= B'_a \cap (R(y)+t).
    \]
    
    If $b \in A\setminus\{a\}$, then we have for all $y \in V$ and for all $p \in R_b(y)$,
    \[
    	d(c_a, p) \geq d(c_a, R_b(x)) - d(R_b(x), p) \geq r_a - \delta(R_b(x), R_b(y)) \geq r_a - r_a/2 = r_a/2,
    \]
    so $B'_a \cap R_b(y) = \emptyset$.
    
    If $t \in \Lambda \setminus \{0\}$, then we have for all $y \in V$ and all $p \in R(y) + t$,
    \[
    	d(c_a, p) \geq d(c_a, R(x)+t) - d(R(x)+t, p) \geq r_a - \delta(R(x)+t, R(y)+t) \geq r_a/2,
    \]
    so $B'_a \cap (R(y)+t) = \emptyset$.
    
    By Lemma~\ref{lem:car:interior}, we deduce that for every $y \in V$, we have the inclusion $\pi_{y}^{-1}(B'_a) \cap \HH \subseteq W_{a}(y)$.
    We get that $V \subseteq G_ {B'_a, C'}^a$,
    so the set $G_0$ has positive measure.
 \end{proof}

\subsubsection{Step \texorpdfstring{2: $\mu(G_0) > 0 \Longrightarrow \mu(G_1) > 0$}{2}}
    
	\begin{lemma} \label{lem:G0:G1}
		We have $\mu(G_0) > 0 \Longrightarrow \mu(G_1) > 0$.
	\end{lemma}
	
	\begin{proof}
		If $\mu(G_0) > 0$, then by Poincaré recurrence theorem, we have
		\[
			\mu \left( \bigcap_{n \in \NN} \bigcup_{k \geq n} F^{-k} G_0 \right) > 0.
		\]
		
		Let $x \in \bigcap_{n \in \NN} \bigcup_{k \geq n} F^{-k} G_0$.
		Using $\limsup_{n \to \infty} \frac{1}{n} \ln \matrixnorm{\pi_x M_{[0,n)}(x)}_1 < 0$ and Remark~\ref{rem:cv:cv}, we deduce from
		Lemma~\ref{lem:primitivity} that there exists $n_0 \in \NN$ such that $M_{[0,n_0)}(x) > 0$.
		Let $n \geq n_0$ such that $F^n(x) \in G_0$.
		Let $u$ be a fixed point of $s(x)$, and $a \in A$ a letter, such that $W_a(u_n)$ has non-empty interior for the topology $\topo{F^nx}$. For every $b \in A$, we have the equality
		\[
			W_b(u) = \bigcup_{c \xrightarrow{t_n, s_n(x)} \dots \xrightarrow{t_0, s_0(y)} b} M_{[0,n)}(x) W_c(u_n) + \sum_{k=0}^{n-1} M_{[0,k)}(x) t_k.
		\]
		And thanks to $M_{[0,n)}(x) > 0$, we know that for every $b$, the letter $c = a$ appears in this union. The interior of $W_a(u_n)$ is non-empty for the topology $\topo{F^n x}$, so by Lemma~\ref{lem:O:matrix} we have the non-emptiness of the interior of $W_b(u)$ for the topology $\topo{x}$, for every $b \in A$. 
		By Lemma~\ref{lem:theta2} and Lemma~\ref{lem:Cx}, there exists $C > 0$ such that $x \in Z_C$.
		We get that $x \in G_1$. So $G_1$ contains the set $\bigcap_{n \in \NN} \bigcup_{k \geq n} F^{-k} G_0$ which has positive measure.
	\end{proof}

\subsubsection{Step \texorpdfstring{3: $\mu(G_1) > 0 \Longrightarrow \mu(G) = 1$}{3}}

\begin{lemma}\label{cor-G-invariant-measure}
	The set $G$ is measurable and $F$-invariant.
\end{lemma}
\begin{proof}
	Since we assume that $\mu$ is a Borel measure, the fact that $G$ is a measurable set is an exercise left to the reader.
	Let us show that $G$ is $F$-invariant. We check that every point of Definition~\ref{defG} is invariant.
	\begin{enumerate}
		\item By Lemma~\ref{lem:theta2}, the limit $\limsup_{n \to \infty} \frac{1}{n} \ln \matrixnorm{\pi_x M_{[0,n)}(x)}_1$ is $F$-invariant.
		\item The total irrationality of the direction is a $F$-invariant property since $F$ acts by integer invertible matrices.
		\item 
				Let us denote $x'=F(x)=M(x)^{-1}x$, and let $v'^{(k)}=v(F^k(x'))=v^{(k+1)}$, $s'=s(x')$, and $u'$ be the word sequence obtained by shifting $u$ by one, i.e., $u'_k=u_{k+1}$.
				Then $u'$ is a fixed point of $s'$, and $u_0=s_0(x)u_1=s_0(x)u'_0$.
				We deduce that the condition 
				$\pi_{v^{(k)}}^{-1}(B(y, r)) \cap \HH \subseteq W_a(u_k)$ can be written as $\pi_{v'^{(k-1)}}^{-1}(B(y, r)) \cap \HH \subseteq W_a(u'_{k-1})$.

				Thus it suffices to replace $(k_n)_{n \in \NN}$ by $(k_{n+1} - 1)_{n \in \NN}$ or by $(k_{n}+1)_{n \in \NN}$,
				and we deduce that this property is preserved by $F$ and $F^{-1}$.
		\item The last property is clearly preserved, and the limit is the same.
	\end{enumerate}
	We conclude that $F^{-1}(G) = G$.

\end{proof}

\begin{lemma} \label{lem:full:measure}
    If $\mu(G_1) > 0$, then $\mu(G) = 1$.
\end{lemma}

\begin{proof}
Let $(B_a)_{a \in A}$ be a family of balls of positive radius and $C > 0$ such that
\[
	\mu(\bigcap_{a \in A} G_{B_a,C}^a) > 0.
\]

Using Lemma~\ref{lem:ouvert:petite:mesure}, let $O \subseteq X$ be an open set containing all the non totally irrational directions such that $\mu(O) < \mu(\bigcap_{a \in A} G_{B_a,C}^a)$.

First of all we claim that
    \[
        \bigcap_{n_0 \in \NN} \bigcup_{n \geq n_0} F^{-n} \left( \bigcap_{a \in A} G_{B_a,C}^a \cap X_0 \setminus O \right) \subseteq G
    \]
    Indeed if $m$ is inside $\bigcap_{n_0 \in \NN} \bigcup_{n \geq n_0} F^{-n} \left( \bigcap_{a \in A} G_{B_a,C}^a  \cap X_0 \setminus O \right)$, then there exists infinitely many $n$ such that $F^n(m)$ belongs to $\bigcap_{a \in A} G_{B_a,C}^a \cap X_0 \setminus O$.
This gives the third property of Definition~\ref{defG}.
The last property follows from the fact that the set $O$ is open.

Now we apply the Poincaré recurrence theorem:
We have
$\mu(\bigcap_{a \in A} G_{B_a,C}^a \cap X_0 \setminus O) > 0$,
thus $\mu$-almost every point of
$\bigcap_{a \in A} G_{B_a,C}^a \cap X_0 \setminus O$
comes back to $\bigcap_{a \in A} G_{B_a,C}^a \cap X_0 \setminus O$.
We deduce $\mu(G)>0$.
By Lemma~\ref{cor-G-invariant-measure}, $G$ is an $F$-invariant set,
thus by ergodicity we have $\mu(G)=1$.
 \end{proof}
 
This ends the proof of Proposition~\ref{prop:mesure:good}.

\subsection{Proof of Theorem~\ref{thmA}}\label{sec-A}

By hypothesis we can apply Proposition~\ref{prop:mesure:good}. 
We deduce that $\mu(G) = 1$.
Now we apply Theorem~\ref{thmC} for each point of $G$, and it gives 
that the Rauzy fractal induces a generating partition of the translation by $e_0 - v(x)$ on the torus $P/\Lambda$. And its symbolic coding is a measurable conjugacy with the subshift associated to $x$.
If $\psi\colon P/\Lambda \to \TT^d$ is an isomorphism, then we get that the subshift is measurably conjugate to the translation by $\psi(e_0 - v(x))$ on the torus $\TT^d$.

\section{Examples of continued fraction algorithms} \label{sec-examples}

Here we list some classical examples of continued fraction algorithms and we check if the hypotheses of Theorem~\ref{thmA} are fulfilled.

\subsection{Classical continued fraction algorithm}\label{exemple-1d-sturmien}

The algorithm is defined on the whole $X=\P\RR_+^2$. Let $S = \{\tau_0, \tau_1\}$, where
$$\tau_0 = \left\{\begin{array}{ccl} 0 &\mapsto& 0 \\ 1 &\mapsto& 01 \end{array} \right.,
\ \ \ \tau_1 = \left\{\begin{array}{ccl} 0 &\mapsto& 10 \\ 1 &\mapsto& 1 \end{array} \right.$$

\nomenclature[G]{$\tau_0,\tau_1$}{Sturmian substitutions}

Remark that this example is constructed on the same set $S$ as in
Example~\ref{exple-sturmien}, and that the abelianization of the substitutions
are $$ M_0=\begin{pmatrix}1&1\\0&1\end{pmatrix}, \ \ \ M_1=\begin{pmatrix}1&0\\
1&1\end{pmatrix}.$$

We define the extended continued fraction algorithm as:
$$s_0 = \map{X}{S}{\left[(x_0, x_1)\right]}{\begin{cases}\tau_0\quad \text{if } x_0 \geq x_1\\
\tau_1 \quad \text{if } x_0 < x_1\end{cases}}$$

The associated continued fraction algorithm is:
$$F = \map{X}{X}{\left[(x_0, x_1)\right]}{\begin{cases}\left[(x_0-x_1,x_1)\right]\quad \text{if } x_0 \geq x_1\\
\left[(x_0,x_1-x_0)\right] \quad \text{if } x_0 < x_1\end{cases}}$$

This algorithm is known as the \emph{additive continued fraction algorithm} in dimension one, see \cite{Arn.Nog.93}.

Remark that with the change of coordinates $x=\frac{x_0}{x_1}$, we obtain the map
$$\map{(0,+\infty)}{(0,+\infty)}{x}{\begin{cases}x-1\quad \text{if } x \geq 1\\
\ \frac{x}{1-x} \, \ \quad \text{if } x < 1\end{cases}}$$

There exists an ergodic invariant measure for this algorithm which is absolutely
continuous with respect to Lebesgue measure, it density can be expressed
$\frac{1}{x}$ in this coordinate system, but this measure has infinite volume.
So we cannot apply our Theorem~\ref{thmA}.

The usual acceleration of this algorithm restricted to $(0,1)$ is given by the
map $\map{(0,1)}{(0,1)}{x}{\{\frac{1}{x}\}}$. This map has an invariant ergodic
probability measure which is absolutely continuous with respect to Lebesgue
measure, with density $\frac{1}{\log 2}\frac{1}{1+x}$, see \cite{Arn.Nog.93}.
But it cannot be described with a finite number of matrices, so we cannot either
apply our Theorem~\ref{thmA} with this acceleration.

However, this additive algorithm is very-well know, and for every totally
irrational direction, fixed points of the directive sequence $s(x)$ are constituted of Sturmian words.
See~\cite{Pyth.02} for more details.
It could be shown that for every totally irrational direction $x \in \P\RR^1$, the directive sequence $s(x)$ is good.
Hence, we deduce by Theorem~\ref{thmC} that for such direction $x$ there exists a generating partition of the translation by $e_0 - v(x)$ on torus $P/\Lambda \simeq \TT^1$
whose symbolic coding is a measurable conjugacy with the subshift $\Omega_{s(x)}$.
And we easily check that the set of $e_0 - v(x)$, for $x$ a totally irrational direction, is the set of totally irrational vectors of $P/\Lambda$.
On the other hand, the complexity of the subshift $\Omega_{s(x)}$ is $p(n) = n+1$.
Thus, we get, for every irrational translation of $\TT^1$, a generating partition whose symbolic coding has
complexity $n+1$.

This result is already well known. We know that Sturmian words have complexity $n+1$,
and that there exists a partition of the torus with two intervals whose symbolic coding is measurably conjugated to the subshift.
See~\cite{Pyth.02} for more details.

\subsection{Brun algorithm}

Now we give an example of a continued fraction algorithm which does not have associated substitutions (i.e., not an extended continued fraction algorithm). Let $X$ be $\P\RR_+^3$.
For $\zeta\in S_3$ (the permutation group on the set $\{0,1,2\}$) we define $X_\zeta=\{\left[(x_0, x_1, x_2)\right] \in X \mid x_{\zeta(0)}<x_{\zeta(1)}<x_{\zeta(2)}\}$. Then we define the six matrices
$$B_{012}=\begin{pmatrix}1&0&0\\0&1&0\\ 0&1&1\end{pmatrix}, B_{021}=\begin{pmatrix}1&0&0\\0&1&1\\ 0&0&1\end{pmatrix}, B_{120}=\begin{pmatrix}1&0&1\\0&1&0\\ 0&0&1\end{pmatrix},$$ 
$$B_{102}=\begin{pmatrix}1&0&0\\0&1&0\\ 1&0&1\end{pmatrix}, B_{201}=\begin{pmatrix}1&0&0\\1&1&0\\ 0&0&1\end{pmatrix}, B_{210}=\begin{pmatrix}1&1&0\\0&1&0\\ 0&0&1\end{pmatrix}.$$
	
Then we define $M\colon X\to GL_3(\ZZ)$ by $Mx=B_\zeta$ if $x\in X_{\zeta}$. If $x$ is in $X$ and not in some $X_{\zeta}$, then we extend the definition arbitrarily. The following will not depend on these choices.

The Brun algorithm is then defined, as all algorithm of continued fraction, by	
$$F = \map{X}{X}{x}{(Mx)^{-1}x}.$$
In other words, the algorithm subtracts from the largest coordinate the largest of the remaining ones.

\begin{lemma}\cite{Arn.Labbe.15}
The following function is a density function of an invariant probability measure for $F$:
$$\frac{1}{2x_{\zeta(1)}(1-x_{\zeta(1)})(1-x_{\zeta(0)}-x_{\zeta(1)})}\quad x\in X_\zeta.$$
\end{lemma}

\begin{remark}
	For the invariant probability measure $\mu$ given by this lemma, we have
	\begin{itemize}
		\item $\mu(X_0) = 1$,
		\item for every $Y \subseteq X$ such that $\mu(Y) = 0$, we have $\mu(F(Y)) = 0$,
		\item $\exists \epsilon>0, \forall x\in X_0, \forall n\in\NN, \mu(F^n(\{y\in X \mid M_{[0,n)}(y) = M_{[0,n)}(x)\})) > \epsilon$. 
\item This measure is ergodic \cite{Schwei.00,Lag.93}.
	\end{itemize}
	Because $\mu$ is absolutely continuous with respect to Lebesgue, and by Proposition~\ref{prop:cf}.
\end{remark}

In other words, $(X, F, \mu)$ is a measured continued fraction algorithm.

General conditions that permit to check the Pisot condition for the Brun algorithm with the measure $\mu$ are given in \cite{Avil.Dele.19}.

As said at the beginning it is not an extended continued fraction algorithm, but we can extend it.
In \cite{Labbe.15}, some choices have been made to associate a finite set of substitutions to this algorithm.
Denoting $b_{\zeta}$ the substitution with matrix $B_\zeta$ such that $b_{\zeta}(a)$ starts with $a$ for every letter $a\in \{0,1,2\}$, we find that
\[
	b_{210} b_{021} b_{102} = \left(\begin{array}{rcl}
												0 &\mapsto& 0210 \\
												1 &\mapsto& 10 \\
												2 &\mapsto& 210
											\end{array}\right)
	                                       = \left(\begin{array}{rcl}
												0 &\mapsto& 10 \\
												1 &\mapsto& 2 \\
												2 &\mapsto& 0
											\end{array}\right)^3.
\]

\begin{figure}[!h]
    \centering
    \begin{tikzpicture}[scale=1]
		\node (pic) at (0,0) {\includegraphics[width=5cm]{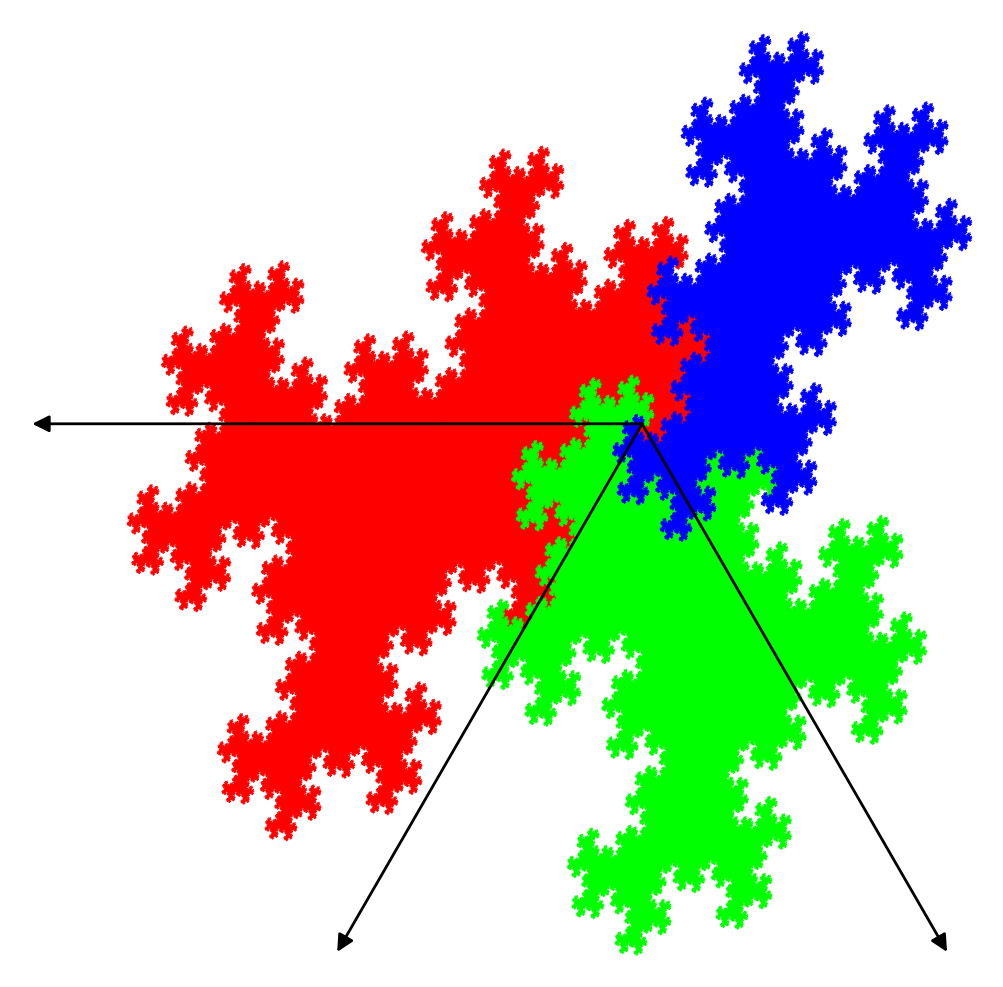}};
		\node at (-2.5cm, .7cm) {$e_1 - e_0$};
		\node at (-1.6cm, -2.1cm) {$e_2 - e_0$};
		\node at (2.7cm, -1.8cm) {$e_2 - e_1$};
		\node at (.7,.6) {$0$};
	\end{tikzpicture}
	\caption{Rauzy fractal of the directive sequence $(b_{210} b_{021} b_{102})^\omega$} \label{fig:ex:Brun}
\end{figure}
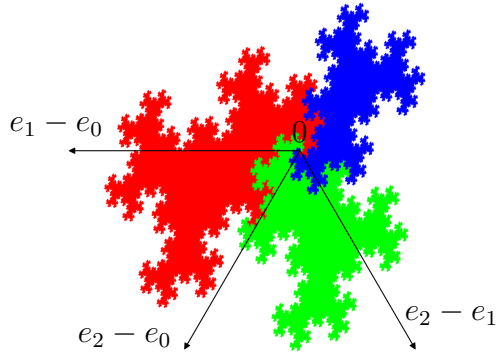

\nomenclature[L]{$b_\zeta$}{substitution for Brun algorithm}

Using Proposition~\ref{prop-decidable},
we can check that the interior of $W_0(u)$ is non-empty for the topology $\topo{x_0}$,
where $u$ is a fixed point of the substitution $b_{210} b_{021} b_{102}$
and $x_0=[v_0]$ with $v_0 = \freq(u)$.
Hence, we can check that $x_0 \in X$ is a seed point.
Therefore, we can apply Theorem~\ref{thmA}, and we get that for $\mu$-almost every point $x$ of $X$,
the $S$-adic subshift associated to $x$ is measurably conjugate to a translation on the torus $P/\Lambda$.

\subsection{Arnoux-Rauzy algorithm}

The Arnoux-Rauzy extended continued fraction algorithm is defined by

	\[
            s_0 = \map{X}{S}{\left[(x_0,x_1,x_2)\right]}{
				\left\{\begin{array}{rl}
						ar_0 &\text{ if } x_0 > x_1+x_2 \\
				        ar_1 &\text{ if } x_1 > x_0+x_2 \\
						ar_2 &\text{ if } x_2 > x_0+x_1
                \end{array}\right.}
	\]

where $S = \{ar_0, ar_1, ar_2\}$ with

\[
        ar_0 = \left\{ \begin{array}{l}
0 \mapsto 0\\
1 \mapsto 10\\
2 \mapsto 20
\end{array}\right., \ \ \
ar_1 = \left\{ \begin{array}{l}
0 \mapsto 01\\
1 \mapsto 1\\
2 \mapsto 21
\end{array}\right., \ \ \
ar_2 = \left\{ \begin{array}{l}
0 \mapsto 02\\
1 \mapsto 12\\
2 \mapsto 2
\end{array}\right.
	\]

The associated continued fraction algorithm is

	\[
            F = \map{X}{X}{\left[(x_0,x_1,x_2)\right]}{
				\left\{\begin{array}{rl}
						\left[(x_0 - x_1 -x_2, x_1, x_2)\right] &\text{ if } x_0 > x_1+x_2 \\
				        \left[(x_0, x_1 - x_0 - x_2, x_2)\right] &\text{ if } x_1 > x_0+x_2 \\
						\left[(x_0, x_1, x_2 - x_1 - x_0)\right] &\text{ if } x_2 > x_0+x_1
                \end{array}\right.}
	\]
\nomenclature[L]{$ar_0, ar_1, ar_2$}{Arnoux Rauzy substitutions}

In other words, the algorithm subtracts from the largest coordinate the sum of the other ones.
Here again we extend this definition to the boundaries of the sets in any choice.
In this case, the set $X$ is defined a posteriori as the subset of points of
$\P\RR_+^{2}$ from which $F^n$ is defined for all $n \in \NN$, it is known as
\emph{the Rauzy gasket}. It is depicted in Figure~\ref{fig:rauzy:gasket}.

\begin{figure}[!h]
    \centering
    \begin{tikzpicture}
		\node (pic) at (0,0) {\includegraphics[width=5cm]{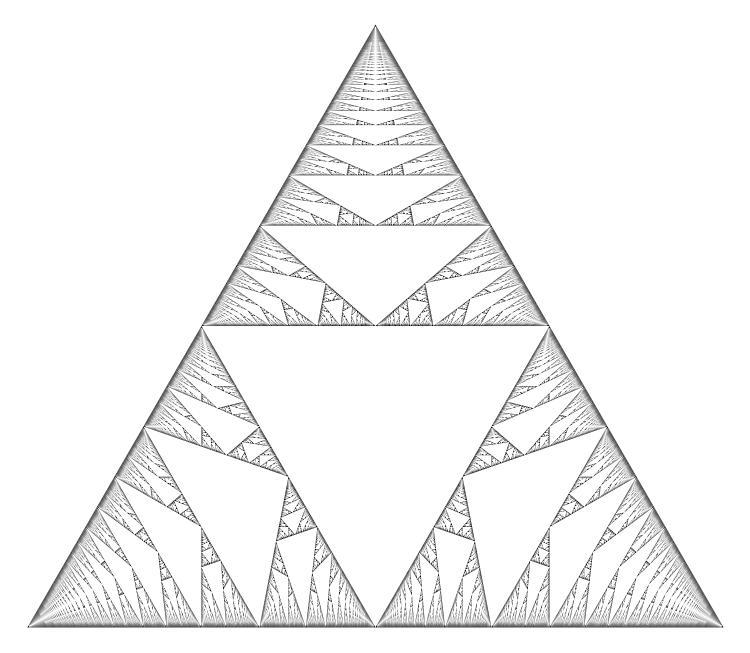}};
		\node at (pic.south east) {[(0,1,0)]};
		\node at (pic.south west) {[(1,0,0)]};
		\node at (pic.north) {[(0,0,1)]};
	\end{tikzpicture}
	\caption{Rauzy gasket $X \subset \P\RR_+^2$} \label{fig:rauzy:gasket}
\end{figure}

For this set $S$ of substitutions, the Dumont-Thomas alphabet is
$\Sigma = \{0, e_0, e_1, e_2\}$,
and the automaton $\A$ is depicted in Figure~\ref{fig:apa:arnoux:rauzy}.
	
\begin{figure}[ht]
    \centering
	\begin{tikzpicture}[scale=1, shorten >=1pt,node distance=2cm,auto] 
	   \node[state] (q_1) at (2, 3.46410161513775) {$1$}; 
	   \node[state] (q_0) at (0, 0) {$0$}; 
	   \node[state] (q_2) at (4, 0) {$2$};
		\path[->]
		(q_0) edge  [in=140, out=170, loop] node {$0, ar_0$} (q_0)
		      edge  [in=210, out=240, loop] node {$0, ar_1$} (q_0)
		      edge  [in=280, out=310, loop] node[below] {$0, ar_2$} (q_0)
		      edge [bend left=20] node {$e_0, ar_1$} (q_1)
		      edge [bend left=20] node {$e_0, ar_2$} (q_2)
		(q_1) edge  [in=5, out=35, loop] node {$0, ar_0$} (q_1)
		      edge  [in=75, out=105, loop] node {$0, ar_1$} (q_1)
		      edge  [in=145, out=175, loop] node {$0, ar_2$} (q_1)
		      edge [bend left=16] node[above] {$e_1, ar_0$} (q_0)
		      edge [bend left=20] node {$e_1, ar_2$} (q_2)
		(q_2) edge  [in=230, out=260, loop] node[below] {$0, ar_0$} (q_2)
		      edge  [in=300, out=330, loop] node {$0, ar_1$} (q_2)
		      edge  [in=370, out=400, loop] node {$0, ar_2$} (q_2)
		      edge [bend left=20] node[above] {$e_2, ar_1$} (q_1)
		      edge [bend left=20] node {$e_2, ar_0$} (q_0);
	\end{tikzpicture}
	\caption{Abelianized prefix automaton $\A$ for the Arnoux-Rauzy substitutions}
	\label{fig:apa:arnoux:rauzy}
\end{figure}
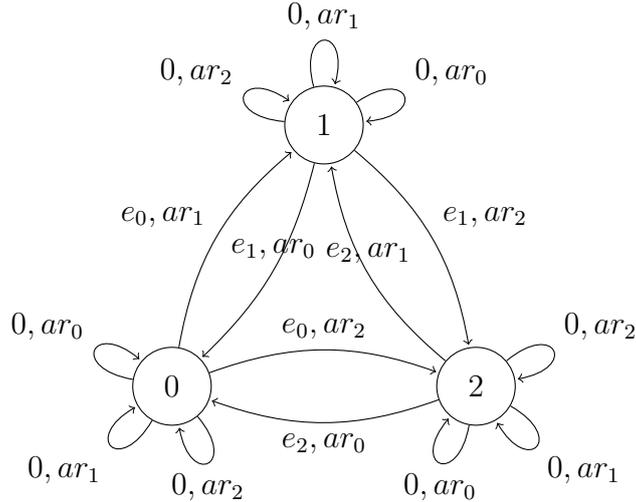

This algorithm has been well studied, see~\cite{Arn.Rau.91,Avi.Hub.Skrip.16}. 
In \cite{Avil.Dele.19}, some sufficient conditions for a measured continued fraction algorithm to satisfy Pisot condition are given. One of these conditions is independent of the ergodic measure.  It is called Pisot property. They prove that Pisot property is satisfied for this algorithm.

It appears that this algorithm has a lot of ergodic measures.
One of them has been introduced in \cite{Avi.Hub.Skrip.Inv}.
And this measure is a good candidate, but we haven't check that it fulfils all the hypotheses needed in Definition~\ref{def:meas:cont:frac}.

\section{Application: Cassaigne algorithm and two-dimen\-sional translations}\label{sec-cassaigne}
First we define the Cassaigne extended measured continued fraction algorithm, and show that it fulfills the hypotheses of Theorem~\ref{thmA}. Then we will prove Theorem~\ref{thmB}.

\subsection{Description of the algorithm}
The algorithm is defined on the whole $X=\P\RR_+^{3}$.
Let $\mu$ be the measure on $X$ with density $\frac{1}{(1-x_0)(1-x_2)}$
with respect to the Lebesgue measure on $X$.
Let $S=\{c_0,c_1\}$, where
$$
    c_0 = \begin{cases}
        0\mapsto 0\\
        1\mapsto 02\\
        2\mapsto 1
    \end{cases} ,
    \ \ \
    c_1 = \begin{cases}
        0\mapsto 1\\
        1\mapsto 02\\
        2\mapsto 2
    \end{cases}
$$

\nomenclature[L]{$c_0,c_1$}{Cassaigne substitutions}

We define the extended continued fraction algorithm as:
$$s_0 = \map{X}{S}{\left[(x_0, x_1, x_2)\right]}{\begin{cases}c_0\quad \text{if } x_0 \geq x_2\\
c_1 \quad \text{if } x_0 < x_2\end{cases}}$$

The associated continued fraction algorithm is:
$$F = \map{X}{X}{\left[(x_0, x_1, x_2)\right]}{\begin{cases}\left[(x_0-x_2, x_2,x_1)\right]\quad \text{if } x_0 \geq x_2\\
\left[(x_1,x_0, x_2-x_0)\right] \quad \text{if } x_0 < x_2\end{cases}}$$

The matrices associated to the substitutions $c_0$ and $c_1$ are:
$$M_0=\begin{pmatrix} 1&1&0\\ 0&0&1\\ 0&1&0\end{pmatrix} \ \ \ M_1=\begin{pmatrix} 0&1&0\\ 1&0&0\\ 0&1&1\end{pmatrix}.$$

For this set $S$ of substitutions, the Dumont-Thomas alphabet is
$\Sigma = \{0, e_0\}$,
and the automaton $\A$ is depicted in Figure~\ref{fig:apa:cassaigne}.

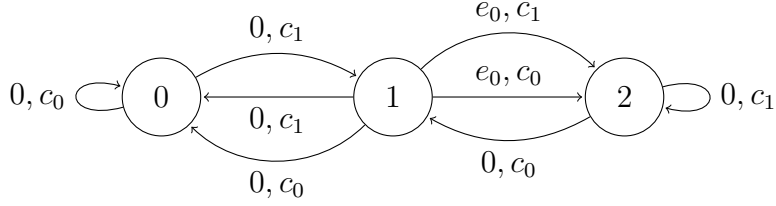
\begin{figure}[ht]
    \centering
	\begin{tikzpicture}[shorten >=1pt,node distance=2cm,auto] 
	   \node[state] (q_1)   {$1$}; 
	   \node[state] (q_0) [left=of q_1] {$0$}; 
	   \node[state] (q_2) [right=of q_1] {$2$};
		\path[->]
		
		(q_0) edge  [loop left] node {$0, c_0$} (q_0)
		      edge  [bend left] node {$0, c_1$} (q_1)
		(q_1) edge  node  {$0, c_1$} (q_0)
		      edge  [bend left=45] node  {$0, c_0$} (q_0)
		      edge node {$e_0, c_0$} (q_2)
		      edge [bend left=45] node {$e_0, c_1$} (q_2)
		(q_2) edge  [loop right] node {$0, c_1$} (q_2) 
		      edge [bend left] node {$0, c_0$} (q_1);
	\end{tikzpicture}
	\caption{Abelianized prefix automaton $\A$ for $S=\{c_0, c_1\}$} 
	\label{fig:apa:cassaigne}
\end{figure}

\begin{lemma}\label{lem-lyap-bst}
$(X,s_0,\mu)$ is an extended measured continued fraction algorithm and satisfies
the Pisot condition.
\end{lemma}
\begin{proof}
We refer to \cite{Arn.Labbe.15} for a proof of the $F$-invariance
of the measure $\mu$.
By \cite{Schwei.00} we know that Selmer algorithm is ergodic.
Moreover we know that Cassaigne algorithm is conjugated to Selmer algorithm
\cite{Cass.Lab.17}, thus we deduce the ergodicity of this measure.
It is well known that for the Selmer algorithm the second Lyapunov exponent is strictly negative,
with $x \mapsto \codim(E_2(x))$ $\mu$-almost surely constant to $1$, see \cite{Nakai.06}.
Thus by conjugation we deduce $\theta_2(F,\mu) < 0$ and $\codim(E_2(x)) = 1$ for $\mu$-almost every $x \in X$.
This algorithm fulfills the condition of Proposition~\ref{prop:cf}
since $\mu$ is absolutely continuous with respect to Lebesgue.
Hence $(X,s_0,\mu)$ is an extended measured continued fraction algorithm.
\end{proof}

\begin{figure}[ht]
    \centering
    \begin{tikzpicture}
    	\def\vert{-4.5cm};
		\node at (.6cm,0) {\includegraphics[width=2.81126206827164cm]{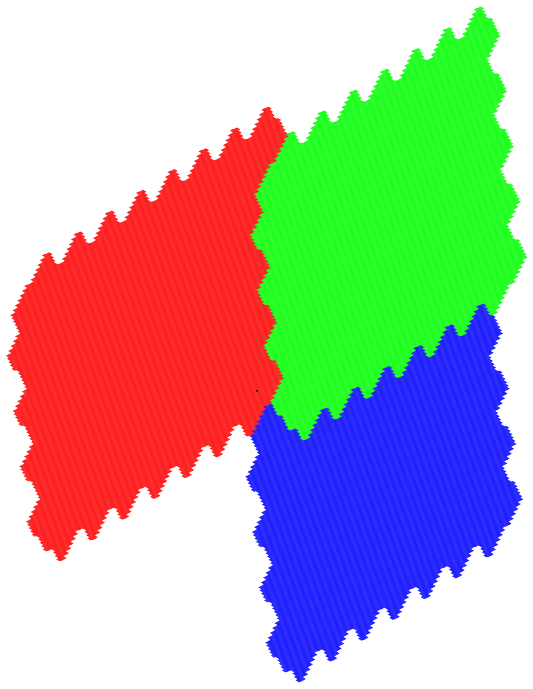}};
		\node at (4cm,0) {\includegraphics[width=2.98883149337768cm]{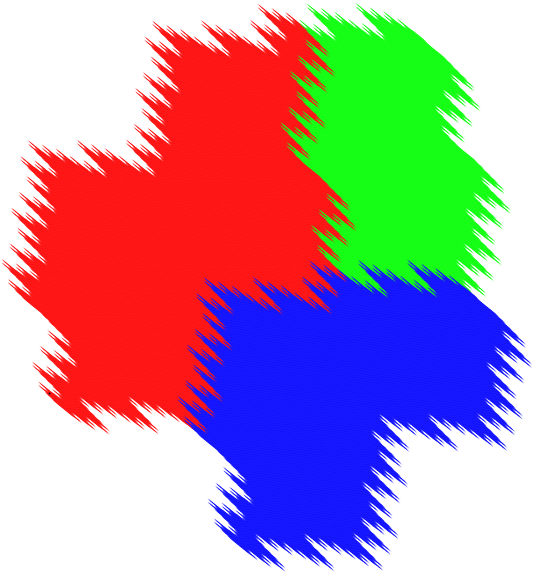}};
		\node at (8cm,0) {\includegraphics[width=4.03884036254884cm]{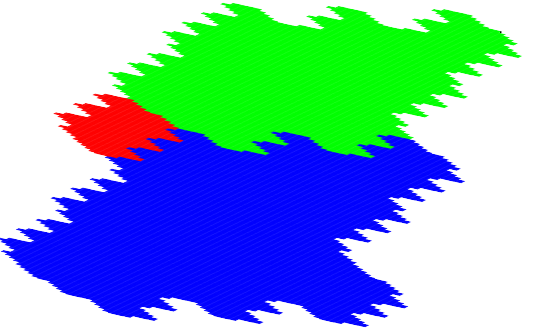}};
		\node at (12.2cm,0) {\includegraphics[width=4.91417584133148cm]{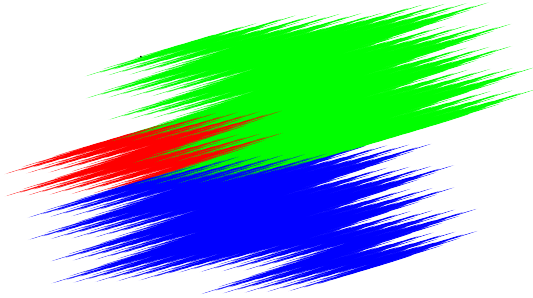}};
		
		\node at (.6cm,\vert) {\includegraphics[width=2.81126206827164cm]{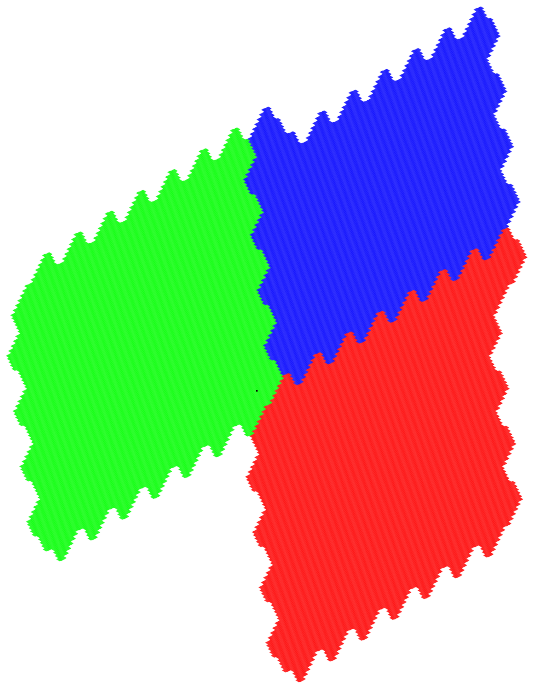}};
		\node at (4cm,\vert) {\includegraphics[width=2.98883149337768cm]{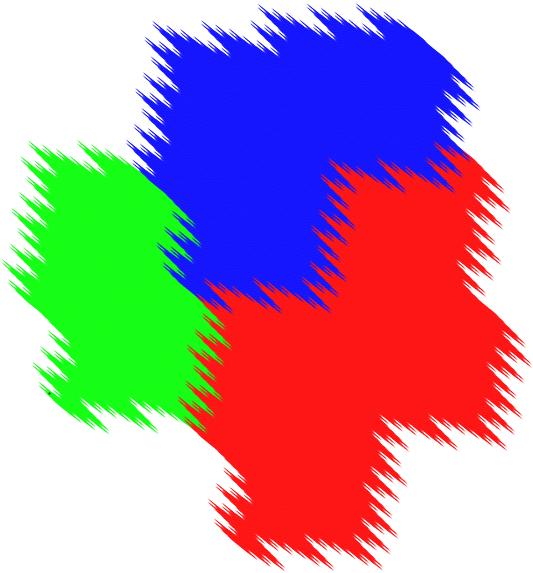}};
		\node at (8cm,\vert) {\includegraphics[width=4.03884036254884cm]{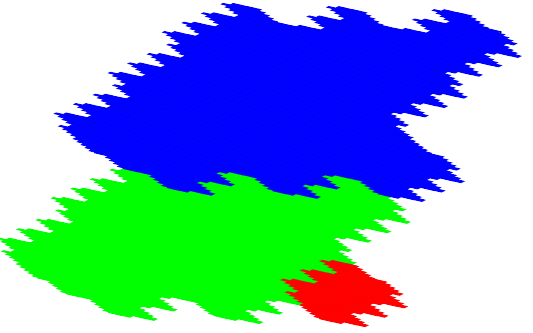}};
		\node at (12.2cm,\vert) {\includegraphics[width=4.91417584133148cm]{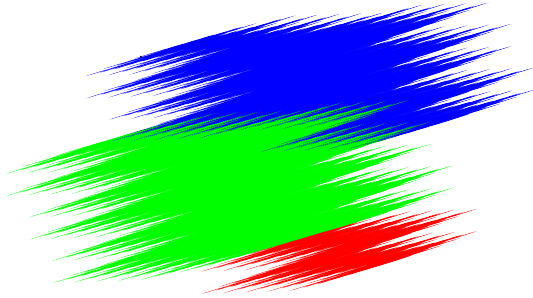}};
		
		\draw[->, thick] (.6cm, \vert*2/5-.2cm) -- (.6cm, 3*\vert/5-.2cm) node[midway, left] {$E$};
        \draw[->, thick] (4cm, \vert*2/5) -- (4cm, 3*\vert/5) node[midway, left] {$E$};
        \draw[->, thick] (8cm, \vert*2/5) -- (8cm, 3*\vert/5) node[midway, left] {$E$};
        \draw[->, thick] (12.2cm, \vert*2/5) -- (12.2cm, 3*\vert/5) node[midway, left] {$E$};
	\end{tikzpicture}
	
	\caption{Approximations of random Rauzy fractals, and the associated domain exchanges}
	\label{fig:ex:fractals}
\end{figure}
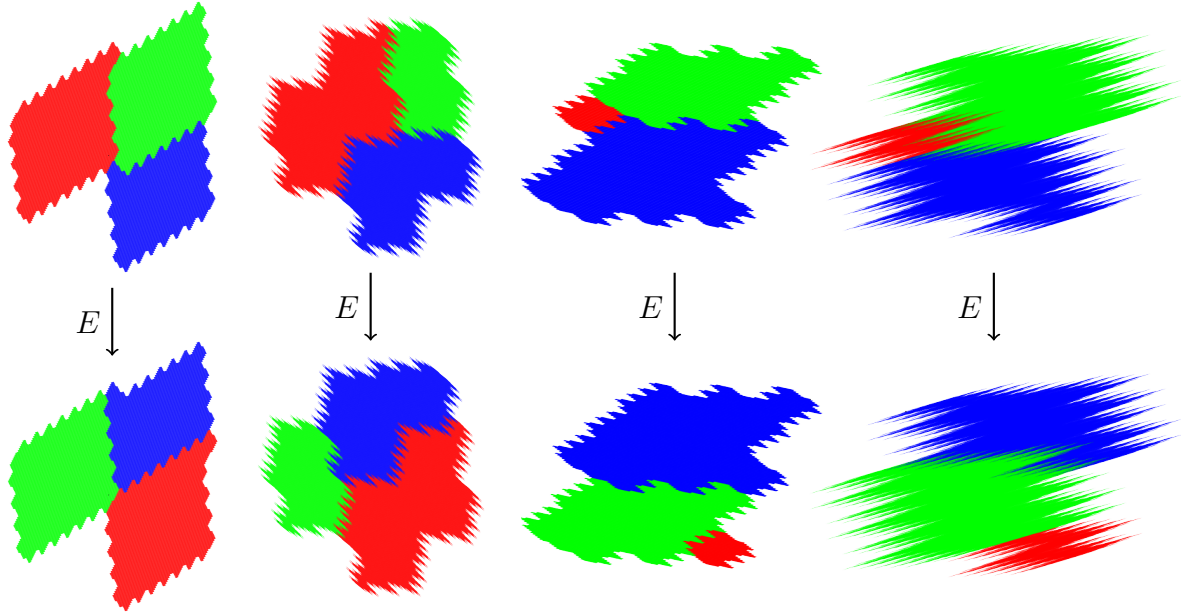

The Figure~\ref{fig:ex:fractals} illustrates approximations of Rauzy fractals $R(x)$ obtained by choosing points $x \in \P\RR_+^d$ randomly for the Lebesgue measure and applying the Cassaigne algorithm to compute the directive sequence up to a certain integer $n$. We plot the set of points
\[
	\{\sum_{k=0}^n \pi_x(M_{[0,k)} t_k) \mid b \xrightarrow{t_n, s_n} ... \xrightarrow{t_0, s_0} a \}
\]
with a color depending on the letter $a$.

If the Rauzy fractal $R_{n+1}(x)$ associated to the point $F^{n+1}(x)$ is bounded and not too large (which occurs with high probability), then the Hausdorff distance between the approximation and the Rauzy fractal is at most some pixels. Rauzy fractals of this article have been drawn using the Sage mathematical software and the badic package. These are available here:
\url{www.sagemath.org} and \url{https://gitlab.com/mercatp/badic}.

\subsection{There exists a seed point}

We consider the substitution   $c_0 c_1 = \begin{cases}
        0 \mapsto 02\\
        1 \mapsto 01\\
        2 \mapsto 1
    \end{cases}$.
    We denote by $u = (c_0c_1)^\omega(0) \in A^\NN$ its unique fixed point.
    Its abelianization is the matrix $M = \left(\begin{array}{rrr}
    1 & 1 & 0 \\
    0 & 1 & 1 \\
    1 & 0 & 0
    \end{array}\right)$. 

    This matrix has for characteristic polynomial $X^3-2X^2+X-1$. It is an irreducible polynomial over $\QQ$, with one eigenvalue greater than $1$ and two other complex eigenvalues of modulus less than $1$. We take $\beta$ the one with negative imaginary part. 
\nomenclature[G]{$\beta$}{complex number}

Thus this substitution is Pisot unimodular.
Let $x_0 \in \P\RR_+^d$ be the class of a Perron eigenvector of $\ab(c_0c_1)$.
The goal is to prove that $x_0$ is a seed point:

\begin{proposition} \label{lem:G:non-empty}
	The point $x_0$
	is in $G_0$ (see Definition~\ref{def:G0}). 
\end{proposition}

We need to prove several lemmas first. 

\begin{lemma} \label{lem:totally:irrational}
  The point $x_0$
  is a totally irrational direction. The map $F^n$ is continuous at $x_0$ for every $n \in \NN$.
\end{lemma}

\begin{proof}
    The characteristic polynomial $X^{3} - 2X^{2} + X - 1$ is irreducible over $\QQ$ and splits with simple roots over the splitting field. Thus the Galois group acts transitively on the eigenvectors. Hence, if $x_0$ was not a totally irrational direction, it would give a rational non-zero vector in the left kernel of the matrix of eigenvectors. And this is absurd because this matrix is invertible. We deduce that $x_0$ is a totally irrational direction, thus we have the continuity of $F^n$ at $x_0$ for every $n \in \NN$.
 \end{proof}

Here, there is a natural projection on $\CC$ for which $M$ acts by multiplication by $\beta$:

\begin{lemma}\label{lem-chgtbase}
	Consider the linear map $\phi$\nomenclature[G$\varphi$]{$\phi$}{linear map from $\RR^3$ to $\CC$} from $\RR^3$ to $\CC$  given by $\phi(v) = e \cdot v$, for the line vector $e = (1, \beta^2-\beta, \beta-1)$.
	This map induces a bijection between $P$ and $\CC$.
	For every $v \in \RR^3$, we have $\phi(M v) = \beta \phi (v)$ and $\phi (\pi_{x_0} v) = \phi (v)$.
\end{lemma}

\begin{proof}
	Remark that the line vector $e$ is a left-eigenvector of $M$ for the eigenvalue $\beta$.
	Let $v \in \RR^3$.
	We have $\phi(v) = e \cdot v$, so we have $\phi( M v ) = e M v = \beta e v = \beta \phi(v)$.
	And $x_0$ is the class of a right eigenvector of $M$ for an eigenvalue different of $\beta$, so we have $e x_0 = 0$, thus we get $\phi (\pi_{x_0} v) = \phi (v - h(v)v(x_0)) = \phi (v)$.
	Now we check that the rank of $\phi$ is $2$, so its kernel is the vector space spanned by $x_0$, which intersect $P$ only at $0$.
	Thus $\phi$ induces a bijection between $P$ and $\CC$.
\end{proof}

With Lemmas~\ref{dumont:thomas} and~\ref{lem-chgtbase}, we can project 
the worm $W(u)$ on the complex plane 
\[
    \phi(W_a(u)) = \{ \sum_{k=0}^{n-1} t_k \beta^k \mid 0 \xrightarrow{t_{n-1}} ... \xrightarrow{t_{0}} a \} \subset \CC,
\]
where $0 \xrightarrow{t_{n-1}} ... \xrightarrow{t_{0}} a$ denotes a path in the automaton $\phi(\A)$ of the Figure~\ref{fig:aut:c0c1} (we do not label the edges by the susbtitution since there is only one in this case). $\phi(\A)$ is the image by $\phi$ of the abelianized prefix automaton $\A$ for the substitution $c_0c_1$.

\begin{figure}
    \centering
    \includegraphics[scale=.5]{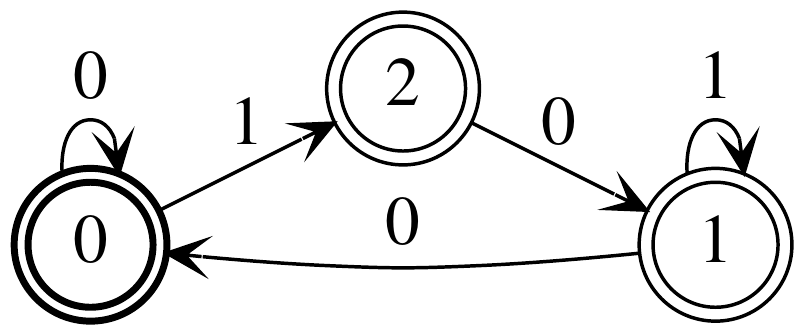}
    \caption{Automaton describing the projection by $\phi$ of the worm $W(u)$} \label{fig:aut:c0c1}
\end{figure}

The Figure~\ref{fig:c0c1:gasket} shows the Rauzy fractal $R$ of the directive sequence $(c_0c_1)^\omega$ and its image $\phi(R)$ by $\phi$.

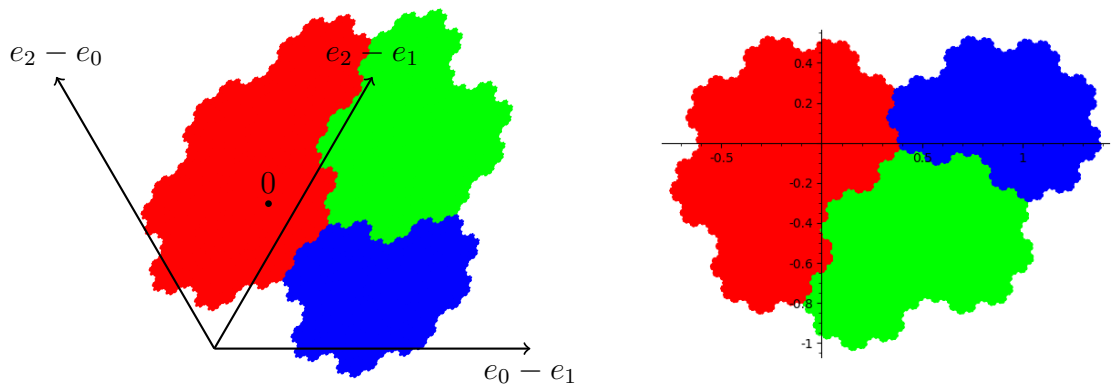
\begin{figure}[!ht]
    \centering
    \begin{tikzpicture}[scale=2.94114425143390]
		\node at (.5cm,.7cm) {\includegraphics[width=5cm]{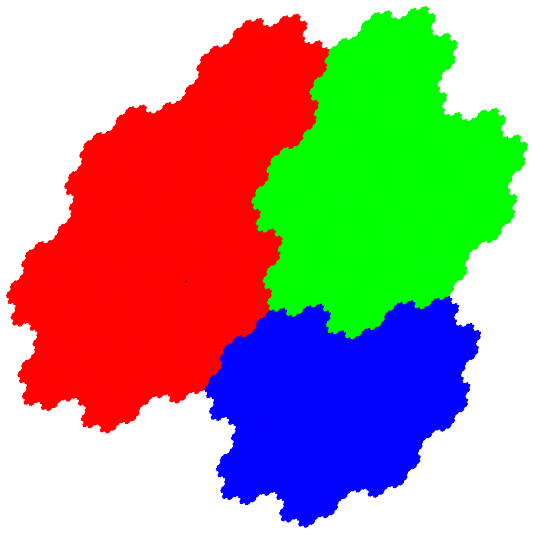}};
		\node at (3cm, .7cm) {\includegraphics[width=6.2cm]{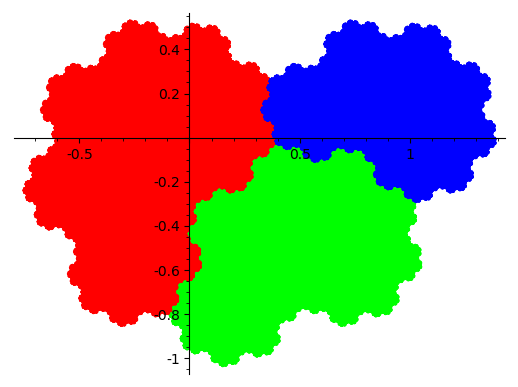}};
		\draw[->, thick] (0,0) -- (1.414213562373095cm, 0) node[below] {$e_0-e_1$};
		\draw[->, thick] (0,0) -- (-0.7071067811865475cm, 1.2247448713915894cm) node[above] {$e_2-e_0$};
		\draw[->, thick] (0,0) -- (0.7071067811865475cm, 1.2247448713915894cm) node[above] {$e_2-e_1$};
		\fill (.5cm-0.257755905389784cm, .7cm-0.0465014874935169cm) circle (.015cm) node[above] {$0$};
	\end{tikzpicture}
	\caption{Rauzy fractal of $(c_0c_1)^\omega$ in $P$ (left) and its projection by $\phi$ on $\CC$ (right)} \label{fig:c0c1:gasket}
\end{figure}

The following lemma permits to find good bounding boxes (for example good disks) that contain the parts of the Rauzy fractal.

\begin{lemma} \label{lem:bounding:sets}
    If there exists $(O_a)_{a \in A}$ open subsets of $\CC$ and an integer $n \in \NN$ such that for every $a \in A$ we have
    \[
        \bigcup_{b \xrightarrow{t_{n-1}} ... \xrightarrow{t_0} a} \big(\beta^n \overline{O_b} + \sum_{k=0}^{n-1} t_k \beta^k\big) \subseteq O_a,
    \]
    then for all $a \in A$, we have $\phi(R_a) \subseteq O_a$.
\end{lemma}

\begin{proof}
    Let $y \in \phi(R_a)$. By Corollary~\ref{cor:bounded:cover}, there exists an infinite path $... \xrightarrow{t_n} ... \xrightarrow{t_0} a$ in the automaton $\phi(\A)$ of Figure~\ref{fig:aut:c0c1} such that
    \[
        y = \sum_{k=0}^\infty t_k \beta^k.
    \]
    Let $b \in A$ such that $b \xrightarrow{t_{n-1}} ... \xrightarrow{t_0} a$ is a path in the automaton.
    Let $l> 0$ be the distance between $\beta^n \overline{O_b} + \sum_{k=0}^{n-1} t_k \beta^k$ and the complement of $O_a$.
    We denote by $D$\nomenclature[L]{$D$}{distance on $\CC$} the usual distance on $\CC$.
    Let $k \in \NN$ be large enough such that $\abs{\beta^{kn}} \max_{c \in A} \sup_{z \in \phi(R)} D(z, \overline{O_c}) < l$, and let $c \in A$ such that $c \xrightarrow{t_{kn-1}} ... \xrightarrow{t_n} b \xrightarrow{t_{n-1}} ... \xrightarrow{t_0} a$ is a path in the automaton.
    We have $\displaystyle\sum_{p=kn}^\infty t_p \beta^{p-kn} \in \phi(R)$.  So we have
    $
        D(\sum_{p=kn}^\infty t_p \beta^{p}, \beta^{kn} \overline{O_c}) < l.
    $
    
    Thus, we have
    $
        D(y, \beta^{nk} \overline{O_c} + \sum_{p=0}^{kn-1} t_p \beta^p) < l
    $.
    And we have the inclusion
    \[
        \beta^{nk} \overline{O_c} + \sum_{p=0}^{kn-1} t_p \beta^p
            \subseteq \beta^n \overline{O_b} + \sum_{k=0}^{n-1} t_k \beta^k
    \]
    by iterating $k-1$ times the inclusion of the hypothesis.
    So, we get that $y$ is in $O_a$.
\end{proof}

\begin{corollary} \label{cor:balls}
For the Rauzy fractal associated to ${(c_0c_1)}^\omega$ we have the inclusions
\begin{align*}
    \phi(R_0) &\subseteq B(-0.19-0.15i, 0.75) =: O_0, \\
    \phi(R_1) &\subseteq B(0.5 - 0.6i, 0.655) =: O_1, \\
    \phi(R_2) &\subseteq B(0.865 + 0.123i, 0.566) =: O_2.
\end{align*}
\end{corollary}
\begin{proof}
We use Lemma~\ref{lem:bounding:sets} for $n=8$, and check the result by computer, see Figure~\ref{fig:bballs}.
\end{proof}

In the following we denote $z_a, r_a$ the center and the radius of the ball $O_a$ for $a=0,1,2$.

\begin{figure}
    \centering
    \includegraphics[scale=.7]{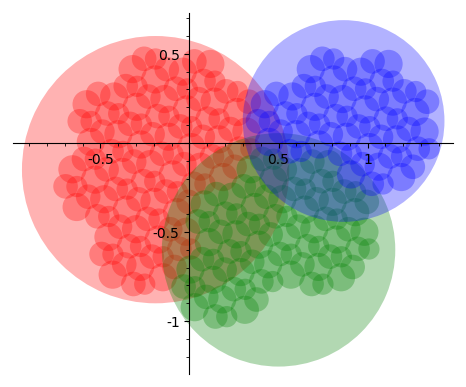}
    \caption{Bounding balls found from Lemma~\ref{lem:bounding:sets} for $n=8$} \label{fig:bballs}
\end{figure}

Now we can prove Proposition~\ref{lem:G:non-empty}: 

\begin{proof}[Proof of Proposition~\ref{lem:G:non-empty}]
We check all the conditions that show that $x_0$ is a seed point:
\begin{itemize}
\item By Lemma~\ref{lem:totally:irrational}, the direction $x_0$ is totally irrational and for all $n \in \NN$, $F^n$ is continuous at $x_0$.

\item 
We have
\[
	\lim_{n \to \infty} \frac{1}{n} \ln \matrixnorm{\pi_{x_0} M_{[0,n)}(x_0)}_1 = \lim_{n \to \infty} \frac{1}{2n} \ln \matrixnorm{\pi_{x_0} {\ab(c_0 c_1)}^n}_1 = \frac{1}{2} \ln \abs{\beta} < 0.
\]

\item We have $0 \not\in R_1 \cup R_2$ by Corollary~\ref{cor:balls} since $0 \not\in B(0.5 - 0.6i, 0.655) \cup B(0.865 + 0.123i, 0.566)$.

\item 
For $t \in \Lambda \setminus \{0, e_1-e_2, e_2-e_1\}$, we check that
we have $\abs{\phi(t)} > 1.5 > \max_{a \in A} r_a + \abs{z_a}$, so by Corollary~\ref{cor:balls}, we get that $0 \not\in R + t$.

For $t \in \{e_1-e_2, e_2-e_1\}$,
we check that we have for all $a \in A$, $\abs{z_a + \phi(t)} > r_a$, thus we have $0 \not\in R+t$.

\item We have $0 \not\in R_1 \cup R_2 \cup \bigcup_{t \in \Lambda \setminus \{0\}} R + t$, so by Lemma~\ref{lem:car:interior} we have that $0$ is in the interior of $W_0(u)$. 
	In particular, the interior of $W_0(u)$ is non-empty. 
	Furthermore, there exists a fixed point $w \in {(A^\NN)}^\NN$ of the directive sequence $s(x_0) = (c_0c_1)^\omega$ such that $w_0$ is the fixed point $u$ of the substitution $c_0c_1$.   
\end{itemize}
\end{proof}

\begin{figure}
        \centering
        \includegraphics[scale=.7]{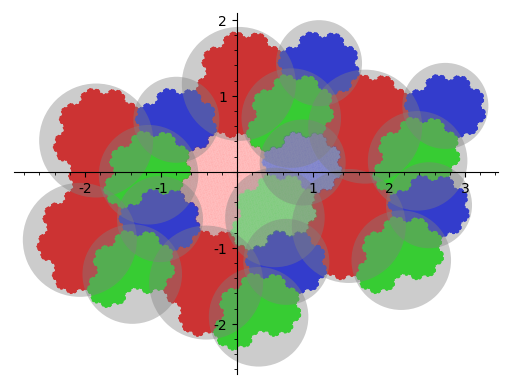}
        \caption{Proof that $0 \not\in R_1 \cup R_2 \cup \bigcup_{t \in \Lambda \setminus \{0\}} R + t$ thanks to covering with balls} \label{fig:balls:disjoint}
\end{figure}

Remark that the previous proof could be adapted to prove the semi-decidability of being a seed point for $F$-periodic points.
But we have even the decidability.

\begin{proposition}\label{prop-decidable}
	If $F$ is computable, then being a seed point is a decidable property for $F$-periodic points.
\end{proposition}

\begin{proof}
	Let $x$ be an $F$-periodic point of period $p$.
	Total irrationality of the direction $x$ is equivalent to irreducibility of the characteristic polynomial of $M_{[0,p)}(x)$, and this can be checked algorithmically.
	The fact that $x$ is not a discontinuity point of $F^n$ can be checked since $F$ is computable, and it is enough to test it for $n \leq p$.
	The hypothesis that $\limsup_{n \to \infty} \norm{\pi_x M_{[0,n)}(x)}_1 < 0$ is equivalent to check that the matrix $M_{[0,p)}$ is Pisot, and this is decidable.
	Then, if $u$ is a fixed point of $s(x)$, then $u_0$ is a periodic point for the substitution $s_{[0,p)}$, see Lemma~\ref{lem-lien-pt-fixes}.
	Then, we use~\cite[Theorem~5.12]{Aki.Merc.18}.
	This theorem allows to describe the interior of $W_a(u_0)$ with a finite automaton, such that the interior is empty if, and only if, the language of the automaton is empty.
	Moreover this automaton is computable from $s_{[0,p)}$.
	And checking if an automaton has an empty language is decidable.
	The computation in~\cite{Aki.Merc.18} is done for the bi-infinite topology, but it is possible to use it to compute the interior for the topology $\topo{x_0}$, by adding a left infinite part to our worm.
\end{proof}

\subsection{Proof of Theorem~\ref{thmB}}\label{proof-thmB}

We refer to \cite[Proposition 6]{Cass.Lab.17} for the proof of the following result:
\begin{lemma}\label{lem-comp-cll}
Consider a directive sequence $s$ in $S^{\NN}$, where $S = \{c_0, c_1\}$.
Assume that $s$ cannot be written as a finite sequence followed by an infinite concatenation of $c_0^2$ and $c_1^2$.
Then $\Omega_s$ is minimal and has complexity $2n+1$.
\end{lemma}

Now we deduce the proof of Theorem~\ref{thmB}:
With Lemmas~\ref{lem-lyap-bst} and~\ref{lem:G:non-empty} we can apply Theorem~\ref{thmA} since $\mu$ is absolutely continuous with respect to the Lebesgue measure.

The map $x\mapsto t_x$ of Theorem~\ref{thmA} is the map $x\mapsto \psi(e_0-v(x))$, where $\psi\colon P/\Lambda \to \TT^2$ is an isomorphism. Now remark that $(e_0-\Delta)\cup (\Delta-e_0)$ form a cover of a measurable fundamental domain of $P$ for the action of $\Lambda$. Thus the set $\{\psi(e_0-v(x))\mid x\in G\}\cup \{\psi(v(x)-e_0)\mid x\in G\}\ $ is of full measure in $\TT^2$.
Hence, we get for Lebesgue-almost every translation of $\TT^2$ a nice generating partition whose symbolic coding is conjugate to the subshift. With Lemma~\ref{lem-comp-cll} we deduce the result.

\section{Renormalization schemes}\label{sec-renormalization}

In general, the first return map of a minimal torus translation on a bounded
remainder set is close to be a torus translation \cite{Ferenczi1992}.
In the present case, we will see how selecting the atoms on which to induce
explicitly leads to another torus translation, and how the induction process
relates to the continued fraction algorithm.

We will focus on Cassaigne algorithm, and then explain how to adapt the
reasoning to other algorithms.

Let us first look at the symbolic level. Let us consider a directive sequence
$s=(s_k)$ starting with $c_0$, and let $u$ be one of its fixed points.

The word
$u_0 = c_0(u_1)$ is the concatenation of the three finite words $0$, $02$ and
$1$.
Those three words are \emph{return words} on the pair $\{0,1\}$, i.e. any word
in $\Omega_u$ starting with $0$ or $1$ can be written in a unique way as a
concatenation of $0$, $02$ and $1$, and $0$ and $1$ appear only at the first
positions of those words.
Hence, inducing the subshift $\Omega_{u_0}$ on the clopen set $[0]\cup[1] =
\Omega_{u_0}\setminus [2]$ leads to a subshift isomorphic to $\Omega_{u_1}$,
whose directive sequence is $(s_{k+1})$.

Now, assume that $s=(s_k)$ starts with $c_1$, and again let $u$ be one of its
fixed points.
In this case, the images of the letters by $c_1$ are not return words, and we
have to look backwards: the reverse of the images of the letters by $c_1$, that
is $1$, $20$ and $2$ are return words on the pair $\{1,2\}$.
An option could be to reverse $c_1(1)$ in the definition of $c_1$ to be $20$ as
it will not change the continued fraction algorithm, but it will increase the
complexity of the associated subshift, which we can not afford.
Instead, we remark that inducing on $T([1])\cup T([2]) = \Omega_u \setminus
T([0])$, where $T$ denotes the shift map, leads again to a subshift isomorphic
to $\Omega_{u_1}$, whose directive sequence is again $(s_{k+1})$.

All those remarks translate to the geometrical level, and we get the following
renormalization scheme.
To simplify the notations, 
we identify the Rauzy fractals $R_i(x)\subseteq P$ with their image by the
projection $q\colon P\to P/\Lambda$.
Let $T_x$ be a translation of the torus, and let $(R_0(x),R_1(x),R_2(x))$ be the
associated partition by Rauzy fractals.

\begin{itemize}
        \item (\emph{bottom type}) if $\lambda(R_0(x)) > \lambda(R_2(x))$, let $U = (P/\Lambda) \setminus R_2(x)$
        \item (\emph{top type}) if $\lambda(R_0(x)) < \lambda(R_2(x))$, let $U = (P/\Lambda) \setminus T_x(R_0(x))$
\end{itemize}

Then, the induced application $(T_x)_U$ is isomorphic to the translation
$T_{F(x)}$, but it is defined on a smaller torus $P/\Lambda'$, with
$U$ being a measurable fundamental domain of $P$ for the action of a lattice $\Lambda'$.
The induced application $(T_x)_U$ can be renormalized to the translation
$T_{F(x)}$ on the reference torus $P/\Lambda$: the linear map that sends $U$ to
$R(F(x))$ is $N^{-1}$, where $N$ is the linear endomorphism of $P$ such that $N \circ \pi_{F x} = \pi_x \circ \ab(s_0(x))$
that was introduced in subsection \ref{proof-ThC-step2}. And we have the relation $\Lambda' = N\Lambda$.

It is remakable to see that this scheme is pretty similar to the famous
Rauzy-Veech induction for interval exchange maps \cite{Rauzy1979} (we named the
\emph{top} and \emph{bottom} types after the naming scheme of
\cite{Yoccoz2010}).

\begin{figure}[!h]
    \centering
	\begin{tikzpicture}[scale=1]
        \def\vert{-4.5cm}
        \def\horiz{6cm}
		\node (pic11) at (0,0) {\includegraphics[width=3.34175343990326cm]{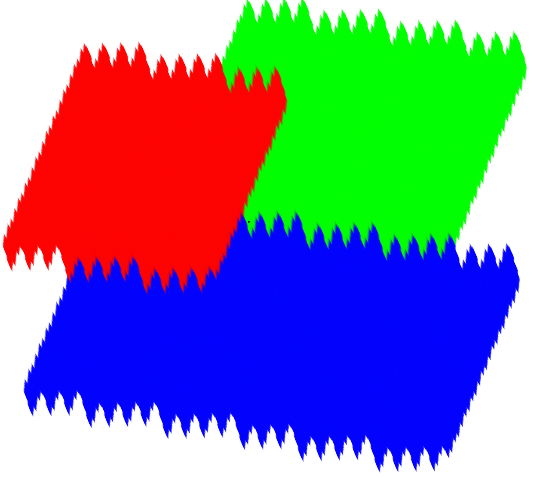}};
		\node (pic12) at (\horiz, 0) {\includegraphics[width=3.28993709564208cm]{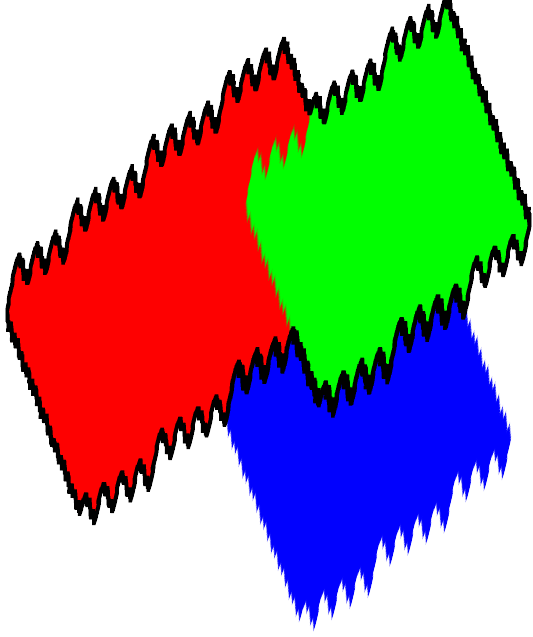}};
		\node (pic13) at (2*\horiz, 0) {\includegraphics[width=4.12394912719726cm]{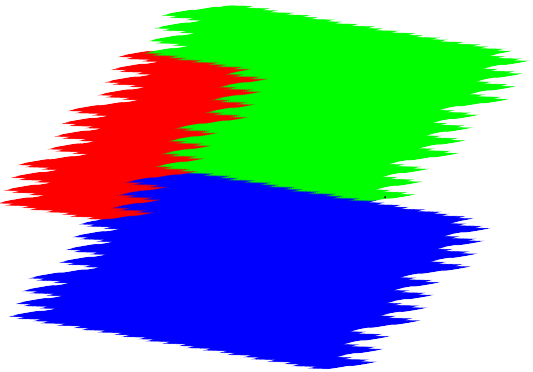}};
		\node (pic21) at (0, \vert) {\includegraphics[width=3.34175343990326cm]{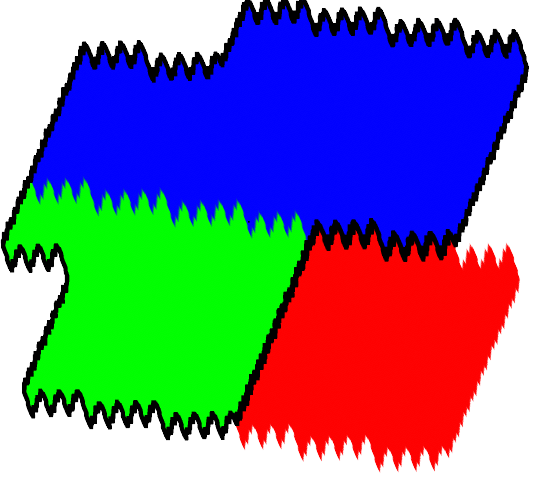}};
		\node (pic22) at (\horiz, \vert) {\includegraphics[width=3.28993709564208cm]{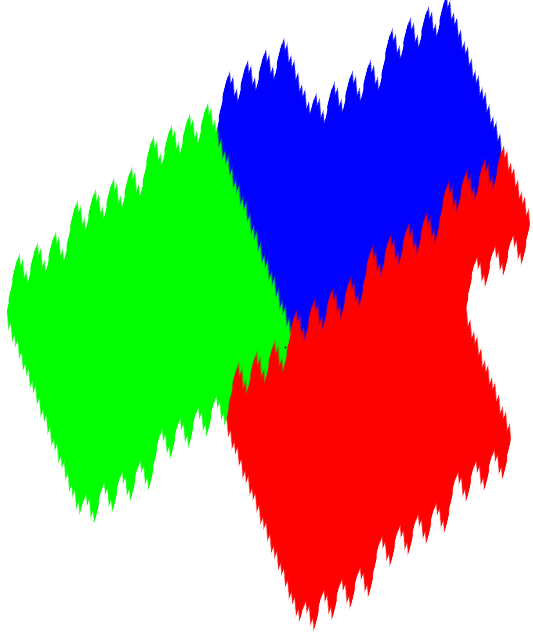}};
		\node (pic23) at (2*\horiz, \vert) {\includegraphics[width=4.12394912719726cm]{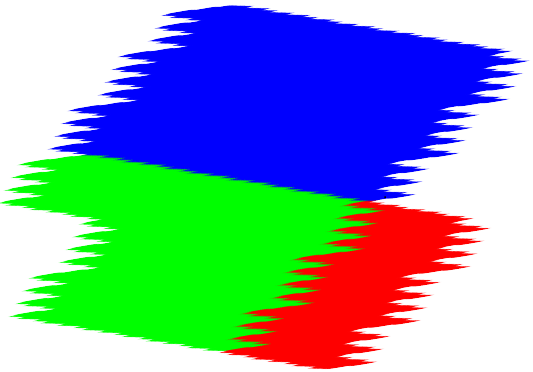}};
        \draw[->, very thick] (\horiz/3, \vert/2) -- (2*\horiz/3, \vert/2) node[midway, above] {\emph{top}};
        \draw[->, very thick] (4*\horiz/3, \vert/2) -- (5*\horiz/3, \vert/2) node[midway, above] {\emph{bottom}};

        \draw[->, thick] (0cm, \vert*2/5) -- (0cm, 3*\vert/5) node[midway, left] {$T_x$};
            \draw[->, thick] (\horiz, \vert*2/5) -- (\horiz, 3*\vert/5) node[midway, left] {$T_{F(x)}$};
            \draw[->, thick] (2*\horiz, \vert*2/5) -- (2*\horiz, 3*\vert/5) node[midway, left] {$T_{F^2(x)}$};
	\end{tikzpicture}
        \caption{Two steps of induction/renormalization. The sets of induction are outlined in black. Note that the pictures look flipped (and stretched) from one step to the next one, this is due to the fact that the renormalization matrices have negative determinant.} 
        \label{fig:renormalization}
\end{figure}
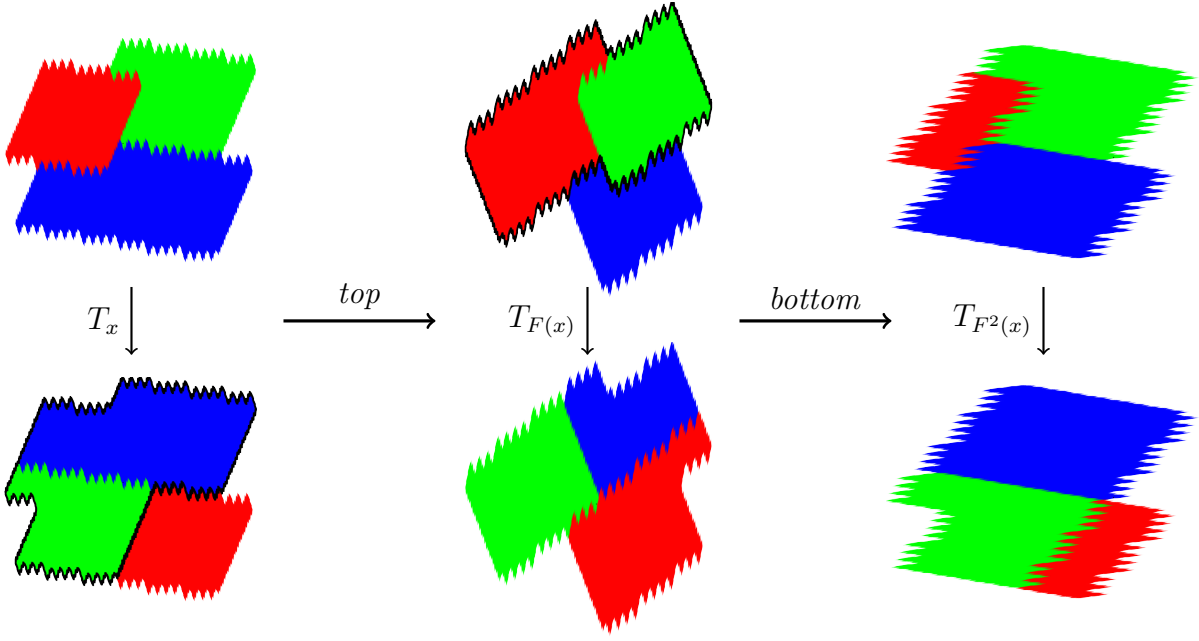

Figure~\ref{fig:renormalization} shows two steps of induction and renormalization starting from $v(x) = (0.256005715380561..., 0.286881483823029..., 0.457112800796410...)$.
The corresponding directive sequence is $s(x) = c_1c_0c_1c_0c_1c_0c_0c_0c_1c_0c_0c_0c_1c_1c_0c_0c_0c_0c_0c_0c_1c_1c_0c_1c_0c_0c_1c_1...$ 
It corresponds to the translation by
\[
	t = (0.743994284619439..., -0.286881483823029..., -0.457112800796410...)
\]
on the torus $P/\Lambda$.
The figure shows the Rauzy fractals $R(x), R(Fx)$ and $R(F^2x)$, with $R_0$ in red, $R_1$ in green and $R_2$ in blue.
The first line shows the decomposition $R = R_0 \cup R_1 \cup R_2$, and the second line shows the decomposition $R = (R_0 + \pi(e_0)) \cup (R_1 + \pi(e_1)) \cup (R_2 + \pi(e_2))$ after applying the domain exchange corresponding to the translation on the torus.

The Brun and Arnoux-Rauzy extended continued fraction algorithms also enjoy a
similar renormalization scheme:
\begin{itemize}
\item For the Brun algorithm (with the substitutions as in \cite{Labbe.15}), it
      suffice to induce on the complementary of the image of the second largest atom.
Another choice of substitutions will lead to a different renormalization scheme.
\item For the Arnoux-Rauzy algorithm, it suffice to induce on the image of the largest atom.
\end{itemize}

In all cases, one renormalization step corresponds to applying one step of the
continued fraction algorithm.

\section{Remarks and open problems}

\subsection{Comments on the results of another paper}

In \cite{Berth-Steiner-Thusw.14} the authors prove two theorems on the same subject. Their Theorem~3.1 is in the same spirit as our Theorem~\ref{thmC}. In their case, they need some hypotheses such as irreducibility and balanceness for the directive sequences or coincidence conditions on the subshift. In our theorem these conditions are not assumed, and are replaced by our notion of good directive sequence.
Theorem~3.3 of \cite{Berth-Steiner-Thusw.14} is in the same spirit as our Theorem~\ref{thmA}. Here again, the hypotheses are not on the same objects.

\subsection{Translation vectors \emph{vs} directions}

The link between a continued fraction algorithm and torus translations was done
by associating to every direction $x\in\P\RR_+^d$, a translation vector of
$\TT^d$ via the composition:
$$\chi\colon\xymatrix{\P\RR_+^d \ar[r]^{v} & \Delta \ar[r]^f & P \ar[r]^{q} &
P/\Lambda \ar[r]^{\psi} & \TT^d}.$$
The map $f=\map{\Delta}{P}{y}{\pi_y(e_0)}$ is affine (and injective), which is
why we could transport results holding for almost every direction to results
holding for almost every torus translation.
Note however that the map $\chi$ is not surjective, so that we had to project the
simplex twice to cover all possible torus translations in dimension $2$ in the
proof of Theorem \ref{thmB} (Section \ref{proof-thmB}).
If $t$ is an element of $\TT^2$, either $t$ or $-t$ is the image of some
direction $x\in\P\RR_+^2$. Since the translation $T_t$ is conjugated to the
translation $T_{-t}$, we got the result for almost every translation of $\TT^2$.

In higher dimensions, if we identify $\TT^d$ with the unit hypercube $[0,1]^d$,
the image of $\P\RR_+^d$ is the convex hull
$S_d$ of $\{0,e_1,\dots,e_{d}\}$, whose Lebesgue measure is only $1/d!$.

If $\alpha\in GL_d(\ZZ)$ is an automorphism of $\TT^d$, $T_t$ is conjugated to
$T_{\alpha(t)}$.
As shown in \cite{Freudenthal1942}, there exists an explicit finite family
$(\alpha_i)_{0\leq i< d!}$ of elements of $GL_d(\ZZ)$ and a family $(n_i)_{0\leq
i< d!}$ of elements of $\ZZ^d$ such that $[0,1]^d = \bigcup_{0\leq i< d!}
\alpha_i(S_d) + n_i$, that is,
$$\TT^d = \bigcup_{0\leq i< d!} \alpha_i(\chi(\P\RR_+^d)).$$
Such tiling is also known as \emph{Kuhn triangulation} \cite{LeeSantos2017}.

Therefore, if we want to go from a particular translation $T_t$ of $\TT^d$ to a
projective direction and study its dynamics through continued fractions, it
suffices to find to which atom $\alpha_i(S_d)$ of the triangulation it belongs,
and to associate the direction $x=\chi^{-1}(\alpha_i^{-1}(t))$ (note that $\chi$ is
injective, except on the finite set $\{[e_0],\ldots,[e_d]\}$).

\subsection{Exceptional directions in Cassaigne algorithm}
A natural question is to understand the set of directions where the conclusion of Theorem~\ref{thmB} is true.
Our proof shows that it works for a subset of measure one in the set of totally irrational directions. 
Can we extend the result of Theorem~\ref{thmB} to  all this set ? It is not possible with our proof, but maybe we can use some other continued-fraction algorithm, or some unrelated method. 
Indeed, there are subshifts defined by the Cassaigne algorithm which are not balanced \cite{Andrieu2020}, so there are directions where we can not use this algorithm to construct symbolic codings of translations of $\TT^2$.
More generally, one may ask whether some subshifts defined by the
Cassaigne algorithm are weakly mixing, see \cite{CassaigneFerencziMessaoudi2008}.

To finish with Cassaigne algorithm we list some properties of exceptional directions.
Let $x = \left[(x_0,x_1,x_2)\right]\in \P\RR_+^{3}$ be a direction and $s(x)=(s_k)$ be its associated directive sequence.
We have equivalence between $\dim_\QQ \vect(x_0,x_1,x_2) = 1$ and the fact that the sequence $(s_k)$ is ultimately constant.
Moreover $\dim_\QQ \vect(x_0,x_1,x_2) = 1$ if, and only if, $s(x)$ is not everywhere growing.
And finally the property $\dim_\QQ \vect(x_0,x_1,x_2) \leq 2$ is equivalent to
the fact that $(s_k)$ can be written as the concatenation of a finite sequence
followed by an infinite concatenation of $c_0^2$ and $c_1^2$ (even runs)
\cite[Lemma 1]{Cass.Lab.17}.

\subsection{Higher dimensions}

Another natural question is to generalize Theorem~\ref{thmB} for $d\geq 3$.
For example we could be interested in the following set $S$
$$\begin{cases}
a\mapsto a\\
b\mapsto c\\
c\mapsto d\\
d\mapsto ab
\end{cases}
\begin{cases}
a\mapsto c\\
b\mapsto b\\
c\mapsto d\\
d\mapsto ab
\end{cases}$$

It seems that the complexity of the subshift is linear, but it is bigger than $3n+1$. Moreover we do not know actually if the other hypotheses are fulfilled.

\subsection{Pisot substitution conjecture and converse of Theorem~\ref{thmC} for constant directive sequences} \label{sec:pisot}
 
 	We say that a substitution is \emph{irreducible} if its matrix has an irreducible characteristic polynomial.
 	
	The \emph{Pisot substitution conjecture} states (or is equivalent to the fact) that the conclusion of our Theorem~\ref{thmC} is true for every directive sequence of the form $\sigma^\omega$,
	with $\sigma$ an irreducible Pisot unimodular substitution: the subshift is measurably conjugated to a translation on a torus.

	But for such particular directive sequences, this is equivalent to being good.

\begin{lemma} \label{lem:pisot:periodic:good}
	Let $\sigma$ be an irreducible Pisot unimodular substitution such that for a periodic point $u \in A^\NN$,
	there exists a letter $a \in A$ such that $W_a(u)$ is not empty for the topology $\topo{x}$, where $x \in \P\RR_+^d$ is the class of a Perron eigenvector of $\ab(\sigma)$.
	Then, the directive sequence $\sigma^\omega$ is good.
\end{lemma}

\begin{proof}
	We check that $s=\sigma^\omega$ satisfies the four points of Definition~\ref{defG}.
	\begin{enumerate}
		\item We have $\lim_{n \to \infty} \frac{1}{n} \norm{\pi_x M_{[0,n)}}_1 = \lim_{n \to \infty} \frac{1}{n} \norm{\pi_x M^n}_1 = \ln(\abs{\beta})$, where $\beta$ is the second biggest eigenvalue of $M = \ab(\sigma)$ in absolute value. We have $\ln(\abs{\beta}) < 0$ since $\sigma$ is Pisot.
		\item The direction $x$ is totally irrational since $\sigma$ is irreducible.
		\item We have for all $n \in \NN$, $x^{(n)} = x$, so $\lim_{n \to \infty} x^{(n)} = x$ exists and is a totally irrational direction, so we have the fourth point.
		\item Let us check the third point.
			Let $u_0$ be a periodic point of $\sigma$ of period $p$, such that there exists $a \in A$ such that $W_a(u_0)$ has non-empty interior for $\topo{x}$,
			and let $u \in {(\A^\NN)}^\NN$ be the word sequence defined by
				\[
					u_n = \sigma^{p - (n\ mod\ p)}(u_0),
				\]
				where $(n\ mod\ p)$ is the remainder in the division of $n$ by $p$.
				We easily check that $u$ is a fixed point of the directive sequence $\sigma^\omega$.
				By Lemma~\ref{lem:primitivity} there exists $n_0 \in \NN$ such that for all $n \geq n_0$, $M_{[0,n)} = M^n > 0$.
				We choose $n \geq n_0$ divisible by $p$. Hence we have $u_n = u_0$.
				Now, for every $b \in A$, we have the equality
				\[
					W_b(u_0) = \bigcup_{c \xrightarrow{t_{n-1}, s_{n-1}} ... \xrightarrow{t_0, s_0} b} M^n W_c(u_0) + \sum_{k=0}^{n-1} M^k t_k.
				\]
				By Lemma~\ref{lem:O:matrix}, $W_b(u_0)$ has non-empty interior for all $b \in A$.
				Hence, we get the third point, with the sequence $(k_n)_{n \in \NN} = (n)_{n \in \NN}$.
	\end{enumerate}
\end{proof}

In~\cite[Theorem~3.3]{Aki.Merc.18}, they prove the following

\begin{theorem}
	Let $\sigma$ be an irreducible Pisot unimodular substitution.
	Then we have the equivalence between
	\begin{itemize}
		\item $\sigma$ satisfies the Pisot substitution conjecture,
		\item there exists a periodic point $u \in A^\NN$ and a letter $a \in A$ such that $W_a(u)$ has non-empty interior for $\topo{x}$,
	\end{itemize}
	where $x \in \P\RR_+^d$ is the class of a Perron eigenvector of $\ab(\sigma)$.
\end{theorem}

From this lemma and this theorem we deduce the following

\begin{corollary}
	The converse of Theorem~\ref{thmC} is true for directive sequences of the form $\sigma^\omega$,
	where $\sigma$ is an irreducible Pisot unimodular substitution.
	In other words, if $\Omega_{\sigma^\omega}$ is measurably conjugated to a translation on a torus, then the directive sequence $\sigma^\omega$ is good.
\end{corollary}

And we can restate the Pisot substitution conjecture as:

\begin{conjecture}[Reformulation of the Pisot substitution conjecture]
	For every irreducible Pisot unimodular substitution $\sigma$, the directive sequence $\sigma^\omega$ is good.
\end{conjecture}

And a generalization of the Pisot substitution conjecture could be:

\begin{conjecture}[Generalization of the Pisot substitution conjecture]
	Let $S$ be a set of unimodular substitutions. Let $s \in S^\NN$ be a directive sequence such that there exists a totally irrational direction $x \in \P\RR_+^d$
	and a constant $C > 0$ such that for every $k$ and $n \in \NN$, $\norm{\pi_x M_{[k,k+n)}}_1 \leq C e^{-C/n}$.
	Then, $s$ is good.
\end{conjecture}

The conjecture could be even more general:

\begin{conjecture}[Generalization of the Pisot substitution conjecture]
	Let $S$ be a set of unimodular substitutions. Let $s \in S^\NN$ be a directive sequence such that there exists a totally irrational direction $x \in \P\RR_+^d$ such that
	$\sum_n \norm{\pi_x M_{[k,k+n)}}_1$ converges uniformly in $k$.
	Then the subshift associated to $s$ is measurably conjugated to a translation on a torus.
\end{conjecture}

\section{Thanks}
The authors would like to thank Mélodie Andrieu, Nicolas Bédaride, Jean-François Bertazzon, Julien Cassaigne, Paul Mercat, and Thierry Monteil for their help in preparing this article.

\printnomenclature

\bibliographystyle{abbrv}
\bibliography{biblio-pytheas-tore}

\begin{thebibliography}{10}

\bibitem{Aki.Merc.18}
S.~Akiyama and P.~Mercat.
\newblock Yet another characterization of the {Pisot} substitution conjecture,
  2018.
\newblock Electronic preprint arXiv:1810.03500.

\bibitem{Andrieu2020}
M.~Andrieu.
\newblock PhD thesis, Aix-Marseille Université, 2020.
\newblock In preparation.

\bibitem{Arn.Labbe.15}
P.~Arnoux and S.~Labb\'{e}.
\newblock On some symmetric multidimensional continued fraction algorithms.
\newblock {\em Ergodic Theory Dynam. Systems}, 38(5):1601--1626, 2018.

\bibitem{Arn.Nog.93}
P.~Arnoux and A.~Nogueira.
\newblock Mesures de {G}auss pour des algorithmes de fractions continues
  multidimensionnelles.
\newblock {\em Ann. Sci. \'{E}cole Norm. Sup. (4)}, 26(6):645--664, 1993.

\bibitem{Arn.Rau.91}
P.~Arnoux and G.~Rauzy.
\newblock Repr\'{e}sentation g\'{e}om\'{e}trique de suites de complexit\'{e}
  {$2n+1$}.
\newblock {\em Bull. Soc. Math. France}, 119(2):199--215, 1991.

\bibitem{Avil.Dele.19}
A.~Avila and V.~Delecroix.
\newblock Some monoids of {P}isot matrices.
\newblock In {\em New trends in one-dimensional dynamics}, volume 285 of {\em
  Springer Proc. Math. Stat.}, pages 21--30. Springer, Cham, 2019.

\bibitem{Avi.Hub.Skrip.Inv}
A.~Avila, P.~Hubert, and A.~Skripchenko.
\newblock Diffusion for chaotic plane sections of 3-periodic surfaces.
\newblock {\em Invent. Math.}, 206(1):109--146, 2016.

\bibitem{Avi.Hub.Skrip.16}
A.~Avila, P.~Hubert, and A.~Skripchenko.
\newblock On the {H}ausdorff dimension of the {R}auzy gasket.
\newblock {\em Bull. Soc. Math. France}, 144(3):539--568, 2016.

\bibitem{Bertazzon2012}
J.-F. Bertazzon.
\newblock Fonction complexit{\'e} associ{\'e}e {\`a} une application ergodique
  du tore.
\newblock {\em Bull. London Math. Soc.}, 44(6):1155--1168, 2012.

\bibitem{Bert.Delec.14}
V.~Berth{\'e} and V.~Delecroix.
\newblock Beyond substitutive dynamical systems: {$S$}-adic expansions.
\newblock In {\em Numeration and substitution 2012}, RIMS K\^oky\^uroku
  Bessatsu, B46, pages 81--123. Res. Inst. Math. Sci. (RIMS), Kyoto, 2014.

\bibitem{Berth-Steiner-Thusw.14}
V.~Berth\'{e}, W.~Steiner, and J.~M. Thuswaldner.
\newblock Geometry, dynamics, and arithmetic of {$S$}-adic shifts.
\newblock {\em Ann. Inst. Fourier (Grenoble)}, 69(3):1347--1409, 2019.

\bibitem{Beda.Bert.13}
N.~Bédaride and J.-F. Bertazzon.
\newblock Minoration of the complexity function associated to a translation on
  the torus.
\newblock {\em Monatsh. Math.}, 171(3-4):291--304, 2013.

\bibitem{CassaigneFerencziMessaoudi2008}
J.~Cassaigne, S.~Ferenczi, and A.~Messaoudi.
\newblock Weak mixing and eigenvalues for {Arnoux-Rauzy} sequences.
\newblock {\em Ann. Inst. Fourier (Grenoble)}, 58(6):1983--2005, 2008.

\bibitem{Cass.Lab.17}
J.~Cassaigne, S.~Labb{\'e}, and J.~Leroy.
\newblock A set of sequences of complexity {$2n+1$}.
\newblock In {\em Combinatorics on words}, volume 10432 of {\em Lecture Notes
  in Comput. Sci.}, pages 144--156. Springer, Cham, 2017.

\bibitem{Chekhova.Hubert.Messaoudi.01}
N.~Chekhova, P.~Hubert, and A.~Messaoudi.
\newblock Propri\'{e}t\'{e}s combinatoires, ergodiques et arithm\'{e}tiques de
  la substitution de {T}ribonacci.
\newblock {\em J. Th\'{e}or. Nombres Bordeaux}, 13(2):371--394, 2001.

\bibitem{Chev.09}
N.~Chevallier.
\newblock Coding of a translation of the two-dimensional torus.
\newblock {\em Monatsh. Math.}, 157(2):101--130, 2009.

\bibitem{Did.98}
G.~Didier.
\newblock Combinatoire des codages de rotations.
\newblock {\em Acta Arith.}, 85(2):157--177, 1998.

\bibitem{Dumont-Thomas}
J.-M. Dumont and A.~Thomas.
\newblock Syst{\`e}mes de num{\'e}ration et fonctions fractales relatifs aux
  substitutions.
\newblock {\em Theoret. Comput. Sci.}, 65:153--169, 1989.

\bibitem{Durand.10}
F.~Durand.
\newblock Combinatorics on {B}ratteli diagrams and dynamical systems.
\newblock In {\em Combinatorics, automata and number theory}, volume 135 of
  {\em Encyclopedia Math. Appl.}, pages 324--372. Cambridge Univ. Press,
  Cambridge, 2010.

\bibitem{Dur.Ler.Rich.13}
F.~Durand, J.~Leroy, and G.~Richomme.
\newblock Do the properties of an {$S$}-adic representation determine factor
  complexity?
\newblock {\em J. Integer Seq.}, 16(2):Article 13.2.6, 30, 2013.

\bibitem{Ferenczi1992}
S.~{Ferenczi}.
\newblock {Bounded remainder sets.}
\newblock {\em {Acta Arith.}}, 61(4):319--326, 1992.

\bibitem{Fer.96}
S.~Ferenczi.
\newblock Rank and symbolic complexity.
\newblock {\em Ergodic Theory Dynam. Systems}, 16(4):663--682, 1996.

\bibitem{Freudenthal1942}
H.~Freudenthal.
\newblock Simplizialzerlegungen von beschränkter flachheit.
\newblock {\em Annals of Mathematics}, 43(3):580--582, 1942.

\bibitem{Furst-Kest}
H.~Furstenberg and H.~Kesten.
\newblock Products of random matrices.
\newblock {\em Ann. Math. Statist.}, 31:457--469, 1960.

\bibitem{Katok.Stepin.67}
A.~B. Katok and A.~M. Stepin.
\newblock Approximations in ergodic theory.
\newblock {\em Uspehi Math. Nauk}, 22(5):81--106, 1967.

\bibitem{Labbe.15}
S.~Labb{\'e}.
\newblock 3-dimensional continued fraction algorithms cheat sheets, 2015.
\newblock Electronic preprint arXiv:1511.078399.

\bibitem{Lag.93}
J.~C. Lagarias.
\newblock The quality of the {D}iophantine approximations found by the
  {J}acobi-{P}erron algorithm and related algorithms.
\newblock {\em Monatsh. Math.}, 115(4):299--328, 1993.

\bibitem{LeeSantos2017}
C.~W. Lee and F.~Santos.
\newblock Subdivisions and triangulations of polytopes.
\newblock In C.~D. T\'oth, J.~O'Rourke, and J.~E. Goodman, editors, {\em
  Handbook of Discrete and Computational Geometry}, pages 415--447. CRC Press,
  3rd edition, 2017.

\bibitem{MorseHedlund1940}
M.~Morse and G.~A. Hedlund.
\newblock Symbolic dynamics ii. sturmian trajectories.
\newblock {\em American Journal of Mathematics}, 62(1):1--42, 1940.

\bibitem{Nakai.06}
K.~Nakaishi.
\newblock Strong convergence of additive multidimensional continued fraction
  algorithms.
\newblock {\em Acta Arith.}, 121(1):1--19, 2006.

\bibitem{Oseledets}
V.~I. Oseledets.
\newblock A multiplicative ergodic theorem: {Lyapunov} characteristic numbers
  for dynamical systems.
\newblock {\em Trans. Moscow Math. Soc.}, 19:197--231, 1968.

\bibitem{Pyth.02}
N.~Pytheas~Fogg.
\newblock {\em Substitutions in dynamics, arithmetics and combinatorics},
  volume 1794 of {\em Lecture Notes in Mathematics}.
\newblock Springer-Verlag, Berlin, 2002.
\newblock Edited by V. Berth{\'e}, S. Ferenczi, C. Mauduit and A. Siegel.

\bibitem{Rauzy1979}
G.~Rauzy.
\newblock Échanges d'intervalles et transformations induites.
\newblock {\em Acta Arithmetica}, 34(4):315--328, 1979.

\bibitem{Rau.82}
G.~Rauzy.
\newblock Nombres alg\'ebriques et substitutions.
\newblock {\em Bull. Soc. Math. France}, 110(2):147--178, 1982.

\bibitem{Schwei.00}
F.~Schweiger.
\newblock {\em Multidimensional continued fractions}.
\newblock Oxford Science Publications. Oxford University Press, Oxford, 2000.

\bibitem{Tab.95}
S.~Tabachnikov.
\newblock {\em Billiards}.
\newblock Number~1 in Panoramas et Synth\`eses. Soci\'et\'e Math\'ematique de
  France, 1995.

\bibitem{Yoccoz2010}
J.-C. {Yoccoz}.
\newblock {Interval exchange maps and translation surfaces.}
\newblock In {\em {Homogeneous flows, moduli spaces and arithmetic. Proceedings
  of the Clay Mathematics Institute summer school, Centro di Recerca
  Mathematica Ennio De Giorgi, Pisa, Italy, June 11--July 6, 2007}}, pages
  1--69. Providence, RI: American Mathematical Society (AMS); Cambridge, MA:
  Clay Mathematics Institute, 2010.

\end{thebibliography}

\end{document}